\documentclass[11pt]{article}
\usepackage{amsfonts}
\usepackage{mathrsfs}
\usepackage{amssymb,amsmath,mathrsfs, amsthm}
\usepackage{palatino,cite}
\usepackage{graphicx}
\usepackage{mathtools}
\usepackage{hyperref,cases,anysize,enumerate}

\numberwithin{equation}{section}

\newtheorem{theorem}{Theorem}[section]

\newtheorem{proposition}{Proposition}[section]
\theoremstyle{definition}
\newtheorem{definition}{Definition}[section]
\newtheorem{remark}{Remark}[section]
\newtheorem{example}{Example}[section]
\theoremstyle{remark}

\date{}

\begin{document}
	
	\title{Geometry of holomorphic invariant strongly pseudoconvex complex Finsler metrics on the classical domains}
	\author{Xiaoshu Ge\; (gexiaoshu@stu.xmu.edu.cn)\\
		School of Mathematical Sciences, Xiamen
		University\\ Xiamen 361005, China
		\and
		Chunping Zhong\footnote{Corresponding author}\; (zcp@xmu.edu.cn)
		\\
		School of Mathematical Sciences, Xiamen
		University\\ Xiamen 361005, China
	}
	
	\date{}
	\maketitle
	\begin{abstract}
		In this paper, a class of holomorphic invariant metrics is introduced on the irreducible classical domains of type $I$-$IV$,
		which are strongly pseudoconvex complex Finsler metrics in the strict sense of M. Abate and G. Patrizio. These metrics are of particular interest
		in several complex variables since they are holomorphic invariant complex Finsler metrics found so far in literature which  enjoy good regularity as well as
		strong pseudoconvexity and can be explicitly expressed so as to admit differential geometric studies. They are, however, not necessarily Hermitian quadratic
		as that of the Bergman metrics. These metrics are explicitly constructed via deformation of the corresponding Bergman metric on the irreducible classical domains
		of type $I-IV$, respectively, and they are all proved to be complete K\"ahler-Berwald metrics. They enjoy very similar curvature properties as that
		of the Bergman metric on the irreducible classical domains, namely their holomorphic sectional curvatures are bounded between two negative constants, 
		and their holomorphic bisectional curvatures are always non
		positive and bounded below by negative constants, respectively. From the viewpoint of complex analysis,
		these metrics are analogues of Bergman metrics in complex Finsler geometry which do not necessarily have Hermitian quadratic restrictions in the viewpoint
		of S.-S. Chern.
	\end{abstract}
	
	\section{Introduction and main results}\label{section-1}
	
	According to S.-S. Chern \cite{Ch}, Finsler geometry is just Riemannian geometry without quadratic restrictions.
	In real Finsler geometry, there are abundant examples of non-quadratic Finsler metrics with various kinds of curvature properties.
	As far as complex Finsler geometry is concerned, there are also lots of strongly pseudoconvex complex Finsler metrics in the strict sense of M. Abate and G. Patrizio \cite{AP}. Indeed, it makes sense to develop  Hermitian geometry without Hermitian quadratic restriction and investigate various topics in complex Finsler geometry at least in the sense of smooth category (cf. \cite{Zh0,Zh1}, \cite{XZ0, XZ1}, \cite{Zh2}).
	
	In function theory of one complex variable, the Riemann mapping theorem implies that the open unit disc is the representative domain in $\mathbb{C}$. In $\mathbb{C}^n(n\geq 2)$, however, even the unit ball $B_n$ and the unit polydisc $P_n$ are not biholomorphically equivalent. Hence the Riemann mapping theorem does not hold in higher dimensional complex space. Thus from the viewpoint of complex analysis, it is more desirable to use holomorphic invariant metrics in dealing with related topics in geometric function theory of
	several complex variables (cf. \cite{Look1,Look2},\cite{Kobayashi0,Kobayashi1}, \cite{Dineen}, \cite{Look3}).
	On bounded domains in $\mathbb{C}^n$, the natural candidates of such metrics are the Bergman metrics, which are Hermitian quadratic (cf. \cite{Look1,Look2,Look3}).
	On the other hand, most of the intrinsic metrics arising in several complex variables are complex Finsler metrics, for instance the Carath\'eodory metric, the Kobayashi metric, the Sibony metric, the Azukawa metric, etc. In general, however, these metrics do not have enough regularities to admit differential geometric studies.  We refer the readers to S. Kobayashi \cite{Kobayashi0, Kobayashi1}, S. Dineen \cite{Dineen} and the references therein for more details.
	On a convex domain  $D\subset \mathbb{C}^n$, a fundamental theorem of Lempert \cite{Le} states
	that the Carath\'eodory metric and Kobayashi metric on $D$ coincide and they have constant holomorphic curvatures $-4$; if furthermore $D$ is bounded and strongly convex, then the Carath\'eodory metric and the Kobayashi metric are strongly pseudoconvex weakly K\"ahler-Finsler metrics with constant holomorphic sectional curvature $-4$ (cf. \cite{AP}). Unless in some very special cases, however, even on convex domains the Carath\'eodory and Kobayashi metrics cannot be explicitly expressed. As pointed out by M. Abate, T. Aikou and G. Patrizio in \cite{AAP}, "the lack of consideration of explicit examples made the choice of the 'right' notions in the complex setting difficult and sometimes rather artificial......, the lack of examples raises the doubt that perhaps metrics satisfying such strong conditions occur very infrequently."
	Therefore from the viewpoint of several complex variables and complex Finsler geometry, it is very natural to ask the following question.

	\textbf{Question 1}\quad  Is there a complex manifold $M$, which admits  an $\mbox{Aut}(M)$-invariant strongly pseudoconvex complex Finsler metric $F:T^{1,0}M\rightarrow [0,+\infty)$ which is not necessary Hermitian quadratic as that of the Bergman metric? Here $\mbox{Aut}(M)$ denotes the holomorphic automorphism of $M$ which consists of biholomorphic mappings from $M$ onto itself. If it exists, one can further ask whether  $F$ is a K\"ahler-Finsler metric, or a K\"ahler-Berwald metric, or a K\"ahler metric?
	
	In \cite{Zh2}, the second author of the present paper proved that on the unit ball $B_n$ in $\mathbb{C}^n$, there does not exist any $\mbox{Aut}(B_n)$-invariant strongly pseudoconvex complex Finsler metric other than a
	constant multiple of the Bergman metric, while on the unit polydisc $P_n$ in $\mathbb{C}^n(n\geq 2)$, there are infinite many $\mbox{Aut}(P_n)$-invariant strongly pseudoconvex complex Finsler metrics other than a
	constant multiple of the Bergman metric. Note that the unit ball $B_n$ is an irreducible bounded symmetric domain of rank $1$, while the unit polydisc $P_n$ is a reducible bounded symmetric domain of rank $n$. On the other hand, by Theorem 7.1 in S. Helgason \cite{Hel}, for every Hermitian symmetric space $M$ of the noncompact type, there exists a bounded symmetric domain $D$ such that $M$ is biholomorphically equivalent to $D$. Note that  irreducible Hermitian symmetric manifolds of non-compact type (resp. of compact type) have been classified by \'E. Cartan \cite{Cartan} and they are divided in six types: $I, II, III, IV, V$ and $VI$,  which can be realized as irreducible bounded symmetric domains via Harish-Chandra embedding. The first four types are called the classic domains and the last two are called the exceptional domains in L. K. Hua \cite{Hua} .

	Denote $\mathfrak{R}_A$ any one of the irreducible classical domains of type $A=I, II, III, IV$.  Denote $C_A$ and $K_A$ the Carath\'eodory metric and Kobayashi metric on $\mathfrak{R}_A$, respectively. Then for $A=I,II,III$ (cf. Theorem 2 in \cite{Suzuki})
	$$
	C_A(0;\xi)=K_A(0;\xi)=\max\left\{\mbox{positive square roots of the eigenvalues of}\; \xi\overline{\xi'}\right\},\quad \xi\in T_0^{1,0}\mathfrak{R}_A
	$$
	and
	\begin{equation}
		C_{IV}(0;\xi)=K_{IV}(0;\xi)=\sqrt{\|\xi\|^2+\sqrt{\|\xi\|^4-|\xi\xi'|^2}},\quad \xi\in T_0^{1,0}\mathfrak{R}_{IV}.\label{s11}
	\end{equation}
	That is,  the  Carath\'eodory metric and  Kobayashi metric on $\mathfrak{R}_A$  are not smooth with respect to nonzero tangent directions $\xi$, so that they are not $\mbox{Aut}(\mathfrak{R}_A)$-invariant strongly pseudoconvex complex Finsler metrics in the strict sense of  \cite{AP}.

	Since the classic domains $\mathfrak{R}_A (A=I, II, III,IV)$ play a very fundamental role in several complex variables and complex algebraic geometry, as well as complex differential geometry, it is very natural to ask the following question.
	
	\textbf{Question 2}\quad Let $\mathfrak{R}_A$ be any one of the irreducible classical domains of type $A=I,II,III,IV$ such that $\mbox{rank}(\mathfrak{R}_A)\geq 2$. Other than a constant multiple of the corresponding Bergman metric on $\mathfrak{R}_A$, are there any other $\mbox{Aut}(\mathfrak{R}_A)$-invariant strongly pseudoconvex complex Finsler metrics  on $\mathfrak{R}_A$? If there are such metrics, whether they are K\"ahler-Finsler metrics or K\"ahler-Berwald metrics?
	
	\begin{remark}
		Since $\mathfrak{R}_I(m,n)$ reduces to the unit ball $B_n$ in $\mathbb{C}^n$ whenever $m=1$ and $B_n$ admits no $\mbox{Aut}(B_n)$-invariant strongly pseudoconvex complex Finsler metrics other than a constant multiple of the Bergman metric on $B_n$, it is natural to assume that $\mbox{rank}(\mathfrak{R}_A)\geq 2$.
	\end{remark}
	
	In this paper, we give affirmative answers to \textbf{Question 1} and \textbf{Question 2} by  explicit construction of such metrics on the irreducible classical domains $\mathfrak{R}_A$ of type $A=I, II, III, IV$, respectively.
	
	To formulate our results, let's denote $\mathscr{M}(m,n)$ the complex vector space which consists of $m\times n$ matrices over the complex number field $\mathbb{C}$, and let $\mathfrak{D}$ be a domain in $\mathscr{M}(m,n)$. We also denote $V\in T_Z^{1,0}\mathfrak{D}$ to mean a complex tangent vector of type $(1,0)$ at the point $Z\in\mathfrak{D}$. Denote $\mathbb{N}$ the set of positive integers. For $i,j,k\in\mathbb{N}$, we set
	\begin{eqnarray}
		\mathfrak{B}_k(Z;V)&:=&\mbox{tr}\left\{\left[(I-Z\overline{Z'})^{-1}V(I-\overline{Z'}Z)^{-1}\overline{V'}\right]^k\right\},\label{s12}\\
		\mathcal{B}_{i,j}(Z;V,W)&:=&\mbox{tr}\left\{[(I-Z\overline{Z'})^{-1}V(I-\overline{Z'}Z)^{-1}\overline{V'}]^i[(I-Z\overline{Z'})^{-1}W(I-\overline{Z'}Z)^{-1}\overline{W'}]^j\right\},\label{s13}
	\end{eqnarray}
	where $Z\in\mathfrak{D}$ and $V,W\in T_Z^{1,0}\mathfrak{D}\cong \mathscr{M}(m,n)$.
	
	The main results in this paper are outlined as follows.
	
	\begin{theorem}\label{th11}
		Let $\mathfrak{R}_I(m;n),\;m\leq n$, be a classical domain of type $I$. For any fixed $t\in[0,+\infty),k\in\mathbb{N}$ and $k\geq 2$, define
		\begin{eqnarray*}
			F_I^2(Z;V)&=&\frac{m+n}{1+t}\left\{\mathfrak{B}_1(Z;V)
			+t\sqrt[k]{\mathfrak{B}_k(Z;V)}\right\}
		\end{eqnarray*}
		for $Z\in \mathfrak{R}_I$ and $V\in T_Z^{1,0}\mathfrak{R}_I$. Then the following assertions are true:
		
		(1) $F_I$ is a complete $\mbox{Aut}(\mathfrak{R}_I)$-invariant strongly pseudoconvex K\"ahler-Berwald metric;
		
		(2) For $Z\in\mathfrak{R}_I$ and nonzero vectors $V,W\in T_Z^{1,0}\mathfrak{R}_I$,
		the holomorphic sectional curvature $K_I$ and the holomorphic bisectional
		curvature $B_I$ of $F_I$ are given respectively by
		\begin{eqnarray*}
			K_I(Z;V)&=&-\frac{4(m+n)}{1+t}\frac{\mathfrak{B}_2(Z;V)
				+t\mathfrak{B}_k^{\frac{1}{k}-1}(Z;V)\mathfrak{B}_{k+1}(Z;V)}{F_{I}^4(Z;V)},\\
			B_{I}(Z; V, W)
			&=&-\frac{2(m+n)}{(1+t)F_{I}^2(Z;V)F_I^2(Z;W)}
			\Big\{\mathcal{B}_{1,1}(Z;V,W)+\mathcal{B}_{1,1}(\overline{Z'};\overline{V'},\overline{W'})\\
			&+&t{\mathfrak{B}_{k}^{\frac{1}{k}-1}(Z;V)}
			\left[\mathcal{B}_{k,1}(Z;V,W)+\mathcal{B}_{k,1}(\overline{Z'};\overline{V'},\overline{W'})\right]\Big\};
		\end{eqnarray*}
		
		(3) The holomorphic sectional curvature of $F_{I}$ is always negative, more precisely, it is bounded from below by $ -\frac{4}{m+n}, $ and bounded from above by $-\frac{4}{m+n}\cdot \frac{1+t}{m+t\sqrt[k]{m}}$;
		
		(4) The holomorphic bisectional curvature of $F_{I}$ is nonpositive, more precisely, it is bounded from below by $ -\frac{4}{m+n}$ and bounded from above by $0$.
	\end{theorem}
	
	For the proof of Theorem \ref{th11}, we refer to Theorem \ref{KB-FI}, Theorem \ref{C-FI}, Theorem \ref{thm-4.4} and Theorem \ref{thm4}.

	\begin{theorem}\label{th12}
		Let $\mathfrak{R}_{II}(p)$ be a classical domain of type $II$. For any fixed $t\in[0,+\infty),k\in\mathbb{N}$ and $k\geq 2$, define
		\begin{eqnarray*}
			F_{II}^2(Z;V)&=&\frac{p+1}{1+t}\left\{\mathfrak{B}_1(Z;V)
			+t\sqrt[k]{\mathfrak{B}_k(Z;V)}\right\}
		\end{eqnarray*}
		for $Z\in \mathfrak{R}_{II}$ and $V\in T_Z^{1,0}\mathfrak{R}_{II}$. Then the following assertions are true:
		
		(1) $F_{II}$ is a complete $\mbox{Aut}(\mathfrak{R}_{II})$-invariant strongly pseudoconvex K\"ahler-Berwald metric;
		
		(2) For $Z\in\mathfrak{R}_{II}(p)$ and nonzero vectors $V,W\in T_Z^{1,0}\mathfrak{R}_{II}$, the holomorphic sectional curvature $K_{II}$ and the holomorphic bisectional
		curvature $B_{II}$ of $F_{II}$ are given respectively by
		\begin{eqnarray*}
			K_{II}(Z;V)&=&-\frac{4(p+1)}{1+t}\frac{\mathfrak{B}_2(Z;V)
				+t\mathfrak{B}_k^{\frac{1}{k}-1}(Z;V)\mathfrak{B}_{k+1}(Z;V)}{F_{II}^4(Z;V)},\\
			B_{II}(Z; V, W)
			&=&-\frac{4(p+1)}{1+t}\frac{\mathcal{B}_{1,1}(Z;V,W)+t{\mathfrak{B}_{k}^{\frac{1}{k}-1}(Z;V)}
				\mathcal{B}_{k,1}(Z;V,W)}{F_{II}^2(Z;V)F_{II}^2(Z;W)};
		\end{eqnarray*}
		
		(3) The holomorphic sectional curvature of $F_{II}$ is always negative, more precisely, it is bounded from below by $ -\frac{4}{1+p}$, and bounded from above by $-\frac{4}{1+p}\cdot \frac{1+t}{p+t\sqrt[k]{p}}$;	
		
		(4) The holomorphic bisectional curvature of $F_{II}$ is nonpositive, more precisely, it is bounded from below by $-\frac{4}{1+p}$, and bounded from  above by $0$.
	\end{theorem}
	
	For the proof of Theorem \ref{th12}, we refer to Theorem \ref{KB-FII}, Theorem \ref{C-FII},
	Theorem \ref{thm-4.4} and Theorem \ref{thm4}.

	\begin{theorem}\label{th13}
		Let $ \mathfrak{R}_{III}(q)$ be a classical domain of type $III$. For any fixed $t\in[0,+\infty),k\in\mathbb{N}$ and $k\geq 2$, define
		\begin{eqnarray*}
			F_{III}^2(Z;V)&=&\frac{q-1}{1+t}\left\{\mathfrak{B}_1(Z;V)
			+t\sqrt[k]{\mathfrak{B}_k(Z;V)}\right\}
		\end{eqnarray*}
		for $Z\in \mathfrak{R}_{III}$ and $V\in T_Z^{1,0}\mathfrak{R}_{III}$. Then the following assertions are true:
		
		(1) $F_{III}$ is a complete $\mbox{Aut}(\mathfrak{R}_{III})$-invariant strongly pseudoconvex K\"ahler-Berwald metric;
		
		(2) For $Z\in\mathfrak{R}_{III}(q)$ and nonzero vectors $V,W\in T_Z^{1,0}\mathfrak{R}_{III}$,
		the holomorphic sectional curvature $K_{III}$ and the holomorphic bisectional
		curvature $B_{III}$ of $F_{III}$ are given respectively by
		\begin{eqnarray*}
			K_{III}(Z;V)&=&-\frac{4(q-1)}{1+t}\frac{\mathfrak{B}_2(Z;V)
				+t\mathfrak{B}_k^{\frac{1}{k}-1}(Z;V)\mathfrak{B}_{k+1}(Z;V)}{F_{III}^4(Z;V)},\\
			B_{III}(Z; V, W)
			&=&-\frac{4(q-1)}{1+t}\frac{\mathcal{B}_{1,1}(Z;V,W)+t{\mathfrak{B}_{k}^{\frac{1}{k}-1}(Z;V)}
				\mathcal{B}_{k,1}(Z;V,W)}{F_{III}^2(Z;V)F_{III}^2(Z;W)};
		\end{eqnarray*}
		
		(3) The holomorphic sectional curvature of $F_{III}$ is always negative, more precisely, it is bounded from below by $ -\frac{4}{q-1}\cdot\frac{1+t}{2+t\sqrt[k]{2}}$,
		and bounded from above by $-\frac{4}{q-1}\cdot\frac{1+t}{2[\frac{q}{2}]+t\sqrt[k]{2[\frac{q}{2}]}};$
		
		(4) The holomorphic bisectional curvature of $F_{III}$ is nonpositive, more precisely, it is bounded from below by $ -\frac{4}{q-1},$ and bounded from  above by $0$.
	\end{theorem}
	
	For the proof of Theorem \ref{th13}, we refer to Theorem \ref{KB-FIII}, Theorem \ref{C-FIII}, Theorem \ref{thm-4.4} and Theorem \ref{thm4}.
	
	\begin{remark}
		For sufficiently small $t>0$, $F_A$ is a strongly convex K\"ahler-Berwald metric on $\mathfrak{R}_A$ for $A=I,II$ and $III$, respectively.
	\end{remark}

	The construction of $\mbox{Aut}(\mathfrak{R}_{IV})$ -invariant strongly pseudoconvex complex Finsler metrics is a little different from that of $\mathfrak{R}_I,\mathfrak{R}_{II}$ and $\mathfrak{R}_{III}$.
	
	Define
	$$
	f_{IV}(\xi):=\sqrt{\pmb{r}\phi(\pmb{s})},\quad\pmb{r}:=\xi\overline{\xi'},\quad \pmb{s}:=\frac{|\xi\xi'|^2}{\pmb{r}^2}\quad \mbox{with}\quad  0\neq \xi=(\xi_1,\cdots,\xi_N)\in\mathbb{C}^N,
	$$
	where $\phi:[0,1]\rightarrow (0,+\infty)$ is a smooth function of $\pmb{s}\in[0,1]$ satisfying \eqref{sn} in Proposition \ref{prop}. Let $\phi_{z_0}\in \mbox{Aut}(\mathfrak{R}_{IV})$ be given by \eqref{p-IV} such that $\phi_{z_0}(z_0)=0$.

	\begin{theorem}\label{th14}
		Let $ \mathfrak{R}_{IV}(N)$ be a classical domain of type $IV$. For any nonzero vector $v\in T_{z_0}^{1,0}\mathfrak{R}_{IV}$ at the point $z_0\in \mathfrak{R}_{IV}$, define
		$$F_{IV}(z_0;v):=f_{IV}\left(\sqrt{2N}(\phi_{z_0})_\ast(v)\right)=\sqrt{\widetilde{\pmb{r}}\phi(\widetilde{\pmb{s}})},$$
		where $\widetilde{\pmb{r}}$ and $\widetilde{\pmb{s}}$ are given by \eqref{tr} and \eqref{ts}, respectively.
		
		Then the following assertions are true:
		
		(1) $F_{IV}(z;v)$ is a complete $\mbox{Aut}(\mathfrak{R}_{IV})$-invariant strongly pseudoconvex K\"ahler-Berwald metric;
		
		(2) For any nonzero vectors $v,w\in T_0^{1,0}\mathfrak{R}_{IV}$,
		the holomorphic sectional curvature $K_{IV}$ and the holomorphic bisectional
		curvature $B_{IV}$ of $F_{IV}$ are given respectively by
		\begin{eqnarray*}
			K_{IV}(0; v)&=&-\frac{2}{N\phi^2(\widetilde{\pmb{s}})}\left\{\phi(\widetilde{\pmb{s}})+(1-\widetilde{\pmb{s}})\left[\phi(\widetilde{\pmb{s}})-2\widetilde{\pmb{s}}\phi'(\widetilde{\pmb{s}})\right]\right\},\\
			B_{IV}(0; v, w)		&=&-\frac{2}{N}\cdot\frac{\phi\left[\widetilde{\pmb{r}}(0;v)\widetilde{\pmb{r}}(0;w)-|vw'|^2\right]+2\widetilde{\pmb{s}}\phi'|vw'|^2+(\phi-2\widetilde{\pmb{s}}\phi')|v\overline{w'}|^2}{F_{IV}^2(0;v)F_{IV}^2(0;w)};
		\end{eqnarray*}
		
		(3) The holomorphic sectional curvature of $F_{IV}$ is always negative, more precisely, it is  bounded from below by a negative constant $-A $ and bounded from above by a negative constant $-B$, where
		\begin{equation*}
			A=\max\limits_{\widetilde{\pmb{s}}\in[0,1]}\left\{\frac{2\left\{\phi(\widetilde{\pmb{s}})+(1-\widetilde{\pmb{s}})\left[\phi(\widetilde{\pmb{s}})-2\widetilde{\pmb{s}}\phi'(\widetilde{\pmb{s}})\right]\right\}}{N\phi^2(\widetilde{\pmb{s}})}\right\}
			\label{b-A}
		\end{equation*}
		and
		\begin{equation*} B=\min\limits_{\widetilde{\pmb{s}}\in[0,1]}\left\{\frac{2\left\{\phi(\widetilde{\pmb{s}})+(1-\widetilde{\pmb{s}})\left[\phi(\widetilde{\pmb{s}})-2\widetilde{\pmb{s}}\phi'(\widetilde{\pmb{s}})\right]\right\}}{N\phi^2(\widetilde{\pmb{s}})}\right\};
			\label{b-B}
		\end{equation*}
		
		(4) If $\phi'\geq0$,  then the holomorphic bisectional curvature of  $F_{IV}$ is nonpositive and it is bounded from below by a negative constant $-C$ with
		\begin{equation*}
			C=\max\limits_{\widetilde{\pmb{s}}\in[0,1]}\left\{\frac{2}{N}\cdot\frac{\phi\left[\widetilde{\pmb{r}}(0;v)\widetilde{\pmb{r}}(0;w)-|vw'|^2\right]+2\widetilde{\pmb{s}}\phi'|vw'|^2+(\phi-2\widetilde{\pmb{s}}\phi')|v\overline{w'}|^2}{F_{IV}^2(0;v)F_{IV}^2(0;w)}\right\}.\label{b-C}
		\end{equation*}
		
	\end{theorem}
	
	\begin{remark} First, we note that if $\mathfrak{R}_A$ reduces to the open unit disc in $\mathbb{C}$, then $F_A$ reduces to the well-known Poincar\'e metric on the open unit disc for $A=I,II,III,IV$ respectively.
		Second, for $t=0$, $F_I, F_{II}$ and $F_{III}$ reduce to the Bergman metrics on the classical domains of type $I,II$ and $III$ respectively,
		which are complete K\"ahler metrics, hence Hermitian quadratic with respect to fiber coordinates. For any fixed $t\in(0,+\infty), k\in\mathbb{N}$ and $k\geq 2$,
		however, $F_I, F_{II}$ and  $F_{III}$ are holomorphic invariant complete K\"ahler-Berwald metrics which are not Hermitian quadratic with respect to fiber coordinates.
		For $\phi(\widetilde{\pmb{s}})\equiv 1$,
		$F_{IV}$ reduces to the Bergman metric on $\mathfrak{R}_{IV}$, which is also a complete K\"ahler metric, hence a Hermitian quadratic metric. There are infinite many $\phi(\widetilde{\pmb{s}})$ such that $F_{IV}=\sqrt{\widetilde{\pmb{r}}\phi(\widetilde{\pmb{s}})}$ are $\mbox{Aut}(\mathfrak{R}_{IV})$-invariant complete K\"ahler-Berwald metrics and non-Hermitian quadratic. If we take $\phi(\widetilde{\pmb{s}})=1+\sqrt{1+\widetilde{\pmb{s}}}$, then the corresponding metric $F_{IV}=\sqrt{\widetilde{\pmb{r}}\phi(\widetilde{\pmb{s}})}$ satisfies assertion (4) in Theorem \ref{th14}.
	\end{remark}

	\begin{remark}
		Let $\mathfrak{R}_A$ be any one of the irreducible classical domains of type $A=I, II, III, IV$.
		Then any $\mbox{Aut}(\mathfrak{R}_A)$-invariant Hermitian metric is a constant multiple of the corresponding Bergman metric on
		$\mathfrak{R}_A$ for $A=I, II, III, IV$, respectively. This is an immediate consequence of Theorem \ref{th11}, Theorem \ref{th12}, Theorem \ref{th13} and Theorem \ref{th14}.
		Thus from the viewpoint of several complex variables, the differential geometry of holomorphic invariant Hermitian quadratic metrics is unique on $\mathfrak{R}_A$, that is the differential geometry of the corresponding  Bergman metric on $\mathfrak{R}_A$ (\cite{Look1,Look2,Look3}). In the $\mbox{Aut}(\mathfrak{R}_A)$-invariant non-Hermitian quadratic case,  although there are infinitely many holomorphic invariant
		non-Hermitian quadratic metrics on $\mathfrak{R}_A$, Theorem \ref{th11}-Theorem \ref{th12} and Theorem \ref{th14} tell us  that these metrics enjoy very similar curvature properties as that of the corresponding Bergman metric on $\mathfrak{R}_A$. Both of them contribute to the geometry of holomorphic invariant strongly pseudoconvex complex Finsler metrics in the viewpoint of S.-S. Chern \cite{Ch}.
	\end{remark}

	\section{Complex Finsler metrics on domains in $\mathbb{C}^n$}\label{section-2}
	
	In this section, we recall some notions of complex Finsler geometry. We restrict our attention to domains $D\subset \mathbb{C}^n$ with $n\geq 2$ since in the case $n=1$ any strongly pseudoconvex complex Finsler metric is necessary a Hermitian metric. We refer to \cite{AP}, \cite{BCS},  \cite{Ch} for an excellent and comprehensive introduction to general manifold cases.
	
	Let $D$ be a domain in $\mathbb{C}^n$ with the canonical complex coordinates $z=(z_1,\cdots,z_n)$. A (complex) tangent vector $\pmb{v}$ at $z_0\in D$ is expressed as
	$$\pmb{v}=\sum_{i=1}^nv_i\frac{\partial}{\partial z_i}\Big|_{z_0},$$
	which can be identified with a point $(z_0;v)\in D\times \mathbb{C}^n$. So that the holomorphic tangent bundle
	$T^{1,0}D$ of $D$ is isomorphic to $D\times \mathbb{C}^n$, which is a trivial holomorphic vector bundle over $D$. Therefore we may use $(z;v)=(z_1,\cdots,z_n; v_1,\cdots,v_n)$ as global complex coordinates to denote points in $T^{1,0}D$.
	
	A (smooth) complex Finsler metric on $D$ is a continuous function $F:T^{1,0}D\rightarrow [0,+\infty)$ such that $G:=F^2$ is $C^\infty$ on
	$\widetilde{T^{1,0}D}:=T^{1,0}D\setminus\{\mbox{zero section}\}$ and satisfies
	\begin{equation*}
		F(z;\lambda v)=|\lambda| F(z;v)
	\end{equation*}
	for any $\lambda\in\mathbb{C}$ and $(z;v)\in T^{1,0}D$.
	
	If furthermore for each fixed point $z_0\in D$, the indicatrix $I_{z_0}:=\{v\in T_{z_0}^{1,0}D:F(z_0;v)<1\}$ is a strongly pseudoconvex domain in $T_{z_0}^{1,0}D\cong \mathbb{C}^n$, or equivalently
	\begin{equation}
		(G_{i\overline{j}}):=\left(\frac{\partial^2G}{\partial v_i\partial\overline{v_j}}\right)\label{Gab}
	\end{equation}
	is positive definite at any points $(z;v)\in \widetilde{T^{1,0}D}$, then we say that $F$ is a strongly pseudoconvex complex Finsler metric on $D$.
	
	Indeed, for each fixed point $z_0\in D$, a strongly pseudoconvex complex Finsler metric $F$ on $D$ induces  a strongly pseudoconvex complex norm $f_{z_0}(v):=F(z_0;v)$ on $T_{z_0}^{1,0}D\cong\mathbb{C}^n$  (also called complex Minkowski norm in \cite{AP},\cite{A2}). If $f_{z_0}(v)$ is the same  for any $z_0\in D$, then $F$ is called a complex Minkowski metric on $D$.
	
	\begin{remark}For a strongly pseudoconvex complex Finsler metric $F$, if the components $\frac{\partial^2F^2}{\partial v_i\partial\overline{v_j}}$ are independent of
		$v$ for any $i,j=1,\cdots,n$, namely $F^2$ is Hermitian quadratic with respect to $v$, then we say that $F$ is Hermitian quadratic, otherwise we say that $F$ is non-Hermitian quadratic. Indeed, a strongly pseudoconvex complex Finsler metric $F$ on $D$
		is $C^\infty$ Hermitian quadratic if and only
		if $F^2$ is $C^\infty$ on the whole $T^{1,0}D$ (\cite{AP}) and  $F$ is non-Hermitian quadratic if and only if $F^2$ is only $C^\infty$ on $\widetilde{T^{1,0}D}$.
	\end{remark}

	Denote $(G^{\overline{j}l})$ the inverse matrix of $(G_{i\overline{j}})$ such that
	$$
	\sum_{j=1}^nG_{i\overline{j}}G^{\overline{j}l}=\delta_{i}^l=\left\{
	\begin{array}{ll}
		1, & i=l \\
		0, & i\neq l.
	\end{array}
	\right.
	$$
	Then the Chern-Finsler nonlinear connection coefficients $\varGamma_{;i}^l$ associated to $F$ are given by (cf. \cite{AP})
	\begin{equation}
		\varGamma_{;i}^l=\sum_{s=1}^nG^{\overline{s}l}\frac{\partial^2G}{\partial z_i\partial\overline{v_s}}.\label{CN}
	\end{equation}
	
	Let $f\in \mbox{Aut}(D)$. We denote $w=f(z)=(w_1(z),\cdots,w_n(z))$ for $z\in D$ and
	\begin{equation}
		\left(\frac{\partial w_j}{\partial z_i}\right)=\begin{pmatrix}
			\frac{\partial w_1}{\partial z_1} & \cdots & \frac{\partial w_n}{\partial z_1} \\
			\vdots & \vdots & \vdots \\
			\frac{\partial w_1}{\partial z_n} & \cdots & \frac{\partial w_n}{\partial z_n} \\
		\end{pmatrix}.
	\end{equation}
	the Jacobian matrix of $f$ at $z$, which is non-degenerated and holomorphic with respect to $z\in D$. For $v\in T_z^{1,0}D$, denote $\xi:=f_\ast(v)\in T_{f(z)}^{1,0}D$, that is,
	$$\xi_j=\sum_{i=1}^n\frac{\partial w_j}{\partial z_i}v_i,\quad j=1,\cdots,n.$$
	Then
	$z=(z_1,\cdots,z_n)$ and $w=(w_1,\cdots,w_n)$ can be used as complex coordinates on $D$, and both $(z;v)=(z_1,\cdots,z_n; v_1,\cdots,v_n)$ and $(w;\xi)=(w_1,\cdots,w_n; \xi_1,\cdots,\xi_n)$
	can be used
	as complex coordinates on $T^{1,0}D$.
	By \eqref{Gab} and \eqref{CN}, it follows that the behavior of the Christoffel symbols $\varGamma_{;i}^l$ under change of complex coordinates $(z;v)\rightarrow (w;\xi)$ on $T^{1,0}D$ is
	\begin{equation}
		\tilde{\varGamma}_{;i}^l(w;\xi)=\sum_{a,b=1}^n\frac{\partial w_l}{\partial z_a}\varGamma_{;b}^a(z;v)\frac{\partial z_b}{\partial w_i}-\sum_{a,b=1}^n
		\frac{\partial^2w_l}{\partial z_a\partial z_b}\frac{\partial z_a}{\partial w_i} v_b,\label{Nwz}
	\end{equation}
	where we denote $\tilde{\varGamma}_{;i}^l(w;\xi)$ and $\varGamma_{;b}^a(z;v)$ the local expression of the Chern-Finsler nonlinear coefficients with respect to the complex coordinates $(w;\xi)$ and $(z;v)$, respectively.
	
	The horizontal Chern-Finsler connection coefficients $\varGamma_{j;i}^l$ associated to $F$ are given by
	\begin{equation}
		\varGamma_{j;i}^l=\frac{\partial}{\partial v_j}\varGamma_{;i}^l,\label{hccf}
	\end{equation}
	which satisfy
	\begin{equation}
		\sum_{j=1}^n\varGamma_{j;i}^l v_j=\varGamma_{;i}^l.\label{CCP}
	\end{equation}
	
	Using the homogeneity of $F$, it is easy to check that
	\begin{equation}
		\varGamma_{;i}^l(z;\lambda v)=\lambda\varGamma_{;i}^l(z;v),\quad \varGamma_{j;i}^l(z;\lambda v)=\varGamma_{j;i}^l(z;v),\quad \forall \lambda\in\mathbb{C}\setminus\{0\}.
	\end{equation}
	$F$ is called a K\"ahler-Finsler metric if the horizontal Chern-Finsler connection coefficients $\varGamma_{j;i}^l$ associated to $F$ are symmetric with respect to the lower indexes, i.e.,
	\begin{equation}
		\varGamma_{j;i}^l=\varGamma_{i;j}^l,
	\end{equation}
	which is equivalent to the condition $\varGamma_{;i}^l=\displaystyle\sum_{j=1}^n\varGamma_{i;j}^l v_j$ (\cite{CS}); $F$ is called a complex Berwald metric if
	$\varGamma_{j;i}^l$ are locally independent of $v$ (cf. \cite{A1},\cite{A2});  $F$ is called a K\"ahler-Berwald metric if $F$ is both a K\"ahler-Finsler metric and a complex Berwald metric.  It is easy to check that these definitions are independent of the choice of local complex coordinates, hence they make sense on any complex manifolds. Note that K\"ahler-Berwald manifolds are very natural objects studied in complex Finsler geometry (cf. \cite{Zh0}, \cite{XZ0}, \cite{XiaZhang},\cite{Zh2}).
	
	Complex Berwald manifolds $(D,F)$ play an important role in complex Finsler geometry because associated to a general strongly pseudoconvex complex Finsler metric $F$ we can only get a complex non-linear connection \cite{AP}, while associated to a complex Berwald metric there is always a complex linear connection $\nabla$ on $T^{1,0}D$. If we denote $\xi_t$ the parallel displacement of $\xi\in T_p^{1,0}D$ along a smooth curve $c(t)$ passing through $p$ with respect to $\nabla$, then the $F$-length $F(c(t);\xi_t)$ of $\xi_t$ is invariant under parallel displacement, so that parallel displacement gives a complex linear isomorphism between different tangent spaces \cite{A2}. Thus a complex Berwald manifold is also called a complex manifold modeled a complex Minkowski space \cite{A2}.

	Let $F$ be a strongly pseudoconvex complex Finsler metric on a domain $D$. Then the holomorphic (sectional) curvature $K_F$ of $F$ along $(z;v)\in\widetilde{T^{1,0}D}$ is defined by \cite{AP}
	\begin{equation}
		K(z;v)=\frac{2}{G^2(z;v)}\langle \Omega(\chi,\overline{\chi})\chi,\chi\rangle_{(z;v)},\label{hsca}
	\end{equation}
	where $\Omega$ is the curvature operator of the Chern-Finsler connection associated to $F$, and
	$$\chi=\sum_{i=1}^nv_i\left(\frac{\partial}{\partial z_i}-\sum_{l=1}^n\varGamma_{;i}^l(z;v)\frac{\partial}{\partial v_l}\right)$$
	is the complex horizontal radial vector field on $T^{1,0}D$.

	Let $w\in T_z^{1,0}D$. We denote
	$$\chi_v(w)=\sum_{i=1}^nw_i\left(\frac{\partial}{\partial z_i}-\sum_{l=1}^n\varGamma_{;i}^l(z;v)\frac{\partial}{\partial v_l}\right)$$
	the horizontal lift of $w$ along $v$.
	Note that $\chi=\chi_v(v)$. The holomorphic bisectional curvature of $F$ with respect to $v,w\in T_z^{1,0}D$ is defined as (cf. \cite{SZ1}, \cite{WZ})
	\begin{equation}
		B(z;v,w)=\frac{\left\langle \Omega\left(\chi_v(w),\overline{\chi_v(w)}\right)\chi,\chi\right\rangle}{G(z;v)G(z;w)}.\label{hbsca}
	\end{equation}
	\begin{remark}
		It is easy to check that \eqref{hsca} and \eqref{hbsca} are $\mbox{Aut}(D)$-invariant quantities, namely for any $\phi\in\mbox{Aut}(D)$, $z\in D$ and $0\neq v,w\in T_z^{1,0}D$,
		\begin{eqnarray*}
			K(\phi(z);\phi_\ast(v))&=&K(z;v),\\
			B(\phi(z);\phi_\ast(v),\phi_\ast(w))&=&B(z;v,w),
		\end{eqnarray*}
		where $\phi_\ast$ denotes the differential of $\phi$ at $z$.
	\end{remark}

	\subsection{Complex Finsler metrics on homogeneous domains in $\mathbb{C}^n$}\label{subsection-2.1}
	
	In this section, we assume that $D$ is a homogeneous domain in $\mathbb{C}^n$ which contains the origin, namely the holomorphic automorphism $\mbox{Aut}(D)$ acts
	transitively on $D$. We denote $\mbox{Iso}(D)$ the isotropy subgroup of $\mbox{Aut}(D)$ at the origin.
	\begin{theorem}\label{hfm}
		Let $D$ be a homogeneous domain in $\mathbb{C}^n$ which contains the origin. Then $F:T^{1,0}D\rightarrow [0,+\infty)$ is an $\mbox{Aut}(D)$-invariant strongly pseudoconvex complex Finsler metric
		if and only if there is a complex Minkowski  norm $f:\mathbb{C}^n\rightarrow [0,+\infty)$ which is $\mbox{Iso}(D)$-invariant such that for any fixed $z_0\in D$ and any $v\in T_{z_0}^{1,0}D$,
		\begin{equation}
			F(z_0;v):=f(\phi_{\ast}(v)),\label{fz0v}
		\end{equation}
		where $\phi\in \mbox{Aut}(D)$ satisfying $\phi(z_0)=0$ and $\phi_\ast$ denotes the differential of $\phi$ at $z_0$.
	\end{theorem}
	\begin{proof} Suppose that $F:T^{1,0}D\rightarrow [0,+\infty)$ is an $\mbox{Aut}(D)$-invariant strongly pseudoconvex complex Finsler metric. Define
		$$
		f(v):=F(0;v),\quad \forall v\in \mathbb{C}^n\cong T_{0}^{1,0}D.
		$$
		Then $f:\mathbb{C}^n\rightarrow [0,+\infty)$ is clear a complex Minkowski norm, and for any $g\in\mbox{Iso}(D)$,
		we have
		$$f(g_\ast(v))=F(g(0);g_\ast(v))=F(0;v)=f(v),$$
		that is, $f$ is $\mbox{Iso}(D)$-invariant.
		Now for any fixed point $z_0\in D$ there exists a $\phi\in\mbox{Aut}(D)$ such that $\phi(z_0)=0$ since $D$ is homogeneous. Thus for any $v\in T_{z_0}^{1,0}D$ we have
		$$
		F(z_0;v)=F(\phi(z_0);\phi_\ast(v))=F(0;\phi_\ast(v))=f(\phi_\ast(v)).
		$$
		
		Conversely, suppose that $f:\mathbb{C}^n\rightarrow [0,+\infty)$ is a complex Minkowski norm which is
		$\mbox{Iso}(D)$-invariant such that \eqref{fz0v} holds for any fixed $z_0\in D$ and for any $v\in T_{z_0}^{1,0}D$. Then \eqref{fz0v} actually defines an $\mbox{Aut}(D)$-invariant strongly pseudoconvex complex Finsler metric $F:T^{1,0}D\rightarrow [0,+\infty)$. Indeed, for a $g\in\mbox{Aut}(D)$, we set $w_0:=g(z_0)\in D$. Since $D$ is homogeneous,
		there exists a $\psi\in \mbox{Aut}(D)$
		such that $\psi(w_0)=0$.  By \eqref{fz0v}, we have
		\begin{eqnarray*}
			F(g(z_0);g_\ast(v))=F(w_0;g_\ast(v))=f(\psi_\ast(g_\ast(v)))=f((\psi\circ g\circ\phi^{-1})_\ast(\phi_\ast(v))),
		\end{eqnarray*}
		where $(\psi\circ g\circ \phi^{-1})_\ast$ denotes the tangent map at the origin.
		Now it is easy to check that $\psi\circ g\circ \phi^{-1}\in\mbox{Iso}(D)$. By the $\mbox{Iso}(D)$-invariance of $f$, we have
		\begin{eqnarray*}
			f((\phi\circ g\circ\phi^{-1})_\ast(\phi_\ast(v)))=f(\phi_\ast(v))=F(z_0;v).
		\end{eqnarray*}
		Therefore
		$$
		F(g(z_0);g_\ast(v))=F(z_0;v),\quad \forall g\in\mbox{Aut}(D),
		$$
		that is, $F$ is $\mbox{Aut}(D)$-invariant. Note that it is also easy to check that the definition $F(z_0;v)$ in  \eqref{fz0v} is independent of the choice of $\phi$ satisfying $\phi(z_0)=0$.
	\end{proof}
	
	The following theorem gives us a criterion for an $\mbox{Aut}(D)$-invariant strongly pseudoconvex complex Finsler metric on a homogeneous domain $D$ to be a K\"ahler-Berwald metric.
	\begin{theorem}\label{th-2.1}
		Let $D$ be a homogeneous domain in $\mathbb{C}^n$ which contains the origin and $F:T^{1,0}D\rightarrow [0,+\infty)$ is an $\mbox{Aut}(D)$-invariant strongly pseudoconvex complex Finsler metric. If
		\begin{equation}
			\frac{\partial^2F^2}{\partial w_i\partial \overline{\xi_j}}\Big|_{(0;\xi)}=0,\quad \forall i,j=1,\cdots,n\label{Gz}
		\end{equation}
		holds for any nonzero tangent vector $\xi\in T_{0}^{1,0}D$, then $F$ is an $\mbox{Aut}(D)$-invariant K\"ahler-Berwald metric.
	\end{theorem}
	
	\begin{proof}
		Since $D$ is homogeneous, for any point $z_0\in D$, there exist $f\in \mbox{Aut}(D)$ such that $f(z_0)=0$. Denote $w=f(z)=(w_1(z),\cdots,w_n(z))$ and $\left(\frac{\partial w_i}{\partial z_j}\right)$ the Jacobian matrix of $f$ at $z$.
		Since for any nonzero tangent vector $\xi\in T_{0}^{1,0}D$, we have \eqref{Gz}, thus by \eqref{CN}, we have
		\begin{equation*}
			\tilde{\varGamma}_{;b}^a(0;\xi)\equiv 0,\quad \forall a,b=1,\cdots,n,
		\end{equation*}
		holds for any nonzero tangent vectors $\xi\in T_{0}^{1,0}D$. Taking $z=z_0$ in \eqref{Nwz},
		we have
		$$\sum_{a,b=1}^n\frac{\partial w_l}{\partial z_a}\varGamma_{;b}^a(z_0;v)\frac{\partial z_b}{\partial w_i}-\sum_{a,b=1}^n
		\frac{\partial^2 w_l}{\partial z_a\partial z_b}\frac{\partial z_a}{\partial w_i} v_b=0,$$
		from which we get
		\begin{equation}
			\varGamma_{;c}^e(z_0;v)=\sum_{b,l=1}^n\frac{\partial^2 w_l}{\partial z_b\partial z_c}\frac{\partial z_e}{\partial w_l}v_b,\label{nce}.
		\end{equation}
		Differentiating \eqref{nce} with respect to $v_b$, we get
		\begin{equation}
			\varGamma_{b;c}^e(z_0;v)=\sum_{b,l=1}^n\frac{\partial^2 w_l}{\partial z_b\partial z_c}\frac{\partial z_e}{\partial w_l},\label{gab}
		\end{equation}
		which are obvious independent of $v$ and symmetric with the lower indexes $b$ and $c$.
		Since $z_0$ is any fixed point in $D$, changing $z_0$ to $z$ if necessary we see that \eqref{gab} holds for any $z\in D$ and  $0\neq v\in T_z^{1,0}D$,
		hence $F$ is a K\"ahler-Berwald metric.
	\end{proof}
	
	\begin{remark}\label{RK1}
		Notice that the right hand side of the equality \eqref{gab} depends only on $\mbox{Aut}(D)$, namely they are independent of the choice of $F$ whenever $F$ satisfies the assumption of the above theorem. In particular, if $D$ is an irreducible classical  domain of type $I-IV$, and \eqref{Gz} holds at the origin $0\in D$ for any nonzero tangent vector $\xi\in T_{0}^{1,0}D$, then the horizontal Chern-Finsler connection coefficients associated to any $F$ coincides with the connection coefficients of the Chern connection associated to the Bergman metric (which is a complete metric on $D$). Consequently,  $F$ and the Bergman metric on $D$ share the same connection coefficients $\varGamma_{b;c}^e$ given by \eqref{gab},   hence $F$ shares the same geodesic (as point set) as that of the Bergman metric on $D$, namely $F$ is a complete metric on $D$.
	\end{remark}

	\begin{proposition}\label{hsc}
		Let $F:T^{1,0}D\rightarrow [0,+\infty)$ be an $\mbox{Aut}(D)$-invariant K\"ahler-Berwald metric on a homogeneous domain $D$ which contains the origin in $\mathbb{C}^n$. If furthermore,
		$$\frac{\partial^2F^2}{\partial z_i\partial\overline{v_l}}(0;v)=0\quad \mbox{or equivalently}\quad \varGamma_{;i}^l(0;v)=0,\quad\forall i,l=1,\cdots,n.$$
		Then the holomorphic sectional curvature along any nonzero tangent vector $v\in T_0^{1,0}D$ is given by
		\begin{equation}
			K(0;v)=-\frac{2}{F^4(0;v)}\sum_{i,j=1}^n\frac{\partial^2F^2}{\partial z_i\partial \overline{z_j}}(0;v)v_i\overline{v_j}.\label{hscb}
		\end{equation}
	\end{proposition}
	\begin{proof}In local coordinates $(z;v)=(z_1,\cdots,z_n;v_1,\cdots,v_n)$ on $T^{1,0}D$, the holomorphic sectional curvature $K$ is given by (cf. (2.5.11) on page 109 in \cite{AP}):
		\begin{equation}
			K(z;v)=-\frac{2}{F^4(z;v)}\sum_{i,j,l=1}^n\frac{\partial F^2}{\partial v_l}\delta_{\bar{j}}(\varGamma_{;i}^l)v_i\overline{v_j},\label{kzv}
		\end{equation}
		where $\delta_{\bar{j}}=\frac{\partial}{\partial \overline{z_j}}-\overline{\varGamma_{;j}^l}\frac{\partial}{\partial \overline{v_l}}$.
		Since $F$ is a K\"ahler-Berwald metric,  $\varGamma_{;i}^l$ are holomorphic with respect to  $v$.Thus $\delta_{\bar{j}}(\varGamma_{;i}^l)=\frac{\partial}{\partial\overline{z_j}}(\varGamma_{;i}^l)$. Using \eqref{CN} and the $(1,1)$-homogeneity property of $G$ with respect to $v$, it follows from \eqref{kzv} that at the point $(0;v)$ we have
		$$
		K(0;v)=-\frac{2}{F^4}\sum_{i,j,l=1}^n\left(-\frac{\partial^2F^2}{\partial\overline{z_j} \partial v_l}\varGamma_{;i}^l(0;v)+\frac{\partial^2F^2}{\partial z_i\partial\overline{z_j}}\right)v_i\overline{v_j}=-\frac{2}{F^4(0;v)}\sum_{i,j=1}^n\frac{\partial^2F^2}{\partial z_i\partial\overline{z_j}}(0;v)v_i\overline{v_j},
		$$
		where in the last equality we use the condition $\varGamma_{;i}^l(0;v)=0$.
	\end{proof}
	
	\begin{proposition}\label{bhsc}
		Let $F:T^{1,0}D\rightarrow [0,+\infty)$ be an $\mbox{Aut}(D)$-invariant K\"ahler-Berwald metric on a homogeneous domain $D$ which contains the origin in $\mathbb{C}^n$. If furthermore,
		$$\frac{\partial^2F^2}{\partial z_i\partial\overline{v_l}}(0;v)=0\quad \mbox{or equivalently}\quad \varGamma_{;i}^l(0;v)=0,\quad\forall i,l=1,\cdots,n.$$
		Then the holomorphic bisectional curvature along any nonzero tangent vectors $v,w\in T_0^{1,0}D$ is given by
		\begin{equation}
			B(0;v,w)=-\frac{2}{F^2(0;v)F^2(0;w)}\frac{\partial^2F^2}{\partial z_i\partial \overline{z_j}}(0;v)w_i\overline{w_j}.\label{hbscb}
		\end{equation}
	\end{proposition}
	
	\begin{proof}
		In local coordinates $(z;v)=(z_1,\cdots,z_n;v_1,\cdots,v_n)$ on $T^{1,0}D$, we have
		$$
		B(z;v,w)=-\frac{2}{F^2(z;v)F^2(z;w)}\sum_{a,i,j,l=1}^n\frac{\partial F^2}{\partial v_l} \delta_{\bar{j}}(\varGamma_{a;i}^l)w_i\overline{w_j}v_a.
		$$
		As in the proof of Proposition \ref{hsc}, using the K\"ahler-Berwald condition and the condition $\varGamma_{;i}^l(0;v)=0$, we have
		\begin{equation}
			B(0;v,w)=-\frac{2}{F^2(0;v)F^2(0;w)}\sum_{i,j=1}^n\frac{\partial^2F^2}{\partial z_i\partial \overline{z_j}}w_i\overline{w_j}.\label{hbscc}
		\end{equation}
	\end{proof}
	
	In the next section, we shall use \eqref{hscb} and \eqref{hbscc} to derive the holomorphic sectional curvatures and bisectional curvatures of \eqref{FI}, \eqref{FII} and \eqref{FIII} and \eqref{FIV}, respectively.

	\section{Holomorphic invariant metrics on the irreducible classical domains}\label{section-3}

	Denote $\mathscr{M}(m,n)$ the set of all $m\times n$ matrices over the complex number field $\mathbb{C}$, which is a complex vector space isomorphic to $\mathbb{C}^{m n}$. In the following, $\mathbb{C}^n$ means $\mathscr{M}(1,n)$.
	If $m=n$, we denote $\mathscr{M}(n)$ instead of $\mathscr{M}(n,n)$.
	Let $Z\in\mathscr{M}(m,n)$, we denote $Z'$ the transpose of $Z$ and $\overline{Z}$ the complex conjugate of $Z$. If $A\in \mathscr{M}(n)$ be a Hermitian matrix, we also denote $A>0$ to mean
	$A$ is positive definite, and $A\geq 0$ to mean $A$ is semi-positive definite. If $A,B\in\mathscr{M}(n)$ are Hermitian matrices, we denote $A\leq B$ to mean
	$B-A\geq 0$.
	
	A classical domain is a domain which is equivalent to one, or to the topological product of several, of the following irreducible classical domains (cf. \cite{Hua},\cite{Look1, Look2, Look3}):
	\begin{eqnarray*}
		\mathfrak{R}_I(m;n)    &:& I_m-Z\overline{Z'}>0,\; Z\in \mathscr{M}(m,n),\quad (m\leq n),\\
		\mathfrak{R}_{II}(p) &:& I_p-Z\overline{Z}>0,\; Z=Z',\;Z\in\mathscr{M}(p), \quad p\geq 2,\\
		\mathfrak{R}_{III}(q)&:& I_q+Z\overline{Z}>0,\; Z=-Z',\; Z\in\mathscr{M}(q), \quad q\geq 4,\\
		\mathfrak{R}_{IV}(N) &:& 1+|zz'|^2-2z\overline{z'}>0,\;1-|zz'|>0, z=(z_1,\cdots,z_N)\in\mathbb{C}^N,\quad N\geq 5,
	\end{eqnarray*}
	where $I_m$ denotes the $m\times m$ unit matrix. These are complete circular domains with center $0$ and convex. The holomorphic automorphism  $\mbox{Aut}(\mathfrak{R}_A)$ acts on $\mathfrak{R}_A$ transitively for $A=I,II,III,IV$, respectively.
	
	Note that
	$$\mathfrak{R}_I\subset \mathbb{C}^{mn}, \quad \mathfrak{R}_{II}\subset\mathbb{C}^{\frac{p(p+1)}{2}},\quad \mathfrak{R}_{III}\subset\mathbb{C}^{\frac{q(q-1)}{2}},\quad \mathfrak{R}_{IV}\subset\mathbb{C}^N$$
	and
	\begin{eqnarray*}
		\dim_{\mathbb{C}}\mathfrak{R}_I&=&mn,\quad \dim_{\mathbb{C}}\mathfrak{R}_{II}=\frac{p(p+1)}{2},\quad \dim_{\mathbb{C}}\mathfrak{R}_{III}=\frac{q(q-1)}{2},\quad\dim_{\mathbb{C}}\mathfrak{R}_{IV}=N,\\
		\mbox{rank}\,\mathfrak{R}_I&=&m,\quad\mbox{rank}\,\mathfrak{R}_{II}=p,\quad \mbox{rank}\,\mathfrak{R}_{III}=\left[\frac{q}{2}\right],\quad \mbox{rank}\,\mathfrak{R}_{IV}=2.
	\end{eqnarray*}
	
	It is well-known that (cf. \cite{Hua},\cite{Look3})
	\begin{eqnarray*}
		ds_I^2
		&=&(m+n)\mbox{tr}\left[(I-Z\overline{Z'})^{-1}dZ(I-\overline{Z'}Z)^{-1}\overline{dZ'}\right],\\
		ds_{II}^2
		&=&(p+1)\mbox{tr}\left[(I-\overline{Z}Z)^{-1}\overline{dZ}(I-Z\overline{Z})^{-1}dZ\right],\\
		ds_{III}^2
		&=&(q-1)\mbox{tr}\left[(I+Z\overline{Z})^{-1}dZ(I+\overline{Z}Z)^{-1}\overline{dZ'}\right],\\
		ds_{IV}^2&=&\frac{2N}{(1+|zz'|^2-2z\overline{z'})^2}\Big[
		(1+|zz'|^2-2z\overline{z'})dz\overline{dz'}+4z\overline{z'}|zdz'|^2\nonumber\\
		&&-2(\overline{zz'}zdz'z\overline{dz'}+zz'\overline{zdz'}\overline{z}dz')
		-2(|zdz'|^2-|z\overline{dz'}|^2)
		\Big]
	\end{eqnarray*}
	are the corresponding Bergman metrics on $\mathfrak{R}_{A}(A=I,II,III,IV)$ respectively. They are all complete K\"ahler metrics with negative holomorphic sectional curvatures and nonpositive bisectional curvatures, and they are all Hermitian quadratic metrics.
	
	A natural question one may ask is whether $\mathfrak{R}_A$ admits $\mbox{Aut}(\mathfrak{R}_A)$-invariant strongly pseudoconvex complex Finsler metrics which are non-Hermitian quadratic and can be explicitly expressed? For $m=1$, $\mathfrak{R}_I$ is the usual open unit ball $B_n$ in $\mathbb{C}^n$. Since it was proved in \cite{Zh2} that, other than a positive constant multiple of the Bergman metric, there exists no $\mbox{Aut}(B_n)$-invariant strongly pseudoconvex complex Finsler metrics on $B_n$, in the following we assume that $m\geq 2$.

	\subsection{Complex Minkowski norms on $\mathscr{M}(m,n)$ which are $U(m)\otimes U(n)$-invariant}\label{subsection-3.1}
	
	Since $\mathfrak{R}_A (A=I,II,III,IV)$ are homogeneous domains, by Theorem \ref{hfm}, to construct $\mbox{Aut}(\mathfrak{R}_A)$-invariant strongly pseudoconvex complex Finsler metrics it suffices to construct $\mbox{Iso}(\mathfrak{R}_A)$-invariant complex Minkowski norms on $T_0^{1,0}\mathfrak{R}_A$.  For this purpose, we first construct some complex Minkowski norms on the more general complex vector space $\mathscr{M}(m,n)$.
	
	Note that there is a natural Hermitian inner product on $\mathscr{M}(m,n)$:
	\begin{equation*}
		\langle Z,W\rangle:=\mbox{tr}\{Z\overline{W'}\},\quad \forall Z,W\in\mathscr{M}(m,n).
	\end{equation*}

	\begin{definition}\label{def-4-2}
		Let $t\in[0,+\infty)$, $k\in\mathbb{N}$ and $k\geq 2$. Define a complex norm $\pmb{f}:\mathscr{M}(m,n)\rightarrow[0,+\infty)$ by setting
		\begin{equation}
			\pmb{f}^2(V):=\frac{1}{1+t}\left\{\mbox{tr}\left\{V\overline{V'}\right\}+t\sqrt[k]{\mbox{tr}\left\{(V\overline{V'})^k\right\}}\right\},\quad \forall V\in\mathscr{M}(m,n).\label{bf}
		\end{equation}
	\end{definition}
	
	\begin{remark} If $m=1$, then
		\begin{equation*}
			\mbox{tr}\left\{(V\overline{V'})^k\right\}=\left\{\mbox{tr}(V\overline{V'})\right\}^k
		\end{equation*}
		for any $k\in\mathbb{N}$ with $k\geq 2$, hence $\pmb{f}^2(V)=\mbox{tr}\{V\overline{V'}\}$ is the usual
		complex Euclidean norm on $\mathscr{M}(1,n)\cong \mathbb{C}^n$ which is $U(1\times n)$-invariant;
		If $m\geq 2$, however, taking
		$
		V=\begin{pmatrix}
			a & 0 \\
			0 & b \\
		\end{pmatrix}
		$
		with $a,b\in\mathbb{C}$ and $ab\neq 0$, then
		$$\mbox{tr}\left\{(V\overline{V'})^k\right\}=|a|^{2k}+|b|^{2k}\neq (|a|^2+|b|^2)^k=\left\{\mbox{tr}(V\overline{V'})\right\}^k.$$
		Therefore, as a complex norm on $\mathscr{M}(m,n)\cong\mathbb{C}^{mn}$ with $m\geq 2$, $\pmb{f}$ is not $U(m\times n)$-invariant.
	\end{remark}
	
	For $m\in\mathbb{N}$, we denote $U(m)$ the unitary group consisting of all matrices $A\in\mathscr{M}(m)$ satisfying $A\overline{A'}=I_m$.
	The following proposition is immediate.
	\begin{proposition}\label{prop-4-2}
		For any $A\in U(m)$ , $B\in U(n)$ and $ \forall V\in\mathscr{M}(m,n) $ we have
		\begin{eqnarray}
			\pmb{f}(AVB)&=&\pmb{f}(V),\label{ftd}\\
			\pmb{f}(AV'B)&=&\pmb{f}(V).\label{ftd1}
		\end{eqnarray}
	\end{proposition}
	\begin{proof}
		Indeed, using the identities $\mbox{tr}\{XY\}=\mbox{tr}\{YX\}$ and   $\mbox{tr}\{Z\}=\mbox{tr}\{Z'\}$ for any $X,Y,Z\in\mathscr{M}(m,n)$, we have
		\begin{eqnarray*}
			\mbox{tr}\left\{(AVB)\overline{(AVB)'}\right\}=\mbox{tr}\left\{V\overline{V'}\right\},\quad \mbox{tr}\left\{(AV'B)\overline{(AV'B)'}\right\}
			=\mbox{tr}\left\{V'\overline{V}\right\}
			=\mbox{tr}\left\{V\overline{V'}\right\}
		\end{eqnarray*}
		and
		\begin{eqnarray*}
			\mbox{tr}\left\{\left[(AVB)\overline{(AVB)'}\right]^k\right\}
			&=&\mbox{tr}\left\{\underbrace{(AVB)(\overline{B'V'A'})\cdots(AVB)(\overline{B'V'A'})}_{k\;\mbox{times}}\right\}\\
			&=&\mbox{tr}\left\{A(V\overline{V'})^k \overline{A'}\right\}=\mbox{tr}\left\{\overline{A'}A(V\overline{V'})^k \right\}
			=\mbox{tr}\left\{(V\overline{V'})^k\right\}\\
			\mbox{tr}\left\{\left[(AV'B)\overline{(AV'B)'}\right]^k\right\}
			&=&\mbox{tr}\left\{A(V '\overline{V})^k \overline{A'}\right\}
			=\mbox{tr}\left\{(V'\overline{V})^k \right\}
			=\mbox{tr}\left\{(V\overline{V'})^k\right\}.
		\end{eqnarray*}
		This completes the proof of \eqref{ftd}.
	\end{proof}
	
	\begin{remark}\label{3.2} For every $V\in\mathscr{M}(m,n), $ there exist unitary matrices $A,B$ such that
		$$ AVB=\begin{pmatrix}
			\lambda_1(V) &0 & 0&0&\cdots&0 \\
			0 & \ddots &0& \vdots &\ddots&\vdots\\
			0 & 0 & \lambda_{m}(V)&0 &\cdots&0\\
		\end{pmatrix},\quad \lambda_1(V) \geq\lambda_2(V)\cdots\geq\lambda_m(V)\geq0.
		$$
		Thus
		\begin{equation}
			\pmb{f}^2(V):=\frac{1}{1+t}\left\{\sum_{i=1}^m\lambda_i^2(V)+t\sqrt[k]{\sum_{i=1}^m\lambda_i^{2k}(V)}\right\}.\label{bf1}
		\end{equation}
		
		In fact any  function  $\pmb{f}:\mathscr{M}(m,n)\rightarrow[0,+\infty)$ satisfying condition \eqref{ftd} must be a symmetric function with respect to $ \lambda_1(V),\cdots,\lambda_m(V). $
	\end{remark}

	The following proposition is key to us in the constructions of holomorphic invariant metrics on $\mathfrak{R}_I,\mathfrak{R}_{II}$ and $\mathfrak{R}_{III}$, which shows that the complex norm $\pmb{f}$ has good smoothness and convexity. This together with Proposition \ref{prop-4-2} makes it possible for us to construct holomorphic invariant strongly pseudoconvex complex Finsler metrics  on $\mathfrak{R}_I,\mathfrak{R}_{II}$ and $\mathfrak{R}_{III}$, respectively, which are not necessary Hermitian quadratic.
	
	\begin{proposition}\label{prop-4-3}
		$\pmb{f}$ is a complex Minkowski norm on $\mathscr{M}(m,n)$ for any fixed $t\in[0,+\infty),k\in\mathbb{N}$ and $k\geq 2$.
	\end{proposition}
	\begin{proof} For $t=0$, $\pmb{f}^2(V)=\mbox{tr}\{V\overline{V'}\}$ is the canonical complex Euclidean norm of $V\in \mathscr{M}(m,n)\cong \mathbb{C}^{mn}$, which is smooth and strongly pseudoconvex on the whole $\mathscr{M}(m,n)$, hence a complex Minkowski norm on $\mathscr{M}(m,n)$.
		
		For any $t\in(0,+\infty)$ and any $k\in\mathbb{N}$ with $k\geq 2$, $\pmb{f}^2$ is smooth on $\mathscr{M}(m,n)\setminus\{0\}$. Thus in order to show that $\pmb{f}$ is a complex Minkowski norm on $\mathscr{M}(m,n)$, it suffices to show that $\pmb{f}^2$ is strongly pseudoconvex at any $V=(V_{ij})\in \mathscr{M}(m,n)\setminus\{0\}$.
		Notice that
		$$
		\frac{\partial}{\partial V_{ij}}V=E_{ij},\quad \frac{\partial}{\partial V_{ij}}\overline{V'}=0,\quad \frac{\partial}{\partial\overline{V_{kl}}}\overline{V'}=E_{kl}'=E_{lk},\quad \frac{\partial}{\partial \overline{V_{kl}}}V=0,
		$$
		where $E_{ij}\in \mathscr{M}(m,n)$ whose only nonzero element is $1$ on row $i$ and column $j$, and $E_{kl}'\in\mathscr{M}(n,m)$ denotes transpose of $E_{kl}$. Denote $\mathcal{A}:=\mbox{tr}\left\{(V\overline{V'})^k\right\}$, then
		\begin{eqnarray}
			\frac{\partial^2 \pmb{f}^2}{\partial V_{ij}\partial \overline{V_{ab}}}
			&=&\frac{1}{1+t}\left\{\delta_{ia}\delta_{jb}
			+\frac{t}{k}\mathcal{A}^{\frac{1}{k}-1}\frac{\partial^2\mathcal{A}}{\partial V_{ij}\partial \overline{V_{ab}}}+\frac{t}{k^2}(1-k)\mathcal{A}^{\frac{1}{k}-2}\frac{\partial\mathcal{A}}{\partial V_{ij}}\frac{\partial\mathcal{A}}{\partial \overline{V_{ab}}}\right\}.
			\label{lfof}
		\end{eqnarray}
		
		Now let $W=(W_{ij})\in \mathscr{M}(m,n)$, by a direct calculation we get
		\begin{eqnarray*}
			\sum_{i,a=1}^m\sum_{j,b=1}^n\frac{\partial^2\mathcal{A}}{\partial V_{ij}\partial \overline{V_{ab}}}W_{ij}\overline{W_{ab}}
			&=&k\left\{\mbox{tr}\left\{W\overline{W'}(V\overline{V'})^{k-1}\right\}+\sum_{i=0}^{k-2}\mbox{tr}\left\{W(\overline{V'}V)^{i+1}\overline{W'}(V\overline{V'})^{k-i-2}\right\}\right\},\\
			\sum_{i,a=1}^m\sum_{j,b=1}^n\frac{\partial\mathcal{A}}{\partial V_{ij}}\frac{\partial\mathcal{A}}{\partial \overline{V_{ab}}}W_{ij}\overline{W_{ab}}
			&=&k^2\mbox{tr}\left\{W\overline{V'}(V\overline{V'})^{k-1}\right\}\mbox{tr}\left\{(V\overline{V'})^{k-1}V\overline{W'}\right\},
		\end{eqnarray*}
		which together with \eqref{lfof} yields
		\begin{eqnarray*}
			\sum_{i,a=1}^m\sum_{j,b=1}^n\frac{\partial^2 \pmb{f}^2}{\partial V_{ij}\partial \overline{V_{ab}}}W_{ij}\overline{W_{ab}}
			&=&\frac{1}{1+t}\Bigg\{\mbox{tr}\left\{W\overline{W'}\right\}+t\mathcal{A}^{\frac{1}{k}-1}
			\mbox{tr}\left\{W\overline{W'}(V\overline{V'})^{k-1}\right\}\\
			&&+t\mathcal{A}^{\frac{1}{k}-2}\sum_{i=0}^{k-2}\Big[\mathcal{A}\cdot\mbox{tr}\left\{W(\overline{V'}V)^{i+1}\overline{W'}(V\overline{V'})^{k-i-2}\right\}\\
			&&-\mbox{tr}\left\{W\overline{V'}(V\overline{V'})^{k-1}\right\}\mbox{tr}\left\{(V\overline{V'})^{k-1}V\overline{W'}\right\}\Big]\Bigg\}.
		\end{eqnarray*}
		Notice that $\mbox{tr}\left\{W\overline{W'}\right\}$ and $\mbox{tr}\left\{W\overline{W'}(V\overline{V'})^{k-1}\right\}$ are nonnegative,
		in order to show that the right hand side of the above equality is nonnegative, it suffices to show that
		\begin{equation}
			\mathcal{A}\cdot\mbox{tr}\left[W(\overline{V'}V)^{i+1}\overline{W'}(V\overline{V'})^{k-i-2}\right]
			-\mbox{tr}\left[W\overline{V'}(V\overline{V'})^{k-1}\right]\mbox{tr}\left[(V\overline{V'})^{k-1}V\overline{W'}\right]\label{tiq}
		\end{equation}
		is nonnegative for any fixed integer $k\in\mathbb{N}$, $k\geq 2$ and $i\in\{0,1,2,\cdots,k-2\}$. Indeed, using the following identities:
		$$
		\mbox{tr}\{AB\}=\mbox{tr}\{BA\}, \quad \mbox{tr}\{C\}=\mbox{tr}\{C'\},\quad \forall A,B\in\mathscr{M}(m,n),C\in\mathscr{M}(m),
		$$
		we are able to rewrite
		\begin{eqnarray*}
			\mathcal{A}=\mbox{tr}\left[(V\overline{V'})^k\right]&=&\mbox{tr}\left[V(\overline{V'}V)^{i+1}\overline{V'}(V\overline{V'})^{k-i-2}\right],\\
			\mbox{tr}\left[W\overline{V'}(V\overline{V'})^{k-1}\right]&=&\mbox{tr}\left[W(\overline{V'}V)^{i+1}\overline{V'}(V\overline{V'})^{k-i-2}\right],\\
			\mbox{tr}\left[(V\overline{V'})^{k-1}V\overline{W'}\right]&=&\overline{\mbox{tr}\left[W(\overline{V'}V)^{i+1}\overline{V'}(V\overline{V'})^{k-i-2}\right]}.
		\end{eqnarray*}
		So that in order to show \eqref{tiq} is nonnegative, it suffices to show that
		\begin{eqnarray*}
			\left|\mbox{tr}\left[W(\overline{V'}V)^{i+1}\overline{V'}(V\overline{V'})^{k-i-2}\right]\right|^2\leq \mbox{tr}\left\{(V\overline{V'})^k\right\}\mbox{tr}\left\{W(\overline{V'}V)^{i+1}\overline{W'}(V\overline{V'})^{k-i-2}\right\}.
		\end{eqnarray*}
		
		Indeed, on the one hand, we have
		\begin{eqnarray*}
			&&\mbox{tr}\left\{W(\overline{V'}V)^{i+1}\overline{W'}(V\overline{V'})^{k-i-2}\right\}\nonumber\\
			&=&\left\{
			\begin{array}{ll}
				\mbox{tr}\left\{\left[(V\overline{V'})^{s}W(\overline{V'}V)^{l}\right]\overline{\left[(V\overline{V'})^{s}W(\overline{V'}V)^{l}\right]'}\right\}, & \hbox{if}\;\;i=2l-1, k=2s+i+2 \\
				\mbox{tr}\left\{\left[\overline{V'}(V\overline{V'})^{s}W(\overline{V'}V)^{l}\right]\overline{\left[\overline{V'}(V\overline{V'})^{s}W(\overline{V'}V)^{l}\right]'}\right\}, &\hbox{if}\;\; i=2l-1, k=2s+i+3 \\
				\mbox{tr}\left\{\left[(V\overline{V'})^{s}W(\overline{V'}V)^{l}\overline{V'}\right]\overline{\left[(V\overline{V'})^{s}W(\overline{V'}V)^{l}\overline{V'}\right]'}\right\}, &\hbox{if}\;\; i=2l, k=2s+i+2 \\
				\mbox{tr}\left\{\left[\overline{V'}(V\overline{V'})^{s}W(\overline{V'}V)^{l}\overline{V'}\right]\overline{\left[\overline{V'}(V\overline{V'})^{s}W(\overline{V'}V)^{l}\overline{V'}\right]'}\right\}, &\hbox{if}\;\; i=2l, k=2s+i+3.
			\end{array}
			\right.
		\end{eqnarray*}
		On the other hand, using the Cauchy-Schwarz inequality $|\mbox{tr}\{X\overline{Y'}\}|^2\leq \mbox{tr}\{X\overline{X'}\}\mbox{tr}\{Y\overline{Y'}\}$, we have
		\begin{eqnarray*}
			&&\mbox{tr}\left[(V\overline{V'})^{k}\right]\mbox{tr}\left[W(\overline{V'}V)^{i+1}\overline{W'}(V\overline{V'})^{k-i-2}\right]\\
			&=&\left\{
			\begin{array}{ll}
				\mbox{tr}\left[(V\overline{V'})^{2s+2l+1}\right]\mbox{tr}\left[W(\overline{V'}V)^{2l}\overline{W'}(V\overline{V'})^{2s}\right], &\hbox{if}\;\;i=2l-1, k=2s+i+2\\
				\mbox{tr}\left[(V\overline{V'})^{2s+2l+2}\right]\mbox{tr}\left[W(\overline{V'}V)^{2l}\overline{W'}(V\overline{V'})^{2s+1}\right], &\hbox{if}\;\;i=2l-1, k=2s+i+3\\
				\mbox{tr}\left[(V\overline{V'})^{2s+2l+2}\right]\mbox{tr}\left[W(\overline{V'}V)^{2l+1}\overline{W'}(V\overline{V'})^{2s}\right], &\hbox{if}\;\;i=2l, k=2s+i+2\\
				\mbox{tr}\left[(V\overline{V'})^{2s+2l+3}\right]\mbox{tr}\left[W(\overline{V'}V)^{2l+1}\overline{W'}(V\overline{V'})^{2s+1}\right], &\hbox{if}\;\;i=2l, k=2s+i+3\\
			\end{array}
			\right.\\
			&\geq&\left\{
			\begin{array}{ll}
				\left|\mbox{tr}\left\{(\overline{V'}V)^{s+l}\overline{V'}\left[(V\overline{V'})^{s}W(\overline{V'}V)^{l}\right]\right\}\right|^2,&\hbox{if}\;\; i=2l-1, k=2s+i+2\\
				\left|\mbox{tr}\left\{(\overline{V'}V)^{s+l+1}\left[\overline{V'}(V\overline{V'})^{s}W(\overline{V'}V)^{l}\right]\right\}\right|^2, & \hbox{if}\;\; i=2l-1, k=2s+i+3\\
				\left|\mbox{tr}\left\{(V\overline{V'})^{s+l+1}\left[(V\overline{V'})^{s}W(\overline{V'}V)^{l}\overline{V'}\right]\right\}\right|^2, & \hbox{if}\;\;i=2l, k=2s+i+2\\
				\left|\mbox{tr}\left\{(\overline{V'}V)^{s+l+1}\overline{V'}\left[V(\overline{V'}V)^{l}\overline{W'}(V\overline{V'})^{s}V\right]\right\}\right|^2, & \hbox{if}\;\; i=2l, k=2s+i+3
			\end{array}
			\right.\\
			&=&\left|\mbox{tr}\left[W(\overline{V'}V)^{i+1}\overline{V'}(V\overline{V'})^{k-i-2}\right]\right|^2.
		\end{eqnarray*}
		Therefore for any $0\neq V\in\mathscr{M}(m,n)$ and any $W\in\mathscr{M}(m,n)$, we have
		\begin{equation*}
			\sum_{i,a=1}^m\sum_{j,b=1}^n\frac{\partial^2 \pmb{f}^2}{\partial V_{ij}\partial \overline{V_{ab}}}W_{ij}\overline{W_{ab}}
			\geq\frac{1}{1+t}\left\{ \mbox{tr}\{W\overline{W'}\}+t\mathcal{A}^{\frac{1}{k}-1}\mbox{tr}\left[W\overline{W'}(V\overline{V'})^{k-1}\right]\right\}\geq 0,
		\end{equation*}
		and equality holds if and only if $W=0$. Hence $\pmb{f}$ is a complex Minkowski norm on $\mathscr{M}(m,n)$.
	\end{proof}
	
	\begin{remark}\label{rm-sc}
		It is clear that the complex Minkowski norm $\pmb{f}$ in Proposition \ref{prop-4-3} is strongly convex at least for sufficiently small $t>0$.
	\end{remark}

	\subsection{$\mbox{Aut}(\mathfrak{R}_I)$-invariant K\"ahler-Berwald metrics on $\mathfrak{R}_I$}\label{subsection-3.2}
	
	Note that a point $Z=(Z_{ij})\in \mathfrak{R}_I(m;n)$ is identified  with a point
	\begin{equation}
		z=(z_{11},\cdots,z_{1n},z_{21},\cdots,z_{2n},\cdots,z_{m1},\cdots,z_{mn})\in\mathbb{C}^{mn}\label{cI}
	\end{equation}
	by setting
	\begin{equation}
		z_{ij}:=Z_{ij},\quad \forall 1\leq i\leq m\;\mbox{and}\; 1\leq j\leq n. \label{cIa}
	\end{equation}
	A tangent vector $\pmb{V}=\displaystyle\sum_{i=1}^m\sum_{j=1}^nV_{ij}\frac{\partial}{\partial Z_{ij}}$ at $Z_0\in\mathfrak{R}_I$ is isomorphic to a matrix $V=(V_{ij})\in\mathscr{M}(m,n)$,
	which is also identified  with a vector
	\begin{equation}
		v=(v_{11},\cdots,v_{1n},v_{21},\cdots,v_{2n},\cdots,v_{m1},\cdots,v_{mn})\in\mathbb{C}^{mn}\label{vI}
	\end{equation}
	at $z_0\in\mathbb{C}^{mn}$ (which corresponds to $Z_0$) by setting
	\begin{equation}
		v_{ij}:=V_{ij}, \quad \forall 1\leq i\leq m\;\mbox{and}\;1\leq j\leq n.\label{vIa}
	\end{equation}
	
	In the following $(z;v)$ serves as the global complex coordinates on $T^{1,0}\mathfrak{R}_I\cong \mathfrak{R}_I\times \mathscr{M}(m,n)$. For convenience, we also denote $(Z;V)$  to mean the "complex coordinates" on $T^{1,0}\mathfrak{R}_I$ and write $V\in T_Z^{1,0}\mathfrak{R}_I$ to mean that $V$ is a tangent vector at $Z\in \mathfrak{R}_I$ in the sense of identifications \eqref{cI},\eqref{cIa} and \eqref{vI},\eqref{vIa}.
	
	Since $\mathfrak{R}_I$ is homogeneous. For any fixed point $Z_0\in\mathfrak{R}_I$, there exists $\varPhi_{Z_0}\in \mbox{Aut}(\mathfrak{R}_I)$ such that $\varPhi_{Z_0}(Z_0)=0$. Indeed, (cf. \cite{Hua},\cite{Look3})
	\begin{equation}
		W=\varPhi_{Z_0}(Z)=A(Z-Z_0)(I_n-\overline{Z_0'}Z)^{-1}D^{-1},\label{pI}
	\end{equation}
	where $A\in\mathscr{M}(m), D\in\mathscr{M}(n)$ satisfy
	\begin{equation}
		\overline{A'}A=(I_m-Z_0\overline{Z_0'})^{-1},\quad \overline{D'}D=(I_n-\overline{Z_0'}Z_0)^{-1}.\label{pIa}
	\end{equation}
	
	Note that for $m\neq n$, each element of $\mbox{Aut}(\mathfrak{R}_I)$ can be expressed in the form \eqref{pI}. For $m=n$,  we need to add another mapping
	\begin{equation}
		W=\varPhi_{Z_0}(Z')=A(Z'-Z_0)(I_n-\overline{Z_0'}Z')^{-1}D^{-1},\label{pI2}
	\end{equation}
	where $A\in\mathscr{M}(m), D\in\mathscr{M}(n)$ satisfy \eqref{pIa} (cf. \cite{Klingen1}).

	Thus the isotropic subgroup $ \mbox{Iso}(\mathfrak{R}_I) $ at the origin is given by
	\begin{equation}
		W=\left\{
		\begin{array}{ll}
			AZ\overline{D'}, & \hbox{for}\;m<n, \\
			AZ\overline{D'}\;\mbox{or}\;AZ'\overline{D'}, & \hbox{for}\;m=n,
		\end{array}
		\right.
		\label{isoI}
	\end{equation}
	where  $A\overline{A'}=I_m, D\overline{D'}=I_n$.
	
	By Theorem \ref{hfm}, in order to construct a strongly pseudoconvex complex Finsler metric on $\mathfrak{R}_I$, it suffices to construct a complex Minkowski norm on $T_0^{1,0}\mathfrak{R}_I\cong \mathscr{M}(m,n)$ which is $\mbox{Iso}(\mathfrak{R}_I)$-invariant.
	
	Define
	\begin{equation*}
		f_I^2(V):=(m+n)\pmb{f}^2(V),\quad \forall V\in T_0^{1,0}\mathfrak{R}_I\cong \mathscr{M}(m,n),
	\end{equation*}
	where $\pmb{f}^2(V)$ is given by \eqref{bf}.
	
	Let $A=(A_{ij})\in\mathscr{M}(m),\; B=(B_{ij})\in\mathscr{M}(n)$, the tensor product (also called Kronecker product) of $A$ and $B$ is (cf. \cite{Look3})
	\begin{equation*}
		A\otimes B:=(C_{(ab)(ij)})=\begin{pmatrix}
			A_{11}B & \cdots & A_{1m}B \\
			\vdots & \vdots & \vdots \\
			A_{m1}B & \cdots & A_{mm}B \\
		\end{pmatrix}\in\mathscr{M}(mn)
	\end{equation*}
	whose
	elements $C_{(ab)(ij)}$ at row $(ab)$ and column
	$(ij)$ are given by
	\begin{equation*}
		C_{(ab)(ij)}:=A_{ai}B_{bj}.
	\end{equation*}
	Then a complex linear mapping $W=AVB$ of $V=(V_{ij})\in T_0^{1,0}\mathfrak{R}_I$ can be identified with
	a complex linear mapping $w=v(A'\otimes B)$ of $v\in  \mathbb{C}^{mn}$, where $w$ and $v\in\mathbb{C}^{mn}$ are the vectors corresponding to $W$ and $V\in T_0^{1,0}\mathfrak{R}_I$, respectively.
	Notice that $A'\otimes B\in U(m)\otimes U(n)$ if and only if $A\in U(m)$ and $B\in U(n)$.
	Therefore, considered as a complex Minkowski norm on $\mathbb{C}^{mn}$,  Proposition \ref{prop-4-2} implies that $f_I^2$ is $U(m)\otimes U(n)$-invariant.

	The following proposition is due to Q. K. Lu \cite{Look3} (cf. page 313-314) whenever $z_0=0$.
	\begin{proposition}\label{L-I}
		Let $w=\phi_{z_0}(z)$ be the biholomorphic map associated to \eqref{pI}, where $z_0,z$ and $w$ are the corresponding complex coordinates of $Z_0,Z$ and $W$ in \eqref{pI}. Then
		\begin{eqnarray}
			(\phi_{z_0})_\ast (v)=v(A'\otimes \overline{D'}),\label{dIZ0}
		\end{eqnarray}
		where $(\phi_{z_0})_\ast$ denotes the differential of $\phi_{z_0}$ at $z_0$, $v$ is given by \eqref{vI} and $A,D$ satisfy \eqref{pIa}.
	\end{proposition}
	\begin{proof}
		Differentiating \eqref{pI}, and then setting $Z=Z_0$, we have
		\begin{equation}
			dW=AdZ\overline{D'},\label{d-Ia}
		\end{equation}
		where $A=(A_{ab})$ and $D=(D_{ij})$ satisfy \eqref{pIa}.  Substituting \eqref{cIa} into \eqref{d-Ia}, we have
		\begin{equation}
			dw_{ab}=\sum_{r=1}^m\sum_{s=1}^nA_{ar}dz_{rs}\overline{D_{bs}},
		\end{equation}
		from which it follows that
		$$
		\frac{\partial w_{ab}}{\partial z_{rs}}\Big|_{z=z_0}=A_{ar}\overline{D_{bs}}.
		$$
		Thus  if we setting $(\phi_{z_0})_\ast(v)=u=(u_{11},\cdots,u_{1n},u_{21},\cdots,u_{2n},\cdots,u_{m1},\cdots,u_{mn})$, we have
		\begin{equation*}
			u_{ab}=\sum_{r=1}^m\sum_{s=1}^nv_{rs}\frac{\partial w_{ab}}{\partial z_{rs}}=\sum_{r=1}^m\sum_{s=1}^nA_{ar}v_{rs}\overline{D_{bs}}=[v(A'\otimes \overline{D'})]_{ab},\label{tI}
		\end{equation*}
		where $[\cdots]_{ab}$ denotes the element on the row $a$ and column $b$ in the matrix $[\cdots]$.
	\end{proof}

	\begin{remark}\label{rm-I}
		If we denote $(\varPhi_{Z_0})_\ast(V)$ the matrix $U\in \mathscr{M}(m,n)$ which corresponds to the vector $u=(\phi_{z_0})_\ast(v)\in\mathbb{C}^{mn}$ under the identification \eqref{vI} and \eqref{vIa}, $\phi_{z_0}(z)$ be the biholomorphic map associated to \eqref{pI} or \eqref{pI2}. Then  by \eqref{dIZ0}, we have
		\begin{equation}
			(\varPhi_{Z_0})_\ast(V)=AV\overline{D'},\; or\; (\varPhi_{Z_0})_\ast(V)=AV'\overline{D'},\label{dIM}
		\end{equation}
		where $A$ and $D$ satisfy \eqref{pIa}.
		Thus it is safely to denote $(\varPhi_{Z_0})_\ast (V)$ instead of $(\phi_{z_0})_\ast(v)$.
		And we have
		\begin{eqnarray}
			\mathfrak{B}_l(0;(\varPhi_{Z_0})_\ast(V))&=&\mathfrak{B}_l(Z_0;V),\quad \forall V\in T_{Z_0}^{1,0}\mathfrak{R}_I,\label{Bz0}\\
			\mathcal{B}_{i,j}\big(0;(\varPhi_{Z_0})_\ast(V),(\varPhi_{Z_0})_\ast(W)\big)&=&\mathcal{B}_{i,j}(Z_0;V,W),\quad \forall V,W\in T_{Z_0}^{1,0}\mathfrak{R}_I.\label{Bz1} 	
		\end{eqnarray}
	\end{remark}

	\begin{proposition}
		$f_I^2$ is $\mbox{Iso}(\mathfrak{R}_I)$-invariant.
	\end{proposition}
	\begin{proof}
		For any $\varPhi\in\mbox{Iso}(\mathfrak{R}_I)$, by \eqref{isoI} and Remark \ref{rm-I}, we have
		$$
		\varPhi_\ast (V)=AV\overline{D'}\quad\mbox{or}\quad \varPhi_\ast (V)=AV'\overline{D'},\quad \forall V\in T_0^{1,0}\mathfrak{R}_I,
		$$
		where $A\in U(m)$ and $D\in U(n)$.
		Thus by Proposition \ref{prop-4-2},
		$$
		f_I^2(\varPhi_\ast(V))=f_I^2(V),\quad \forall V\in T_0^{1,0}\mathfrak{R}_I,
		$$
		which completes the proof.
	\end{proof}
	
	\begin{remark}For any $A\in U(m)$ and  $B\in U(n)$, it is easy to check that
		$$\mathfrak{B}_l(0;AVB)=\mathfrak{B}_l(0;V)\quad \mbox{and}\quad \mathcal{B}_{i,j}(0;V,W)=\mathcal{B}_{i,j}(0; AVB,AWB).$$
	\end{remark}
	
	\begin{theorem}\label{KB-FI}
		Let $\mathfrak{B}_l(Z;V)$ be defined  by \eqref{s12} for $Z\in \mathfrak{R}_I$ and $V\in T_Z^{1,0}\mathfrak{R}_I\cong \mathscr{M}(m,n)$. Then
		\begin{eqnarray}
			F_I^2(Z;V)&=&\frac{m+n}{1+t}\left\{\mathfrak{B}_1(Z;V)
			+t\sqrt[k]{\mathfrak{B}_k(Z;V)}\right\}\label{FI}
		\end{eqnarray}
		is a complete $\mbox{Aut}(\mathfrak{R}_I)$-invariant K\"ahler-Berwald metric for any fixed $t\in[0,+\infty),k\in\mathbb{N}$ and $k\geq 2$.
	\end{theorem}
	
	\begin{proof}Let $Z_0$ be any  fixed point in $\mathfrak{R}_{I}$,
		define
		\begin{equation}
			F_I^2(Z_0;V):=f_I^2((\varPhi_{Z_0})_\ast(V)),\quad \forall V\in T_{Z_0}^{1,0}\mathfrak{R}_I.\label{Ie}
		\end{equation}
		By Proposition \ref{prop-4-3},
		$F_I^2(Z_0;V)$ is $\mbox{Aut}(\mathfrak{R}_I)$-invariant.
		
		Substituting \eqref{Bz0} into \eqref{Ie}, we obtain
		\begin{eqnarray*}
			F_I^2(Z_0;V)=\frac{m+n}{1+t}\{\mathfrak{B}_1(Z_0;V)+t\sqrt[k]{\mathfrak{B}_k(Z_0;V)}\}.
		\end{eqnarray*}
		Since $Z_0\in\mathfrak{R}_I$ is any fixed point, changing $Z_0$ to $Z$ if necessary, we obtain \eqref{FI}. Thus by Theorem \ref{hfm},
		$F_I$ is an $\mbox{Aut}(\mathfrak{R}_I)$-invariant strongly pseudoconvex complex Finsler metric for any fixed $t\in[0,+\infty),k\in\mathbb{N}$ and $k\geq 2$.

		Next it is easy to check that
		\begin{equation*}
			\frac{\partial^2F_I^2}{\partial z_{ij}\partial\overline{v_{ab}}}\Big|_{(0;v)}=
			\frac{\partial^2F_I^2}{\partial Z_{ij}\partial\overline{V_{ab}}}\Big|_{(0;V)}=0.\label{F_zv-0-1}
		\end{equation*}
		Thus by Theorem \ref{th-2.1} and Remark \ref{RK1}, $F_I$ is a complete K\"ahler-Berwald metric on $\mathfrak{R}_I$ which completes the proof.
	\end{proof}
	
	\begin{theorem}\label{C-FI}
		Let $F_I$ be defined by \eqref{FI}, then $F_{I}$ has negative holomorphic sectional curvature and nonpositive holomorphic bisectional curvature.
		More precisely,
		the holomorphic sectional curvature $K_I$ and the holomorphic bisectional
		curvature $B_I$ of $F_I$ are given respectively by
		\begin{eqnarray}
			K_I(Z;V)&=&\frac{-4(m+n)}{1+t}\frac{\mathfrak{B}_2(Z;V)
				+t\mathfrak{B}_k^{\frac{1}{k}-1}(Z;V)\mathfrak{B}_{k+1}(Z;V)}{F_{I}^4(Z;V)},\label{hsc-F-I}\\
			B_{I}(Z; V, W)
			&=&\frac{-2(m+n)}{(1+t)F_{I}^2(Z;V)F_I^2(Z;W)}
			\Big\{\mathcal{B}_{1,1}(Z;V,W)+\mathcal{B}_{1,1}(\overline{Z'};\overline{V'},\overline{W'})\nonumber\\
			&+&t{\mathfrak{B}_{k}^{\frac{1}{k}-1}(Z;V)}
			\left[\mathcal{B}_{k,1}(Z;V,W)+\mathcal{B}_{k,1}(\overline{Z'};\overline{V'},\overline{W'})\right]\Big\},\label{hbsc-F-I}
		\end{eqnarray}
		where $\mathfrak{B}_l(Z;V)$ and $\mathcal{B}_{i,j}(Z;V)$ are defined by \eqref{s12} and \eqref{s13} respectively for $Z\in\mathfrak{R}_I$ and nonzero vectors $V,W\in T_Z^{1,0}\mathfrak{R}_I$.
	\end{theorem}
	
	\begin{proof}
		Since in this case we have $z_{ij}=Z_{ij}$ and $v_{ij}=V_{ij}$ for $i=1,\cdots,m$ and $j=1,\cdots,n$. Differentiating \eqref{s12} with respect to $z_{ab}$, we get
		\begin{eqnarray}
			\frac{\partial\mathfrak{B}_l}{\partial z_{ab}}
			&=&l\mbox{tr}\Big\{\left[(I-Z\overline{Z'})^{-1}V(I-\overline{Z'}Z)^{-1}\overline{V'}\right]^{l-1}
			(I-Z\overline{Z'})^{-1}[E_{ab}\overline{Z'}(I-Z\overline{Z'})^{-1}V\nonumber\\
			&&+V(I-\overline{Z'}Z)^{-1}\overline{Z'}E_{ab}](I-\overline{Z'}Z)^{-1}\overline{V'}\Big\},\label{A-k-a}
		\end{eqnarray}
		from which we get
		\begin{equation}
			\frac{\partial\mathfrak{B}_l}{\partial z_{ab}}(0;V)=0,\quad \frac{\partial\mathfrak{B}_l}{\partial \overline{z_{cd}}}(0;V)=0,\quad \forall l=1,2,\cdots.\label{B-l0}
		\end{equation}
		Differentiating \eqref{A-k-a} again with respect to $\overline{z_{cd}}$ and then setting $Z=0$, we get
		\begin{eqnarray}
			\frac{\partial^2\mathfrak{B}_l}{\partial z_{ab}\partial \overline{z_{cd}}}(0;V)
			&=&l\mbox{tr}\left\{\left[V\overline{V'}\right]^{l-1}
			[E_{ab}E_{dc}V+VE_{dc}E_{ab}]\overline{V'}\right\},\quad \forall l=1,2,\cdots.\label{B-ij0}
		\end{eqnarray}
		Thus by Proposition \ref{hsc}, \eqref{B-l0} and \eqref{B-ij0}, for any nonzero tangent vector $V\in T_0^{1,0}\mathfrak{R}_{II}$, we have
		$$K_I(0;V)=\frac{-2(m+n)}{1+t}\frac{1}{F_{I}^4(0;V)}\sum_{a,c=1}^{m}\sum_{b,d=1}^{n}
		\Big\{\frac{\partial^2 \mathfrak{B}_1}{\partial z_{ab}\partial \overline{z_{cd}}}(0;V)
		+\frac{t}{k}\mathfrak{B}_k^{\frac{1}{k}-1}(0;V)\frac{\partial^2\mathcal{B}_k}{\partial z_{ab}\partial \overline{z_{cd}}}(0;V)\Big\}v_{ab}\overline{v_{cd}}.
		$$
		Using \eqref{B-ij0}, it is easy to check that
		\begin{eqnarray*}
			\sum_{a,c=1}^{m}\sum_{b,d=1}^{n}
			\frac{\partial^2 \mathfrak{B}_l}{\partial z_{ab}\partial \overline{z_{cd}}}(0;V)v_{ab}\overline{v_{cd}}
			&=&2l\mathfrak{B}_{l+1}(0;V).
		\end{eqnarray*}
		Thus we have
		\begin{eqnarray}
			K_I(0;V)=-\frac{4(m+n)}{1+t}\frac{1}{F_I^4(0;V)}\left\{\mathfrak{B}_2(0;V)
			+t\mathfrak{B}_k^{\frac{1}{k}-1}(0;V)\mathfrak{B}_{k+1}(0;V)\right\}.
			\label{hsc-b1}
		\end{eqnarray}
		Since $\mathfrak{R}_I$ is transitively with respect to $\mbox{Aut}(\mathfrak{R}_I)$, and the holomorphic sectional curvature is $\mbox{Aut}(\mathfrak{R}_I)$-invariant,
		by \eqref{Bz0}, for any nonzero tangent vector $V\in T_Z^{1,0}\mathfrak{R}_I$ at $Z\in\mathfrak{R}_I$, we have \eqref{hsc-F-I}.
		Notice that $\mathfrak{B}_{l}(Z;V)>0$,  we always have $K_I(Z;V)<0$.
		
		By Proposition \ref{bhsc} and \eqref{B-ij0}, for any nonzero tangent vectors $V,W\in T_0^{1,0}\mathfrak{R}_{I}$, we have
		$$
		B_{I}(0;V,W)=\frac{-2(m+n)}{1+t}\frac{\displaystyle\sum_{a,c=1}^{m}\sum_{b,d=1}^{n}
			\left\{\frac{\partial^2 \mathfrak{B}_1}{\partial z_{ab}\partial \overline{z_{cd}}}(0;V)
			+\frac{t}{k}\mathfrak{B}_k^{\frac{1}{k}-1}\frac{\partial^2\mathfrak{B}_k}{\partial z_{ab}\partial \overline{z_{cd}}}(0;V)\right\}w_{ab}\overline{w_{cd}}}{F_I^2(0;V)F_I^2(0,W)}.
		$$
		Using \eqref{B-ij0}, it is easy to check that
		\begin{eqnarray*}
			\sum_{a,c=1}^{m}\sum_{b,d=1}^{n}\frac{\partial^2\mathfrak{B}_l}{\partial z_{ab}\partial \overline{z_{cd}}}(0;V)w_{ab}\overline{w_{cd}}
			&=&l\left[\mathcal{B}_{l,1}(0;V,W)+\mathcal{B}_{l,1}(0;\overline{V'},\overline{W'})\right].
		\end{eqnarray*}
		So that
		\begin{eqnarray*}
			B_{I}(0;V,W)
			&=&\frac{-2(m+n)}{(1+t)F_I^2(0;V)F_I^2(0;W)}\Big\{\mathcal{B}_{1,1}(0;V,W)+\mathcal{B}_{1,1}(0;\overline{V'},\overline{W'})\\
			&+&t{\mathfrak{B}_{k}^{\frac{1}{k}-1}(0;V)}
			\left[\mathcal{B}_{k,1}(0;V,W)+\mathcal{B}_{k,1}(0;\overline{V'},\overline{W'})\right]\Big\}.
		\end{eqnarray*}
		Since the holomorphic bisectional curvature is holomorphic invariant, by \eqref{Bz0} and \eqref{Bz1},		
		it follows that for any nonzero tangent vectors $V,W\in T_Z^{1,0}\mathfrak{R}_I$ we have \eqref{hbsc-F-I}.
	\end{proof}

	\subsection{$\mbox{Aut}(\mathfrak{R}_{II})$-invariant K\"ahler-Berwald metrics on $\mathfrak{R}_{II}$}\label{subsection-3.3}
	
	In this subsection, we shall construct $\mbox{Aut}(\mathfrak{R}_{II})$-invariant strongly pseudoconvex complex Finsler metrics on $\mathfrak{R}_{II}$. Note that $\mathfrak{R}_{II}$ is a complex submanifold of $\mathfrak{R}_I$ and $T_0^{1,0}\mathfrak{R}_{II}$ is a complex subspace of $T_0^{1,0}\mathfrak{R}_I$. Intuitively, the restriction of the complex Minkowski norm $f_I^2$ to the holomorphic tangent space $T_0^{1,0}\mathfrak{R}_{II}$ leads to a strongly pseudoconvex complex Minkowski norm on $T_0^{1,0}\mathfrak{R}_{II}$, which will be denote by $f_{II}^2$ in the following.

	Note that a point $Z=(Z_{ij})\in\mathfrak{R}_{II}(p)$ is
	identified with a point
	\begin{equation}
		z=(z_{11},\cdots,z_{1p},z_{22},\cdots,z_{2p},\cdots,z_{p-1\,p},z_{pp})\in\mathbb{C}^{\frac{p(p+1)}{2}}\label{c-II}
	\end{equation}
	by setting
	\begin{equation}
		Z_{ji}=Z_{ij}:=\frac{z_{ij}}{\sqrt{2}p_{ij}},\quad p_{ij}=\left\{
		\begin{array}{ll}
			\frac{1}{\sqrt{2}} & \hbox{for}\,i=j, \\
			1, & \hbox{for}\,i<j
		\end{array}
		\right.\quad\mbox{and},\quad \forall 1\leq i\leq j\leq p.\label{c-IIa}
	\end{equation}
	
	A tangent vector
	\begin{equation}
		v=(v_{11},\cdots,v_{1p},v_{22},\cdots,v_{2p},\cdots,v_{p-1\,p},v_{pp})\in\mathbb{C}^{\frac{p(p+1)}{2}}\label{v-II}
	\end{equation}
	at $z_0$ (which corresponds to $Z_0\in\mathfrak{R}_{II}$) can be identified with a symmetric matrix
	\begin{equation}
		V=(V_{ij})\in\mathscr{M}(p)\quad \mbox{with}\quad V_{ji}=V_{ij}:=\frac{v_{ij}}{\sqrt{2}p_{ij}},\quad \forall 1\leq i\leq j\leq p,\label{v-IIa}
	\end{equation}
	where  $p_{ij}$  are defined in \eqref{c-IIa}. Thus $T_0^{1,0}\mathfrak{R}_{II}$ may be identified with a complex subspace of $\mathscr{M}(p)$, which consists of matrices $V$ satisfying \eqref{v-IIa}.
	
	Since $\mathfrak{R}_{II}$ is homogeneous, for any fixed point $Z_0\in\mathfrak{R}_{II}$, there exists a map $\varPhi_{Z_0}\in\mbox{Aut}(\mathfrak{R}_{II})$ such that $\varPhi_{Z_0}(Z_0)=0$. Indeed, (cf. \cite{Hua}, \cite{Look3})
	\begin{equation}
		W=\varPhi_{Z_0}(Z)=A(Z-Z_0)(I_p-\overline{Z_0}Z)^{-1}{\overline{A}}^{-1}\,\label{p-II}
	\end{equation}
	where$A\in\mathscr{M}(p)$ satisfies
	\begin{equation}
		\overline{A'}A=(I_p-Z_0\overline{Z_0})^{-1}.\label{p-IIa}
	\end{equation}
	Note that these mappings give the $\mbox{Aut}(\mathfrak{R}_{II})$ (\cite{Siegel1}).

	\begin{definition}Let $t\in[0,+\infty),k\in\mathbb{N}$ and $k\geq 2$. Define
		\begin{equation}
			f_{II}^2(V):=(p+1)\pmb{f}^2(V),\quad \forall V\in T_0^{1,0}\mathfrak{R}_{II}\cong\{V=V'|V\in\mathscr{M}(p)\},\label{f-II}
		\end{equation}
		where $\pmb{f}^2(V)$ is given by \eqref{bf}.
	\end{definition}
	
	Let $A=(A_{ab})\in\mathscr{M}(p)$, the symmetric tensor product of $A$ is a matrix $A\otimes_s A=(C_{(ab)(ij)})\in\mathscr{M}\left(\frac{p(p+1)}{2}\right)$
	whose elements $C_{(ab)(ij)}$ at row $(ab) (a\leq b)$ and column $(ij) (i\leq j)$ are given by \cite{Look3}
	\begin{equation}
		C_{(ab)(ij)}:=p_{ab}p_{ij}(A_{ai}A_{bj}+A_{aj}A_{bi}),\quad a\leq b\;\mbox{and}\; i\leq j.
	\end{equation}
	
	The following proposition is due to Q. K. Lu (cf. \cite{Look3}, page 314-316) whenever $z_0=0$.
	\begin{proposition}
		Let $w=\phi_{z_0}(z)$ be the biholomorphic map associated to \eqref{p-II}, where $z_0,z$ and $w$ are the corresponding complex coordinates of $Z_0,Z$ and $W$ in \eqref{p-II}, respectively. Then
		\begin{eqnarray}
			(\phi_{z_0})_\ast (v)=v(A'\otimes_s A'),\label{d-IIZ0}
		\end{eqnarray}
		where $(\phi_{z_0})_\ast$ denotes the differential of $\phi_{z_0}$ at $z_0$, $v$ is given by \eqref{v-II} and $A$ satisfies \eqref{p-IIa}.
	\end{proposition}

	\begin{proof}
		Differentiating \eqref{p-II}, and then setting $Z=Z_0$, we have
		\begin{equation}
			dW=AdZA',\label{d-IIa}
		\end{equation}
		where $A=(A_{ab})$ and $D=(D_{ij})$ satisfy \eqref{p-IIa}. Substituting \eqref{c-IIa} into \eqref{d-IIa}, we have
		\begin{eqnarray*}
			dw_{ab}&=&\sqrt{2}p_{ab}\sum_{r,s=1}^nA_{ar}\frac{dz_{rs}}{\sqrt{2}p_{rs}}A_{bs}\\
			&=&p_{ab}\left[\sum_{r<s}dz_{rs}(A_{ar}A_{bs}+A_{as}A_{br})+\sum_{r=s}\sqrt{2}dz_{rs}A_{ar}A_{bs}\right]\\
			&=&\sum_{r\leq s}dz_{rs}p_{ab}p_{rs}(A_{ar}A_{bs}+A_{as}A_{br}),
		\end{eqnarray*}
		from which it follows that
		$$
		\frac{\partial w_{ab}}{\partial z_{rs}}\Big|_{z=z_0}=p_{ab}p_{rs}(A_{ar}A_{bs}+A_{as}A_{br}).
		$$
		Thus  if we setting $(\phi_{z_0})_\ast(v)=u=(u_{11},u_{12},\cdots,u_{1n},u_{22},u_{23},\cdots,u_{2n},\cdots,u_{n-1\,n},u_{nn})\in\mathbb{C}^{\frac{n(n+1)}{2}}$, we have
		\begin{equation}
			u_{ab}=\sum_{r\leq s}v_{rs}\frac{\partial w_{ab}}{\partial z_{rs}}
			=\sum_{r\leq s}v_{rs}p_{ab}p_{rs}(A_{ar}A_{bs}+A_{as}A_{br})=[v(A'\otimes_sA')]_{ab}.\label{t-II}
		\end{equation}
	\end{proof}
	
	\begin{remark}
		If we denote $(\varPhi_{Z_0})_\ast(V)$ the symmetric matrix in $\mathscr{M}(p)$ which corresponds to the vector
		$(\phi_{z_0})_\ast(v)\in\mathbb{C}^{\frac{p(p+1)}{2}}$ under the identification \eqref{v-II} and \eqref{v-IIa}. Then we have
		\begin{equation}
			(\varPhi_{Z_0})_\ast(V)=AVA',\label{d-IM}
		\end{equation}
		where $A$ satisfies \eqref{p-IIa}.
	\end{remark}
	
	\begin{proposition}
		$f_{II}^2$ is an $\mbox{Iso}(\mathfrak{R}_{II})$-invariant complex Minkowski norm for any fixed $t\in[0,+\infty)$ and $k\in\mathbb{N},k\geq 2$.
	\end{proposition}
	
	\begin{proof}
		Taking $Z_0=0$ in \eqref{p-II}, we have
		\begin{equation}
			\varPhi_0(Z)=AZA',\quad \forall Z\in\mathfrak{R}_{II},\label{II-c}
		\end{equation}
		with $A\in U(p)$. Denote $w=\phi_0(z)$ the complex linear map determined by \eqref{II-c}, which is expressed in terms of complex coordinates $z$
		and $w$ corresponding $Z$ and $W$ in $\mathfrak{R}_{II}$,
		then its Jacobian matrix at $z_0$ (which corresponds to $Z_0)$ is given by (cf. (20) on page 316 in \cite{Look3})
		\begin{equation}
			\frac{\partial w}{\partial z}=A'\otimes_s A'.
		\end{equation}
		Thus for a tangent vector $v=(v_{11},v_{12},\cdots,v_{1p},v_{22},v_{23},\cdots,v_{2p},\cdots,v_{pp})$ at the origin, we have
		\begin{equation}
			(\phi_0)_\ast (v)=v(A'\otimes_s A'),\label{JII-0}
		\end{equation}
		where $(\phi_0)_\ast$ denotes the differential of $\phi_0$ at the origin. Identifying $(\phi_0)_\ast(v)$ with a symmetric matrix,
		denoted by $(\varPhi_{Z_0})_\ast(V)\in \mathscr{M}(p)$, then we have
		\begin{equation}
			(\varPhi_0)_\ast (V)=AVA',\quad \forall V\in T_0^{1,0}\mathfrak{R}_{II}.
		\end{equation}
		It is obvious that
		$$
		f_{II}^2((\varPhi_0)_\ast(V))=f_{II}^2(V)
		$$
		for any $A\in U(p)$. Thus $f_{II}$ is $\mbox{Iso}(\mathfrak{R}_{II})$-invariant.

		Denote $w$ the complex coordinates corresponding to $W$ defined by \eqref{II-c}. Then for
		$$u=(u_{11},\cdots,u_{1p},u_{22},\cdots,u_{2p},\cdots,u_{p-1\,p-1},u_{p-1\,p},u_{pp})\in\mathbb{C}^{\frac{p(p+1)}{2}},$$
		we have
		\begin{eqnarray*}
			\sum_{i\leq j}\sum_{r\leq s}\frac{\partial^2f_{II}^2}{\partial v_{ij}\partial \overline{v_{rs}}}u_{ij}\overline{u_{rs}}
			&=&\sum_{i=j}\sum_{r=s}\frac{\partial^2f_{II}^2}{\partial v_{ii}\partial\overline{v_{rr}}}u_{ii}\overline{u_{rr}}
			+\sum_{i=j}\sum_{r<s}\frac{\partial^2f_{II}^2}{\partial v_{ii}\partial\overline{v_{rs}}}u_{ii}\overline{u_{rs}}\\
			&&+\sum_{i<j}\sum_{r=s}\frac{\partial^2f_{II}^2}{\partial v_{ij}\partial\overline{v_{rr}}}u_{ij}\overline{u_{rr}}
			+\sum_{i<j}\sum_{r<s}\frac{\partial^2f_{II}^2}{\partial v_{ij}\partial\overline{v_{rs}}}u_{ij}\overline{u_{rs}}.
		\end{eqnarray*}
		Taking $M=(M_{ij})$ with $M_{ij}=M_{ji}$ for $i\leq j$ and
		$$M_{ij}=\frac{u_{ij}}{\sqrt{2}p_{ij}}\quad\mbox{with}\quad p_{ij}=\left\{
		\begin{array}{ll}
			\frac{1}{\sqrt{2}}, & \hbox{for}\; i=j \\
			1, & \hbox{for}\; i<j
		\end{array}
		\right.
		$$
		then it is easy to check that
		\begin{eqnarray*}
			\sum_{i\leq j}\sum_{r\leq s}\frac{\partial^2f_{II}^2}{\partial v_{ij}\partial \overline{v_{rs}}}u_{ij}\overline{u_{rs}}
			&=&(p+1)\sum_{i,j=1}^p\sum_{r,s=1}^p
			\frac{\partial^2\pmb{f}^2}{\partial V_{ij}\partial\overline{V_{rs}}}M_{ij}\overline{M_{rs}}\geq 0
		\end{eqnarray*}
		for any  nonzero tangent vector $V\cong \mathbb{C}^{\frac{p(p+1)}{2}}$ and $M\in\mathscr{M}(p)\cong \mathbb{C}^{\frac{p(p+1)}{2}}$.
		Thus $f_{II}$ is also a complex Minkowski norm on $T_0^{1,0}\mathfrak{R}_{II}$ for any fixed $t\in[0,+\infty)$ and $k\in\mathbb{N},k\geq 2$.
	\end{proof}
	
	\begin{theorem}\label{KB-FII}
		Let $\mathfrak{B}_l(Z;V)$ be defined by \eqref{s12} for $Z\in \mathfrak{R}_{II}$ and vectors $V\in T_Z^{1,0}\mathfrak{R}_{II}$.
		Then
		\begin{eqnarray}
			F_{II}^2(Z;V)&=&\frac{p+1}{1+t}\left\{\mathfrak{B}_1(Z;V)
			+t\sqrt[k]{\mathfrak{B}_k(Z;V)}\right\}\label{FII}
		\end{eqnarray}
		is a complete $\mbox{Aut}(\mathfrak{R}_{II})$-invariant K\"ahler-Berwald metric for any fixed $t\in[0,+\infty),k\in\mathbb{N}$ and $k\geq 2$.
	\end{theorem}
	
	\begin{proof}
		Let $Z_0$ be any fixed point in $\mathfrak{R}_{II}$ which corresponds to the point $z_0\in \mathbb{C}^{\frac{p(p+1)}{2}}$ as in \eqref{c-II} and \eqref{c-IIa}. Let
		$v\in \mathbb{C}^{\frac{p(p+1)}{2}}$ be a tangent vector at $z_0$, which corresponds to $V=(V_{ij})$ as in \eqref{v-II} and \eqref{v-IIa}.
		Let $\varPhi_{Z_0}\in\mbox{Aut}(\mathfrak{R}_{II})$ be given by \eqref{p-II}.
		Denote $U=(U_{ab})$  a symmetric matrix with $U_{ab}=\frac{u_{ab}}{\sqrt{2}p_{ab}}(a\leq b)$ and
		\begin{equation}
			u_{ab}:=\sum_{i\leq j}\frac{\partial w_{ab}}{\partial z_{ij}}v_{ij}, \quad 1\leq a\leq b\leq p.
		\end{equation}
		Then $U$ can be identified with the differential of \eqref{p-II} at the point $z_0$.
		By \eqref{p-IIa}, we have
		$U=(\varPhi_{Z_0})_\ast(V)$ and if we define
		$$F_{II}^2(Z_0;V):=f_{II}^2((\varPhi_{Z_0})_\ast (V)),\quad \forall V\in T_{Z_0}^{1,0}\mathfrak{R}_{II},$$
		where $(\varPhi_{Z_0})_\ast$ denotes the differential of \eqref{p-II} at the point $Z_0$. Then it is easy to check that
		\begin{eqnarray*}
			F_{II}^2(Z_0;V)&=&\frac{p+1}{1+t}\left\{\mathfrak{B}_1(Z_0;V)
			+t\sqrt[k]{\mathfrak{B}_k(Z_0;V)}\right\}.
		\end{eqnarray*}
		Since $Z_0$ is any fixed point in $\mathfrak{R}_{II}$, changing $Z_0$ to $Z$ if necessary and by Theorem \ref{hfm}, it follows that $F_{II}$ is an $\mbox{Aut}(\mathfrak{R}_{II})$-invariant strongly pseudoconvex complex Finsler metric for any fixed $t\in[0,+\infty),k\in\mathbb{N}$ and $k\geq 2$.
		
		Next at the point $(0;v)$ which corresponding to $(0;V)$, we have
		\begin{eqnarray}
			\frac{\partial^2F_{II}^2}{\partial z_{rs}\partial \overline{v_{ij}}}
			=\frac{1}{2p_{rs}p_{ij}}\frac{\partial^2F_{II}^2}{\partial Z_{rs}\partial\overline{V_{ij}}}=0.\label{F_zv-2}
		\end{eqnarray}
		Thus by Theorem \ref{th-2.1} and Remark \ref{RK1}, $F_{II}$ is a complete K\"ahler-Berwald metric for any fixed $t\in[0,+\infty),k\in\mathbb{N}$ and $k\geq 2$.
	\end{proof}
	
	\begin{theorem}\label{C-FII}
		Let $F_{II}(Z;V)$ be defined by \eqref{FII}. Then $F_{II}$ has negative holomorphic sectional curvature and nonpositive holomorphic bisectional
		curvature. More precisely, the holomorphic sectional curvature $K_{II}$ and the holomorphic bisectional
		curvature $B_{II}$ of $F_{II}$ are given respectively by
		\begin{equation}
			K_{II}(Z;V)=-\frac{4(p+1)}{1+t}\frac{\mathfrak{B}_2(Z;V)
				+t\mathfrak{B}_k^{\frac{1}{k}-1}(Z;V)\mathfrak{B}_{k+1}(Z;V)}{F_{II}^4(Z;V)}\label{hsc-F-II}
		\end{equation}
		and
		\begin{eqnarray}
			B_{II}(Z; V, W)
			&=&-\frac{4(p+1)}{1+t}\frac{\mathcal{B}_{1,1}(Z;V,W)
				+t{\mathfrak{B}_{k}^{\frac{1}{k}-1}(Z;V)}
				\mathcal{B}_{k,1}(Z;V,W)}{F_{II}^2(Z;V)F_{II}^2(Z;V)},\label{hbsc-F-II}
		\end{eqnarray}
		where $\mathfrak{B}_l(Z;V)$ and $\mathcal{B}_{i,j}(Z;V)$ are defined by \eqref{s12} and \eqref{s13} respectively for $Z\in \mathfrak{R}_{II}$ and nonzero vectors
		$V,W\in T_Z^{1,0}\mathfrak{R}_{II}$.
	\end{theorem}
	
	\begin{proof}Recall that for any $Z\in\mathfrak{R}_{II}$ and $V\in T_Z^{1,0}\mathfrak{R}_{II}$, one has $ Z=Z'$, $V=V'$,
		and$Z_{ij}=\frac{z_{ij}}{\sqrt{2}p_{ij}}, V_{ij}=\frac{v_{ij}}{\sqrt{2}p_{ij}} $ for $i\leq j$.
		Since
		\begin{eqnarray}
			v_{ab}\frac{\partial}{\partial z_{ab}}=V_{ab}\frac{\partial}{\partial Z_{ab}},\quad
			v_{ab}\overline{v_{cd}}\frac{\partial^2}{\partial z_{ab}\partial \overline{z_{cd}}}
			=V_{ab}\overline{V_{cd}}\frac{\partial^2}{\partial Z_{ab}\partial \overline{Z_{cd}}},
			\label{v-V}
		\end{eqnarray}
		thus for $l=1,2,\cdots,$ we have
		\begin{eqnarray}
			\frac{\partial\mathfrak{B}_l}{\partial Z_{ab}}
			&=&lp_{ab}^2\mbox{tr}\Big\{\left[(I-Z\overline{Z})^{-1}V(I-\overline{Z}Z)^{-1}\overline{V}\right]^{l-1}
			(I-Z\overline{Z})^{-1}[(E_{ab}+E_{ba})\overline{Z}(I-Z\overline{Z})^{-1}V\nonumber\\
			&&+V(I-\overline{Z}Z)^{-1}\overline{Z}(E_{ab}+E_{ba})](I-\overline{Z}Z)^{-1}\overline{V}\Big\},\label{A2-k-a}
		\end{eqnarray}
		from which we get
		\begin{equation}
			\frac{\partial\mathfrak{B}_l}{\partial Z_{ab}}\Big|_{(0;V)}=0,\quad \frac{\partial\mathfrak{B}_l}{\partial \overline{Z_{cd}}}\Big|_{(0;V)}=0.
		\end{equation}
		Differentiating \eqref{A2-k-a} with respect to $\overline{Z_{cd}}$ and then setting $Z=0$, we get
		$$
		\frac{\partial^2\mathfrak{B}_l}{\partial Z_{ab}\partial \overline{Z_{cd}}}(0;V)
		=lp_{ab}^2p_{cd}^2 \mbox{tr}\Big\{\left[V\overline{V}\right]^{l-1}
		\left[(E_{ab}+E_{ba})(E_{dc}+E_{cd})V+V(E_{dc}+E_{cd})(E_{ab}+E_{ba})\right]\overline{V}\Big\}.\label{A2-k-ab}
		$$
		Now it is easy to check that
		\begin{eqnarray}
			\sum_{a\leq b}\sum_{c\leq d}\frac{\partial^2\mathfrak{B}_l}{\partial Z_{ab}\partial \overline{Z_{cd}}}(0;V)V_{ab}\overline{V_{cd}}&=&2l\mathfrak{B}_{l+1}(0;V).
		\end{eqnarray}
		Thus by Proposition \ref{hsc},  for any nonzero tangent vector $V\in T_0^{1,0}\mathfrak{R}_{II}$, we have
		\begin{eqnarray}
			K_{II}(0;V)&=&-\frac{4(p+1)}{1+t}\frac{\mathfrak{B}_2(0;V)
				+t\mathfrak{B}_k^{\frac{1}{k}-1}(0;V)\mathfrak{B}_{k+1}(0;V)}{F_{II}^4(0;V)}.\label{hsc2-b1}
		\end{eqnarray}
		Since the holomorphic sectional curvature is holomorphic-invariant, and by \eqref{Bz0}, for any nonzero tangent vector $V\in T_Z^{1,0}\mathfrak{R}_{II}$ at $Z\in\mathfrak{R}_{II}$ we have
		\eqref{hsc-F-II}.
		Notice $\mathfrak{B}_{l}(Z;V)>0,$ we always have $K_{II}(Z;V)<0$.
		
		Similarly, for any nonzero tangent vectors $V,W\in T_0^{1,0}\mathfrak{R}_{II}$ we have
		\begin{eqnarray*}
			\sum_{a\leq b}\sum_{c\leq d}\frac{\partial^2\mathfrak{B}_l}{\partial Z_{ab}\partial \overline{Z_{cd}}}(0;V)W_{ab}\overline{W_{cd}}&=&2l\mathcal{B}_{l,1}(0;V,W).
		\end{eqnarray*}
		
		Thus by Proposition \ref{bhsc}, we have
		\begin{eqnarray*}
			B_{II}(0;V,W) &=&-\frac{4(p+1)}{1+t}\frac{\mathcal{B}_{1,1}(0;V,W)+t{\mathfrak{B}_{k}^{\frac{1}{k}-1}(0;V)}\mathcal{B}_{k,1}(0;V,W)}{F_{II}^2(0;V)F_{II}^2(0;W)}.
		\end{eqnarray*}
		Since the holomorphic bisectional curvature is holomorphic-invariant, by \eqref{Bz0} and  \eqref{Bz1}
		we have \eqref{hbsc-F-II}
		for any nonzero tangent vectors $V,W\in T_Z^{1,0}\mathfrak{R}_{II}$ and $Z\in\mathfrak{R}_{II}$.
		Since $ \mathcal{B}_{i,j}(Z;V,W)\geq0, $ we have
		$ B_{II}(Z; V, W)\leq0, $
		with equality holds if and only if
		$$\overline{W'}(I_p-Z\overline{Z'})^{-1}V=0.$$
	\end{proof}

	\subsection{$\mbox{Aut}(\mathfrak{R}_{III})$-invariant K\"ahler-Berwald metrics on $\mathfrak{R}_{III}$}\label{subsection-3.4}
	
	In this subsection, we shall construct $\mbox{Aut}(\mathfrak{R}_{III})$-invariant strongly pseudoconvex complex Finsler metrics on $\mathfrak{R}_{III}$. Note that $\mathfrak{R}_{III}$ is also a complex submanifold of $\mathfrak{R}_I$, hence $T_0^{1,0}\mathfrak{R}_{III}$ is a complex subspace of $T_0^{1,0}\mathfrak{R}_I$ too, that is, $T_0^{1,0}\mathfrak{R}_{III}$ consists of matrices $V\in\mathscr{M}(q)$ which are skew symmetric. Intuitively, the restriction of the complex Minkowski norm $f_I^2$ to the holomorphic tangent space $T_0^{1,0}\mathfrak{R}_{III}$ leads to a strongly pseudoconvex complex Minkowski norm on $T_0^{1,0}\mathfrak{R}_{III}$, which will be denote by $f_{III}^2$ in the following.

	Note that a point $Z=(Z_{ij})\in\mathfrak{R}_{III}(q)$ is identified with a point
	\begin{equation}
		z=(z_{12},\cdots,z_{1n},z_{23},\cdots,z_{2q},\cdots,z_{q-1\,q})\in\mathbb{C}^{\frac{q(q-1)}{2}}\label{c-III}
	\end{equation}
	by setting
	\begin{equation}
		Z_{ji}=-Z_{ij}\quad \mbox{and}\quad Z_{ij}:=z_{ij}\quad\forall 1\leq i<j\leq q.\label{c-IIIa}
	\end{equation}
	
	Similarly, a tangent vector
	\begin{equation}
		v=(v_{12},\cdots,v_{1n},v_{23},\cdots,v_{2q},\cdots,v_{q-1\,q})\in\mathbb{C}^{\frac{q(q-1)}{2}}\label{v-III}
	\end{equation}
	at the point $z_0\in \mathbb{C}^{\frac{q(q-1)}{2}}$ (which corresponds to $Z_0\in\mathfrak{R}_{III}$)
	is identified with a skew symmetric matrix $V=(V_{ij})$ with
	\begin{equation}
		V_{ij}:=v_{ij},\quad \forall 1\leq i<j\leq q.\label{v-IIIa}
	\end{equation}
	
	Since $\mathfrak{R}_{III}$ is homogeneous, for any fixed point $Z_0\in\mathfrak{R}_{III}$, there exists a map $\varPhi_{Z_0}\in\mbox{Aut}(\mathfrak{R}_{III})$ such that $\varPhi_{Z_0}(Z_0)=0$. Indeed, (cf. \cite{Hua}, \cite{Look3})
	\begin{equation}
		W=\varPhi_{Z_0}(Z)=A(Z-Z_0)(I_q+\overline{Z_0}Z)^{-1}{\overline{A}}^{-1},\label{p-III-a}
	\end{equation}
	where $A\in\mathscr{M}(q)$ satisfies
	\begin{equation}
		\overline{A'}A=(I_q+Z_0\overline{Z_0})^{-1}.\label{p-IIIa-a}
	\end{equation}
	\begin{remark}\label{rm-III}
		For $q=4$, if we denote
		$$
		Z=\begin{pmatrix}
			0 & Z_{12} & Z_{13} & Z_{14} \\
			-Z_{12} & 0 & Z_{23} & Z_{24} \\
			-Z_{13} & -Z_{23} & 0 & Z_{34} \\
			-Z_{14} & -Z_{24} & -Z_{34} & 0 \\
		\end{pmatrix}\in\mathfrak{R}_{III},
		$$
		then
		\begin{equation}
			W=\widetilde{Z}=\begin{pmatrix}
				0 & Z_{12} & Z_{13} & Z_{23} \\
				-Z_{12} & 0 & Z_{14} & Z_{24} \\
				-Z_{13} & -Z_{14} & 0 & Z_{34} \\
				-Z_{23} & -Z_{24} & -Z_{34} & 0 \\
			\end{pmatrix}\label{rm-III2}
		\end{equation}
		also belongs to $\mbox{Aut}(\mathfrak{R}_{III})$ which cannot be expressed in the form of \eqref{p-III-a} (cf. \cite{Klingen3 , Yin}).
	\end{remark}
	
	\begin{definition}Let $t\in[0,+\infty),k\in\mathbb{N}$ and $k\geq 2$. Define
		\begin{equation}
			f_{III}^2(V):=(q-1)\pmb{f}^2(V),\quad \forall V\in T_0^{1,0}\mathfrak{R}_{III}\cong\{V=-V'|V\in\mathscr{M}(q)\},\label{f-III}
		\end{equation}
		where $\pmb{f}^2(V)$ is given by \eqref{bf}.
	\end{definition}

	Let $A\in\mathscr{M}(q)$, the skew symmetric tensor product of $A=(A_{ab})$ is a matrix $A\otimes_{sk} A=(C_{(ab)(ij)})\in\mathscr{M}\left(\frac{q(q-1}{2}\right)$
	whose elements $C_{(ab)(ij)}$ at row $(ab)(a<b)$ and column $(ij)(i<j)$ are given by \cite{Look3}
	\begin{equation}
		C_{(ab)(ij)}:=A_{ai}A_{bj}-A_{aj}A_{bi},\quad a<b\;\mbox{and}\;i<j.
	\end{equation}
	
	The following proposition is due to Q. K. Lu \cite{Look3} (see p. 317) whenever $z_0=0$.
	\begin{proposition}
		Let $w=\phi_{z_0}(z)$ be the biholomorphic map associated to \eqref{p-III-a}, where $z_0,z$ and $w$ are the corresponding complex coordinates of $Z_0,Z$ and $W$ in \eqref{p-III-a}. Then
		\begin{eqnarray}
			(\phi_{z_0})_\ast (v)=v(A'\otimes_{sk} A'),\label{d-IIIZ0}
		\end{eqnarray}
		where $A$ satisfies \eqref{p-IIIa-a},  $(\phi_{z_0})_\ast$ denotes the differential of $\phi_{z_0}$ at the point $z_0$ and $v$ is given by \eqref{v-III}.
	\end{proposition}
	
	\begin{proof}
		Differentiating \eqref{p-III-a}, and then setting $Z=Z_0$, we have
		\begin{equation}
			dW=AdZA',\label{d-IIIa}
		\end{equation}
		where $A$ and $D$ satisfy \eqref{p-IIIa-a}. Substituting \eqref{c-IIIa} into \eqref{d-IIIa}, we have
		\begin{eqnarray*}
			dw_{ab}&=&\sum_{r,s=1}^nA_{ar}dz_{rs}A_{bs}=\sum_{r<s}dz_{rs}(A_{ar}A_{bs}-A_{as}A_{br}),
		\end{eqnarray*}
		from which it follows that
		$$
		\frac{\partial w_{ab}}{\partial z_{rs}}\Big|_{z=z_0}=A_{ar}A_{bs}-A_{as}A_{br}.
		$$
		Thus  if we setting 
		\begin{equation}
			u_{ab}=\sum_{r<s}v_{rs}\frac{\partial w_{ab}}{\partial z_{rs}}
			=\sum_{r<s}v_{rs}(A_{ar}A_{bs}-A_{as}A_{br})=[v(A'\otimes_{sk}A')]_{ab}.
		\end{equation}
	\end{proof}
	
	\begin{remark}
		If we denote $(\varPhi_{Z_0})_\ast(V)$ the skew-symmetric matrix in $\mathscr{M}(q)$ which corresponds to the vector
		$(\phi_{z_0})_\ast(v)\in\mathbb{C}^{\frac{q(q-1)}{2}}$ under the identification \eqref{v-III} and \eqref{v-IIIa}. Then we have
		\begin{equation}
			(\varPhi_{Z_0})_\ast(V)=AVA',\label{d-IIIM}
		\end{equation}
		where $A\in\mathscr{M}(q)$  satisfies \eqref{p-IIIa-a}.
	\end{remark}
	
	\begin{proposition}\label{p-III}
		$f_{III}^2$ is an $\mbox{Iso}(\mathfrak{R}_{III})$-invariant strongly pseudoconvex complex Minkowski
		norm for any fixed $t\in[0,+\infty),k\in \mathbb{N}$ and $k\geq 2$.
	\end{proposition}
	
	\begin{proof}
		Taking $Z_0=0$ in \eqref{p-III}, we have
		\begin{equation}
			W=\varPhi_0(Z)=AZA',\quad \forall Z\in\mathfrak{R}_{III},\label{III-c}
		\end{equation}
		where $A\in U(q)$. Let $z$ and $w$ be the complex coordinates corresponding to $Z$ and $W$ in \eqref{III-c}, and denote
		$w=\phi_0(z)$ the complex linear map corresponding to \eqref{III-c}.
		Then the Jacobian matrix of $w=\phi_0(z)$ at $z=0$ is given by (cf. (24) on page 318 in \cite{Look3})
		\begin{equation}
			\frac{\partial w}{\partial z}=A'\otimes_{sk} A',\quad A\in U(q).
		\end{equation}
		Thus for $v\in \mathbb{C}^{\frac{q(q-1)}{2}}$ given by \eqref{v-III}, we have
		\begin{equation}
			(\phi_0)_\ast (v)=v(A'\otimes_{sk} A'),\label{JIII-0}
		\end{equation}
		where $(\phi_0)_\ast$ denotes the differential of $\phi_0$ at the origin. Identifying $(\phi_0)_\ast(v)\in \mathbb{C}^{\frac{q(q-1)}{2}}$ with
		$(\varPhi_0)_\ast(V)\in\mathscr{M}(q)$, we have
		\begin{equation}
			(\varPhi_0)_\ast (V)=AVA',\quad \forall A\in U(q), V\in T_0^{1,0}\mathfrak{R}_{III}.
		\end{equation}
		It is easy to check that
		$$
		f_{III}^2((\varPhi_0)_\ast(V))=f_{III}^2(V).
		$$
		
		Notice that for the mapping $W=\widetilde{Z}$ in \eqref{rm-III2}, $ W\overline{W} $ and $ Z\overline{Z} $ share the same eigenvalues as pointed out by H. Klingen \cite{Klingen3}, and the complex norm $ f_{III} $ is a symmetric function with respect to the eigenvalues of the vector matrices (see Remark \ref{3.2}). Thus $ f_{III} $ is invariant under the actions of $W=\widetilde{Z},$
		namely $f_{III}$ is $\mbox{Iso}(\mathfrak{R}_{III})$-invariant.

		Denote $w$ the complex coordinates corresponding to $W$ defined by \eqref{III-c}.
		Taking $M=(M_{ij})$ with $M_{ij}=M_{ji}$ for $i\leq j$ and
		$M_{ij}=u_{ij}$ for $i<j$.
		Then for
		$$u=(u_{12},\cdots,u_{1q},u_{23},\cdots,u_{2q},\cdots,u_{q-1\,q})\in\mathbb{C}^{\frac{q(q-1)}{2}},$$
		we have
		\begin{eqnarray*}
			\sum_{i<j}\sum_{r<s}\frac{\partial^2f_{III}^2}{\partial v_{ij}\partial \overline{v_{rs}}}u_{ij}\overline{u_{rs}}
			&=&\sum_{i<j}\sum_{r<s}
			\left(\frac{\partial^2f_{III}^2}{\partial V_{ij}\partial\overline{V_{rs}}}
			+\frac{\partial^2f_{III}^2}{\partial V_{ij}\partial\overline{V_{sr}}}
			+\frac{\partial^2f_{III}^2}{\partial V_{ji}\partial\overline{V_{rs}}}
			+\frac{\partial^2f_{III}^2}{\partial V_{ji}\partial\overline{V_{sr}}}
			\right)M_{ij}\overline{M_{rs}}\\
			&=&(q-1)\sum_{i,j=1}^q\sum_{r,s=1}^q
			\frac{\partial^2\pmb{f}^2}{\partial V_{ij}\partial\overline{V_{rs}}}M_{ij}\overline{M_{rs}}\geq 0
		\end{eqnarray*}
		for any  nonzero tangent vector $v$ and $u\in \mathbb{C}^{\frac{q(q-1)}{2}}$.
		Thus $f_{III}$ is also a complex Minkowski norm on $T_0^{1,0}\mathfrak{R}_{III}$.
	\end{proof}

	\begin{theorem}\label{KB-FIII}
		Let $\mathfrak{B}_l(Z;V)$ be defined by \eqref{s12} for $Z\in\mathfrak{R}_{III}$ and $V\in T_Z^{1,0}\mathfrak{R}_{III}$.
		Then
		\begin{eqnarray}
			F_{III}^2(Z;V)&=&\frac{q-1}{1+t}\left\{\mathfrak{B}_1(Z;V)
			+t\sqrt[k]{\mathfrak{B}_k(Z;V)}\right\}\label{FIII}
		\end{eqnarray}
		is a complete $\mbox{Aut}(\mathfrak{R}_{III})$-invariant K\"ahler-Berwald metric for any fixed $t\in[0,+\infty),k\in\mathbb{N}$ and
		$k\geq 2$.
	\end{theorem}
	
	\begin{proof}
		Let $v$ be a tangent vector at $z_0$ which corresponds to $Z_0\in \mathfrak{R}_{III}$.
		Let $V=(V_{ij})$ be a skew symmetric matrix with
		$V_{ij}=-V_{ji}=v_{ij}$ for $i<j$. Then the differential $(\phi_{z_0})_\ast$ maps $v$ to a vector $u=(u_{ab})\in\mathbb{C}^{\frac{q(q-1)}{2}}$
		with
		\begin{equation}
			u_{ab}=\sum_{i<j}\frac{\partial w_{ab}}{\partial z_{ij}}v_{ij},\quad a<b.\label{d-III}
		\end{equation}
		
		If we denote $U=(U_{ab})\in\mathscr{M}(q)$ a skew symmetric matrix with
		$$
		U_{ab}=u_{ab},\quad \forall 1\leq a<b\leq q.
		$$
		Then we can rewrite \eqref{d-III}
		as $U=(\varPhi_{Z_0})_\ast(V)$ and it is easy to check that
		\begin{equation}
			(\varPhi_{Z_0})_\ast(V)=AVA',
		\end{equation}
		where $A$ satisfies \eqref{p-IIIa-a}.
		Thus
		\begin{eqnarray*}
			F_{III}^2(Z_0;V)
			&=&f_{III}^2(AVA')=\frac{q-1}{1+t}\left\{\mathfrak{B}_1(Z_0;V)
			+t\sqrt[k]{\mathfrak{B}_k(Z_0;V)}\right\}.
		\end{eqnarray*}
		Since $Z_0$ is any fixed point in $\mathfrak{R}_{III}$ and $\mathfrak{R}_{III}$ is homogeneous,
		changing $Z_0$ to $Z$ if necessary, we obtain \eqref{FIII}. That $F_{III}$ is an $\mbox{Aut}(\mathfrak{R}_{III})$-invariant
		strongly pseudoconvex complex Finsler metric follows immediately from Theorem \ref{hfm} and Proposition \ref{p-III}.
		
		Next since
		\begin{eqnarray*}
			\frac{\partial^2F_{III}^2}{\partial z_{ab}\partial\overline{v_{ij}}}
			=\frac{1}{2}\left(\frac{\partial^2F_{III}^2}{\partial Z_{ab}\partial \overline{V_{ij}}}
			-\frac{\partial^2F_{III}^2}{\partial Z_{ab}\partial \overline{V_{ji}}}
			-\frac{\partial^2F_{III}^2}{\partial Z_{ba}\partial \overline{V_{ij}}}
			+\frac{\partial^2F_{III}^2}{\partial Z_{ba}\partial \overline{V_{ji}}}
			\right)
		\end{eqnarray*}
		By \eqref{FIII}, it follows that
		\begin{equation}
			\frac{\partial^2F_{III}^2}{\partial z_{ab}\partial\overline{v_{ij}}}\Bigg|_{(0;v)}=0.\label{F_zv-3}
		\end{equation}
		Thus by Theorem \ref{th-2.1} and Remark \ref{RK1}, $F_{III}$ is a complete  K\"ahler-Berwald metric.
	\end{proof}
	\begin{theorem} \label{C-FIII}
		Let $F_{III}$ be defined by \eqref{FIII}. Then $F_{III}$ has negative holomorphic sectional curvature and nonpositive
		holomorphic bisectional curvature.
		More precisely, the holomorphic sectional curvature $K_{III}$ and the holomorphic bisectional
		curvature $B_{III}$ of $F_{III}$ are given respectively by
		\begin{equation}
			K_{III}(Z;V)=-\frac{4(q-1)}{1+t}\frac{\mathfrak{B}_2(Z;V)
				+t{\mathfrak{B}_{k}}^{\frac{1}{k}-1}\mathfrak{B}_{k+1}(Z;V)}{F_{III}^4(Z;V)},\label{hsc-F-III}
		\end{equation}
		and
		\begin{eqnarray}
			B_{III}(Z; V, W)
			&=&-\frac{4(q-1)}{1+t}\frac{\mathcal{B}_{1,1}(Z;V,W)+t\mathfrak{B}_{k}^{\frac{1}{k}-1}(Z;V)
				\mathcal{B}_{k,1}(Z;V,W)}{F_{III}^2(Z;V)F_{III}^2(Z;W)},\label{hbsc-F-III}
		\end{eqnarray}
		where $\mathfrak{B}_l(Z;V)$ and $\mathcal{B}_{i,j}(Z;V)$ are defined by \eqref{s12} and \eqref{s13} respectively
		for $Z\in\mathfrak{R}_{III}$ and
		nonzero vectors $V,W\in T_Z^{1,0}\mathfrak{R}_{III}$.
	\end{theorem}
	\begin{proof}Recall that for any tangent vector $V$ at the point $Z\in\mathfrak{R}_{III}$, one has $ Z=-Z',V=-V'$ and
		$z_{ij}=Z_{ij}$, $v_{ij}=V_{ij}$ for
		$1\leq i<j\leq q$. Thus by \eqref{s12}, for $l=1,2,\cdots,$ we have
		\begin{eqnarray}
			\frac{\partial\mathfrak{B}_l}{\partial z_{ab}}
			&=&-l\mbox{tr}\Big\{\left[(I_q+Z\overline{Z})^{-1}V(I_q+\overline{Z}Z)^{-1}\overline{V'}\right]^{l-1}
			(I_q+Z\overline{Z})^{-1}[(E_{ab}-E_{ba})\overline{Z}(I_q+Z\overline{Z})^{-1}V\nonumber\\
			&&+V(I_q+\overline{Z}Z)^{-1}\overline{Z}(E_{ab}-E_{ba})](I_q+\overline{Z}Z)^{-1}\overline{V'}\Big\},\label{A3-k-a}
		\end{eqnarray}
		from which we get
		$$
		\frac{\partial\mathfrak{B}_l}{\partial z_{ab}}\Big|_{(0;V)}=0,\quad \frac{\partial\mathfrak{B}_l}{\partial \overline{z_{cd}}}\Big|_{(0;V)}=0.
		$$
		Differentiating \eqref{A3-k-a} with respect to $\overline{z_{cd}}$ and then setting $Z=0$, we get
		\begin{eqnarray*}
			\frac{\partial^2\mathfrak{B}_l}{\partial z_{ab}\partial \overline{z_{cd}}}(0;V)
			&=&l\mbox{tr}\Big\{\left[V\overline{V}\right]^{l-1}
			\left[(E_{ab}-E_{ba})(E_{dc}-E_{cd})V+V(E_{dc}-E_{cd})(E_{ab}-E_{ba})\right]\overline{V'}\Big\}.\label{A3-k-ab}
		\end{eqnarray*}
		Now for any  nonzero tangent vector $V\in T_0^{1,0}\mathfrak{R}_{III}$, we have
		\begin{eqnarray*}
			\sum_{a< b}\sum_{c< d}\frac{\partial^2\mathfrak{B}_l}{\partial z_{ab}\partial \overline{z_{cd}}}(0;V)V_{ab}\overline{V_{cd}}&=&2l\mathfrak{B}_{l+1}(0;V).
		\end{eqnarray*}
		Thus by Proposition \ref{hsc}, for any nonzero tangent vector $V\in T_0^{1,0}\mathfrak{R}_{III}$, we have
		\begin{eqnarray}
			K_{III}(0;V)&=&-\frac{4(q-1)}{1+t}\frac{\mathfrak{B}_2(0;V)
				+t\mathfrak{B}_k^{\frac{1}{k}-1}(0;V)\mathfrak{B}_{k+1}(0;V)}{F_{III}^4(0;V)}.\label{hsc3-b1}
		\end{eqnarray}
		Since the holomorphic sectional curvature is holomorphic invariant, and by \eqref{Bz0}, for any non-zero tangent vector $V\in T_Z^{1,0}\mathfrak{R}_{III}$, we have \eqref{hsc-F-III}.
		
		Notice $\mathfrak{B}_{l}(Z;V)>0,$ we always have $K_{III}(Z;V)<0$.
		
		Similarly, for any nonzero tangent vectors $V,W\in T_0^{1,0}\mathfrak{R}_{III}$, we have
		\begin{eqnarray*}
			\sum_{a< b}\sum_{c< d}\frac{\partial^2\mathfrak{B}_l}{\partial Z_{ab}\partial \overline{Z_{cd}}}(0;V)W_{ab}\overline{W_{cd}}&=&2l\mathcal{B}_{l,1}(0;V,W).
		\end{eqnarray*}
		Thus by Proposition \ref{bhsc}, we have
		\begin{eqnarray*}
			B_{III}(0;V,W)&=&-\frac{4(q-1)}{1+t}\frac{\mathcal{B}_{1,1}(0;V,W)
				+t{\mathfrak{B}_{k}^{\frac{1}{k}-1}(0;V)}\mathcal{B}_{k,1}(0;V,W)}{F_{III}^2(0;V)F_{III}^2(0;W)}.\nonumber\\
		\end{eqnarray*}
		Since the holomorphic bisectional curvature is holomorphic-invariant, by \eqref{Bz0} and \eqref{Bz1}
		we have \eqref{hbsc-F-III}.
		Since $ \mathcal{B}_{i,j}(Z;V,W)\geq0, $ we have
		$ B_{III}(Z; V, W)\leq0, $
		with equality holds if and only if
		$$\overline{W}(I_q+Z\overline{Z})^{-1}V=0.$$
	\end{proof}

	\subsection{The supremum and the infimum of holomorphic sectional curvature and bisectional curvature}\label{subsection-3.5}
	
	In \cite{Look3}, Q. K. Lu obtained the supremum and the infimum of holomorphic sectional curvatures for the Bergman metrics on the classical domains. For the  holomorphic sectional curvature of $F_{I}$, $F_{II}$, $F_{III}$, we have
	\begin{theorem} \label{thm-4.4}
		Let  $ \mathfrak{R}_I(m,n), \mathfrak{R}_{II}(p)$ and $ \mathfrak{R}_{III}(q)$ be endowed with the complex Finsler metrics $F_I, F_{II},F_{III}$ defined by
		\eqref{FI}, \eqref{FII} and \eqref{FIII},  respectively. Then the infimum
		and the supremum of the holomorphic sectional curvatures of $(\mathfrak{R}_I(m,n),F_I),$ $(\mathfrak{R}_{II}(p),F_{II}),$ and $(\mathfrak{R}_{III}(q),F_{III})$  are given respectively by
		\begin{eqnarray}
			L(\mathfrak{R}_I(m,n))&=&-\frac{4}{m+n}, \quad\quad\quad\quad\; U(\mathfrak{R}_I(m,n))=\frac{-4}{m+n}\cdot \frac{1+t}{m+t\sqrt[k]{m}};\label{zyc-1}\\
			L(\mathfrak{R}_{II}(p))&=&-\frac{4}{p+1},\quad\quad\quad\quad\quad\;\;\,U(\mathfrak{R}_{II}(p))=\frac{-4}{p+1}\cdot\frac{1+t}{p+t\sqrt[k]{p}};\label{zyc-2}\\
			L(\mathfrak{R}_{III}(q))&=&-\frac{4}{q-1}\cdot\frac{1+t}{2+t\sqrt[k]{2}},\quad
			U(\mathfrak{R}_{III}(q))=\frac{-4}{q-1}\cdot\frac{1+t}{2[\frac{q}{2}]+t\sqrt[k]{2[\frac{q}{2}]}}. \label{zyc-3}
		\end{eqnarray}
	\end{theorem}
	
	\begin{proof} We first show \eqref{zyc-1}.
		Note that $V\in T_0^{1,0}\mathfrak{R}_I$ is an $m\times n$ matrix over the complex number field $\mathbb{C}$,
		thus by Theorem 1.1 on page $299$ in \cite{Look3},
		there exists  $A\in U(m)$ and $B\in U(n)$ such that
		$$
		V=A\begin{pmatrix}
			\lambda_1 & 0 & \cdots & 0 & 0 & \cdots & 0 \\
			0 & \lambda_2 & \cdots & 0 & 0 & \cdots & 0 \\
			\vdots & \vdots & \cdots & \vdots & \vdots & \cdots & \vdots \\
			0 & 0 &  & \lambda_m& 0 &  & 0 \\
		\end{pmatrix}B,\quad \lambda_1\geq \cdots\geq \lambda_m\geq 0,
		$$
		where $\lambda_1^2,\cdots, \lambda_m^2$ are eigenvalues of $V\overline{V'}$. Since $K_I(0;AVB)=K_I(0;V)$ for any $A\in U(m), B\in U(n)$ and $0\neq V\in T_0^{1,0}\mathfrak{R}_I$,
		it suffices to assume that
		$$
		V=\begin{pmatrix}
			\lambda_1 & 0 & \cdots & 0 & 0 & \cdots & 0 \\
			0 & \lambda_2 & \cdots & 0 & 0 & \cdots & 0 \\
			\vdots & \vdots &  & \vdots & \vdots & & \vdots \\
			0 & 0 & \cdots & \lambda_m& 0 & \cdots & 0 \\
		\end{pmatrix},\quad \lambda_1\geq \cdots\geq \lambda_m\geq 0,\quad\lambda_1\neq 0.
		$$
		Setting $x_i:=\lambda_i^2$ for $i=1,\cdots,m$, then by \eqref{hsc-F-I} we have
		\begin{equation}
			K_I(0;V)=-\frac{4(1+t)}{m+n}\cdot\frac{\sum\limits_{i=1}^mx_i^{2}+t\left(\sum\limits_{i=1}^m x_i^{k}\right)^{\frac{1}{k}-1}\left(\sum\limits_{i=1}^m x_i^{k+1}\right)}
			{\left[\sum\limits_{i=1}^m x_i+t\left(\sum\limits_{i=1}^m x_i^{k}\right)^\frac{1}{k}\right]^2}.\label{cei}
		\end{equation}
		Rewriting
		\begin{eqnarray*}
			\left[\sum\limits_{i=1}^m x_i+t\left(\sum\limits_{i=1}^m x_i^{k}\right)^\frac{1}{k}\right]^2
			&=&\left(\sum\limits_{i=1}^m x_i\right)^2+t\left(\sum_{i=1}^mx_i\right)\left(\sum\limits_{i=1}^m x_i^{k}\right)^\frac{1}{k}\\
			&&+t\left(\sum_{i=1}^mx_i^k\right)^{\frac{1}{k}}\left\{\left(\sum_{i=1}^mx_i\right)+t\left(\sum_{i=1}^mx_i^k\right)^{\frac{1}{k}}\right\}
		\end{eqnarray*}
		and using the following inequalities
		$$
		\left(\sum\limits_{i=1}^m x_i\right)^2+t\left(\sum\limits_{i=1}^mx_i\right)\left(\sum\limits_{i=1}^m x_i^{k}\right)^\frac{1}{k}\geq (1+t)\sum_{i=1}^mx_i^2
		$$
		and
		$$
		\left(\sum_{i=1}^mx_i^k\right)\left\{\left(\sum\limits_{i=1}^mx_i\right)+t\left(\sum\limits_{i=1}^mx_i^k\right)^{\frac{1}{k}}\right\}\geq (1+t)\sum_{i=1}^mx_i^{k+1},
		$$
		we get
		\begin{eqnarray*}
			\frac{\sum\limits_{i=1}^mx_i^{2}}{\left(\sum\limits_{i=1}^m x_i\right)^2+t\left(\sum\limits_{i=1}^mx_i\right)\left(\sum\limits_{i=1}^m x_i^{k}\right)^\frac{1}{k}}\leq \frac{1}{1+t},\quad
			\frac{t\left(\sum\limits_{i=1}^m x_i^{k}\right)^{\frac{1}{k}-1}\left(\sum\limits_{i=1}^m x_i^{k+1}\right)}{t\left(\sum\limits_{i=1}^mx_i^k\right)^{\frac{1}{k}}\left\{\left(\sum\limits_{i=1}^mx_i\right)+t\left(\sum\limits_{i=1}^mx_i^k\right)^{\frac{1}{k}}\right\}}\leq \frac{1}{1+t}
		\end{eqnarray*}
		for $t\in(0,+\infty)$.
		Therefore
		\begin{eqnarray}
			K_I(0;V)&\geq&-\frac{4}{m+n}\cdot \frac{1+t}{1+t}=-\frac{4}{m+n}\label{hc-l}
		\end{eqnarray}
		for any $t\in(0,+\infty)$. It is obvious that \eqref{hc-l} also holds for $t=0$.
		
		On the other hand,  the function $\left(\displaystyle\sum_{i=1}^mx_i^l\right)^{\frac{1}{l}}$ (considered as a function of $l$ for any fixed non-negative real numbers $x_1,\cdots,x_m$) is monotone decreasing with respect to $l$, it follows that
		\begin{eqnarray*}
			\left(\sum_{i=1}^mx_i^{k+1}\right)^{\frac{1}{k+1}}\leq \left(\sum_{i=1}^mx_i^k\right)^{\frac{1}{k}}.
		\end{eqnarray*}
		Using the H\"older inequality, we have
		$$
		\sum_{i=1}^mx_i^k\leq \left(\sum_{i=1}^m 1^{k+1}\right)^{\frac{1}{k+1}}\left(\sum_{i=1}^mx_i^{k+1}\right)^{\frac{k}{k+1}}=m^{\frac{1}{k+1}}\left(\sum_{i=1}^mx_i^{k+1}\right)^{\frac{k}{k+1}},
		$$
		from which yields
		$$
		\left(\sum_{i=1}^mx_i^k\right)^{\frac{k+1}{k}}\leq \sqrt[k]{m}\sum_{i=1}^mx_i^{k+1}.
		$$
		So that
		\begin{equation}
			\sqrt[k]{m}\left(\sum_{i=1}^mx_i^k\right)^{\frac{1}{k}-1}\left(\sum_{i=1}^mx_i^{k+1}\right)\geq \left(\sum_{i=1}^mx_i^k\right)^{\frac{2}{k}},\label{iq-a}
		\end{equation}
		which together with the following inequality
		\begin{eqnarray}
			\sum\limits_{i=1}^mx_i^{2}\leq\left(\sum\limits_{i=1}^mx_i\right)^2
			\leq m\sum\limits_{i=1}^mx_i^2 \label{iq-b}
		\end{eqnarray}
		implies that
		\begin{eqnarray*}
			&&(m+t\sqrt[k]{m})\left\{\sum\limits_{i=1}^mx_i^{2}+t\left(\sum\limits_{i=1}^m x_i^{k}\right)^{\frac{1}{k}-1}\left(\sum\limits_{i=1}^m x_i^{k+1}\right)\right\}-\left[\sum\limits_{i=1}^m x_i+t\left(\sum\limits_{i=1}^m x_i^{k}\right)^\frac{1}{k}\right]^2\\
			&\geq&t\left\{\sqrt[k]{m}\sum\limits_{i=1}^mx_i^{2}
			+m\left(\sum\limits_{i=1}^m x_i^{k}\right)^{\frac{1}{k}-1}\left(\sum\limits_{i=1}^mx_i^{k+1}\right)
			-2\left(\sum\limits_{i=1}^mx_i\right)\left(\sum\limits_{i=1}^mx_i^{k}\right)^\frac{1}{k}\right\}\\
			&\geq&2t\sqrt{m\left(\sum\limits_{i=1}^mx_i^{2}\right)
				\cdot \sqrt[k]{m}\left(\sum\limits_{i=1}^mx_i^{k}\right)^{\frac{1}{k}-1}\left(\sum\limits_{i=1}^mx_i^{k+1}\right)}
			-2t\left(\sum\limits_{i=1}^mx_i\right)\left(\sum\limits_{i=1}^mx_i^{k}\right)^\frac{1}{k}\\	
			&\geq&0,
		\end{eqnarray*}
		where in the last inequality we use again \eqref{iq-a} and \eqref{iq-b}.
		Therefore
		\begin{eqnarray}
			K_I(0;V)\leq -\frac{4}{m+n}\cdot \frac{1+t}{m+t\sqrt[k]{m}}.
		\end{eqnarray}
		Notice that $ K_{I}(0;E_{11})=\frac{-4}{m+n} $ and $ K_{I}(0;E_{11}+\cdots+E_{mm})=\frac{-4}{m+n}\cdot\frac{1+t}{m+t\sqrt[k]{m}}.$ Since $\mathfrak{R}_I$ is transitive, the lower and upper bounds of $K_I(Z;V)$ at any point $Z\in\mathfrak{R}_I$ and along any nonzero tangent direction $V\in T_Z^{1,0}\mathfrak{R}_I$ are equal to those at the point $Z=0$ and along any nonzero tangent direction $V\in T_0^{1,0}\mathfrak{R}_I$.
		Hence we have \eqref{zyc-1}.
		Morever,  $  -\frac{4}{m+n}$  and $  -\frac{4}{m+n}\cdot \frac{1+t}{m+t\sqrt[k]{m}} $ are infimum and supremum of $ K_{I} $ respectively.
		
		Next we show \eqref{zyc-2}.
		In this case, each $V\in T_0^{1,0}\mathfrak{R}_{II}$ is a $p\times p$ symmetric matrix over the complex number field $\mathbb{C}$, and $K_{II}(0;A'VA)=K_{II}(0;V)$ for any nonzero vector $V\in T_0^{1,0}\mathfrak{R}_{II}$ and $A\in U(p)$. Thus by Theorem 1.2 on page $301$ in \cite{Look3}, it suffices to take
		$$
		V=\begin{pmatrix}
			\lambda_1 & 0 & \cdots & 0 \\
			0 & \lambda_2 & \cdots & 0 \\
			\vdots & \vdots &  & \vdots\\
			0 & 0 & \cdots & \lambda_p \\
		\end{pmatrix},\quad (\lambda_1\geq \lambda_2\geq \cdots\geq \lambda_p\geq 0,\lambda_1>0).
		$$
		Now setting $x_i:=\lambda_i^2$ for $i=1,\cdots,p$, it follows from \eqref{hsc-F-II} that
		\begin{equation}
			K_{II}(0;V)=-\frac{4(1+t)}{p+1}\cdot\frac{\sum\limits_{i=1}^px_i^{2}+t\left(\sum\limits_{i=1}^p x_i^{k}\right)^{\frac{1}{k}-1}\left(\sum\limits_{i=1}^p x_i^{k+1}\right)}
			{\left[\sum\limits_{i=1}^p x_i+t\left(\sum\limits_{i=1}^p x_i^{k}\right)^\frac{1}{k}\right]^2}.\label{cei-II}
		\end{equation}
		By a similar argument as the proof of \eqref{zyc-1}, we obtain \eqref{zyc-2}.
		Notice that $ K_{II}(0;E_{11})=\frac{-4}{p+1} $ and $ K_{II}(0;E_{11}+\cdots+E_{pp})=\frac{-4}{p+1}\cdot\frac{1+t}{p+t\sqrt[k]{p}}.$
		Thus  $  -\frac{4}{p+1}$  and $  -\frac{4}{p+1}\cdot \frac{1+t}{p+t\sqrt[k]{p}} $ are infimum and supremum of $ K_{II} $ respectively.
		
		Now we  show \eqref{zyc-3}. In this case, each $V\in T_0^{1,0}\mathfrak{R}_{III}$ is a $q\times q$ skew-symmetric matrix over the complex number field $\mathbb{C}$,
		and
		$K_{III}(0;A'VA)=K_{III}(0;V)$ for any nonzero vector $V\in T_0^{1,0}\mathfrak{R}_{III}$ and any $A\in U(q)$,
		thus by Theorem 1.4 on page 303 of \cite{Look3}, it suffices to take
		\begin{equation}
			V=\left\{
			\begin{array}{ll}
				\begin{pmatrix}
					0 & \lambda_1 \\
					-\lambda_1 & 0 \\
				\end{pmatrix}
				\dot{+}\cdots\dot{+}
				\begin{pmatrix}
					0 & \lambda_\nu \\
					-\lambda_\nu & 0 \\
				\end{pmatrix}
				, & \hbox{for}\;q=2\nu \\
				\begin{pmatrix}
					0 & \lambda_1 \\
					-\lambda_1 & 0 \\
				\end{pmatrix}
				\dot{+}\cdots\dot{+}
				\begin{pmatrix}
					0 & \lambda_\nu \\
					-\lambda_\nu & 0 \\
				\end{pmatrix}\dot{+}0& \hbox{for}\;q=2\nu+1,
			\end{array}
			\right.\label{V-d}
		\end{equation}
		where $\dot{+}$ mean direct sum of matrices, $\lambda_1\geq \lambda_2\geq \cdots\geq \lambda_\nu\geq 0$ and $\lambda_1>0$. As before, we denote $x_i:=\lambda_i^2$ for $i=1,\cdots,\nu$.
		Substituting \eqref{V-d} into \eqref{hsc-F-III}, we obtain
		$$
		K_{III}(0;V)=-\frac{4(1+t)}{q-1}\cdot \frac{2\sum\limits_{i=1}^\nu x_i^2+t\left(2\sum\limits_{i=1}^\nu x_i^k\right)\overline{}^{\frac{1}{k}-1}\cdot \left(2\sum\limits_{i=1}^\nu x_i^{k+1}\right)}{\left\{2\sum\limits_{i=1}^\nu x_i+t\left(2\sum\limits_{i=1}^\nu x_i^{k}\right)^{\frac{1}{k}}\right\}^2},\quad\forall 0\neq V\in T_0^{1,0}\mathfrak{R}_{III}.
		$$
		
		Rewriting
		\begin{eqnarray*}
			\left[2\sum\limits_{i=1}^\nu x_i+t\left(2\sum\limits_{i=1}^\nu x_i^{k}\right)^\frac{1}{k}\right]^2
			&=&\left(2\sum\limits_{i=1}^\nu x_i\right)^2+t\left(2\sum_{i=1}^\nu x_i\right)\left(2\sum\limits_{i=1}^\nu x_i^{k}\right)^\frac{1}{k}\\
			&&+t\left(2\sum_{i=1}^\nu x_i^k\right)^{\frac{1}{k}}\left\{\left(2\sum_{i=1}^\nu x_i\right)+t\left(2\sum_{i=1}^\nu x_i^k\right)^{\frac{1}{k}}\right\}
		\end{eqnarray*}
		and using the following inequalities
		$$
		\left(2\sum\limits_{i=1}^\nu x_i\right)^2+t\left(2\sum\limits_{i=1}^\nu x_i\right)\left(2\sum\limits_{i=1}^\nu x_i^{k}\right)^\frac{1}{k}
		\geq (2+t\sqrt[k]{2})\left(2\sum_{i=1}^\nu x_i^2\right)
		$$
		and
		$$
		\left(2\sum_{i=1}^\nu x_i^k\right)\left\{\left(2\sum\limits_{i=1}^\nu x_i\right)+t\left(2\sum\limits_{i=1}^\nu x_i^k\right)^{\frac{1}{k}}\right\}
		\geq (2+t\sqrt[k]{2})\left(2\sum_{i=1}^\nu x_i^{k+1}\right),
		$$
		we get
		\begin{eqnarray*}
			\frac{2\sum\limits_{i=1}^\nu x_i^{2}}{\left(2\sum\limits_{i=1}^\nu x_i\right)^2+t\left(2\sum\limits_{i=1}^\nu x_i\right)\left(2\sum\limits_{i=1}^\nu x_i^{k}\right)^\frac{1}{k}}&\leq& \frac{1}{2+t\sqrt[k]{2}},\\
			\frac{t\left(2\sum\limits_{i=1}^\nu x_i^{k}\right)^{\frac{1}{k}-1}\left(2\sum\limits_{i=1}^\nu x_i^{k+1}\right)}{t\left(2\sum\limits_{i=1}^\nu x_i^k\right)^{\frac{1}{k}}\left\{\left(2\sum\limits_{i=1}^\nu x_i\right)+t\left(2\sum\limits_{i=1}^\nu x_i^k\right)^{\frac{1}{k}}\right\}}&\leq& \frac{1}{2+t\sqrt[k]{2}}
		\end{eqnarray*}
		for $t\in(0,+\infty)$.
		Therefore
		\begin{eqnarray}
			K_{III}(0;V)&\geq&-\frac{4}{q-1}\cdot \frac{1+t}{2+t\sqrt[k]{2}}\label{hc-l-III}
		\end{eqnarray}
		for any $t\in(0,+\infty)$. It is obvious that \eqref{hc-l-III} also holds for $t=0$.
		
		Similarly, the function $\left(2\displaystyle\sum_{i=1}^\nu x_i^l\right)^{\frac{1}{l}}$ (considered as a function of $l$ for any fixed $x_1,\cdots,x_m$) is also monotone decreasing with respect to $l$, thus
		\begin{eqnarray*}
			\left(2\sum_{i=1}^\nu x_i^{k+1}\right)^{\frac{1}{k+1}}\leq \left(2\sum_{i=1}^\nu x_i^k\right)^{\frac{1}{k}}.
		\end{eqnarray*}
		Using the H\"older inequality, we have
		$$
		2\sum_{i=1}^\nu x_i^k\leq \left(\sum_{i=1}^\nu 1^{k+1}\right)^{\frac{1}{k+1}}\left(\sum_{i=1}^\nu (2x_i^k)^{\frac{k+1}{k}}\right)^{\frac{k}{k+1}}=\left\{2\left[\frac{q}{2}\right]\right\}^{\frac{1}{k+1}}\left(2\sum_{i=1}^\nu  x_i^{k+1}\right)^{\frac{k}{k+1}},
		$$
		from which yields
		$$
		\left(2\sum_{i=1}^\nu x_i^k\right)^{\frac{k+1}{k}}\leq \sqrt[k]{2\left[\frac{q}{2}\right]}\left(2\sum_{i=1}^\nu x_i^{k+1}\right).
		$$
		So that
		\begin{equation}
			\sqrt[k]{2\left[\frac{q}{2}\right]}\left(2\sum_{i=1}^\nu x_i^k\right)^{\frac{1}{k}-1}\left(2\sum_{i=1}^\nu x_i^{k+1}\right)\geq \left(2\sum_{i=1}^\nu x_i^k\right)^{\frac{2}{k}},\label{iq-a-III}
		\end{equation}
		which together with the following inequality
		\begin{eqnarray}
			4\sum\limits_{i=1}^\nu x_i^{2}\leq\left(2\sum\limits_{i=1}^\nu x_i\right)^2
			\leq 2\left[\frac{q}{2}\right]\left(2\sum\limits_{i=1}^\nu x_i^2\right) \label{iq-b-III}
		\end{eqnarray}
		implies that
		\begin{eqnarray*}
			&&\left(2\left[\frac{q}{2}\right]+t\sqrt[k]{2\left[\frac{q}{2}\right]}\right)
			\left\{2\sum\limits_{i=1}^\nu x_i^{2}+t\Big(2\sum\limits_{i=1}^\nu x_i^{k}\Big)^{\frac{1}{k}-1}\Big(2\sum\limits_{i=1}^\nu x_i^{k+1}\Big)\right\}
			-\left[2\sum\limits_{i=1}^\nu x_i+t\Big(2\sum\limits_{i=1}^\nu x_i^{k}\Big)^{\frac{1}{k}}\right]^2
		\geq0,
	\end{eqnarray*}
	where in the last inequality we use again \eqref{iq-a-III} and \eqref{iq-b-III}.
	Therefore
	\begin{eqnarray}
		K_{III}(0;V)\leq -\frac{4}{q-1}\cdot \frac{1+t}{2\left[\frac{q}{2}\right]+t\sqrt[k]{2\left[\frac{q}{2}\right]}}.
	\end{eqnarray}
	Notice that $ K_{III}(0;E_{12}-E_{21})=-\frac{4}{q-1}\cdot \frac{1+t}{2+t\sqrt[k]{2}} $ and for
	$$
	V_0=\left\{
	\begin{array}{ll}
		\begin{pmatrix}
			0 & 1 \\
			-1 & 0 \\
		\end{pmatrix}
		\dot{+}\cdots\dot{+}
		\begin{pmatrix}
			0 & 1 \\
			-1 & 0 \\
		\end{pmatrix}
		, & \hbox{for}\;q=2\nu \\
		\begin{pmatrix}
			0 & 1 \\
			-1 & 0 \\
		\end{pmatrix}
		\dot{+}\cdots\dot{+}
		\begin{pmatrix}
			0 & 1 \\
			-1 & 0 \\
		\end{pmatrix}\dot{+}0& \hbox{for}\;q=2\nu+1,
	\end{array}
	\right.
	$$
	$ K_{III}(0;V_0)= -\frac{4}{q-1}\cdot \frac{1+t}{2\left[\frac{q}{2}\right]+t\sqrt[k]{2\left[\frac{q}{2}\right]}}.$
	Thus  $  -\frac{4}{q-1}\cdot \frac{1+t}{2+t\sqrt[k]{2}} $  and $  -\frac{4}{q-1}\cdot \frac{1+t}{2\left[\frac{q}{2}\right]+t\sqrt[k]{2\left[\frac{q}{2}\right]}} $ are infimum and supremum of $ K_{III} $ respectively.
\end{proof}

As for the supremum and the infimum of holomorphic bisectional curvature for $F_A (A=I,II,III)$, we have
\begin{theorem} \label{thm4}
	Let  $\mathfrak{R}_I(m,n), \mathfrak{R}_{II}(p)$ and $\mathfrak{R}_{III}(q)$ be endowed with the $\mbox{Aut}(\mathfrak{R}_A)$-invariant K\"ahler-Berwald metrics respectively defined  by \eqref{FI},
	\eqref{FII} and \eqref{FIII}. Then the holomorphic bisectional curvatures of $(\mathfrak{R}_I(m,n),F_I),(\mathfrak{R}_{II}(p),F_{II})$
	and $(\mathfrak{R}_{III}(q),F_{III})$ satisfy
	\begin{eqnarray}
		-\frac{4}{m+n}&\leq& B_{I}(\mathfrak{R}_I(m,n))\leq 0,\label{hb-I}\\
		-\frac{4}{m+n}&\leq& B_{II}(\mathfrak{R}_{II}(p))\leq 0,\label{hb-II}\\
		-\frac{4(1+t)}{q-1}\cdot \frac{1}{2+t\sqrt[k]{2}}&\leq& B_{III}(\mathfrak{R}_{III}(q))\leq 0\label{hb-III}
	\end{eqnarray}
	and these give the supremum and the infimum of holomorphic bisectional curvatures for $F_{A} (A=I,II,III)$, respectively.
\end{theorem}

\begin{proof}	
	
	(1) Since for any $V,W\in T_0^{1,0}\mathfrak{R}_I$, there exist unitary matrices $U_1, U_2\in U(m)$ and $L_1,L_2\in U(n)$ such that
	$$
	U_1V\overline{L_1'}=\begin{pmatrix}
		\lambda_1 &\cdots  & 0 & 0 & \cdots & 0 \\
		\vdots& \ddots & \vdots & \vdots & \vdots & \vdots \\
		0& \cdots & \lambda_m & 0 & \cdots & 0 \\
	\end{pmatrix}_{m\times n},\quad
	U_2W\overline{L_2'}=\begin{pmatrix}
		\mu_1 &\cdots  & 0 & 0 & \cdots & 0 \\
		\vdots& \ddots & \vdots & \vdots & \vdots & \vdots \\
		0& \cdots & \mu_m & 0 & \cdots & 0 \\
	\end{pmatrix}_{m\times n},
	$$
	where $\lambda_i,\mu_i\geq0$ for $i=1,\cdots, m.$	
	Denote $P:=(P_{ij})=U_1\overline{U_2'}$, then $P\in U(m)$ and we have
	\begin{eqnarray*}
		0\leq \mbox{tr}\left\{V\overline{V'}W\overline{W'}\right\}&=&\sum_{i=1}^m\sum_{j=1}^m\overline{P_{ij}}\lambda_i^2 P_{ij}\mu_j^2\leq \max_{1\leq i\leq m}\{\lambda_i^2\}\cdot\left(\sum_{j=1}^m\mu_j^2\right),\\
		0\leq \mbox{tr}\left\{t\mathfrak{B}_k^{\frac{1}{k}-1}(0;V)\left[V\overline{V'}\right]^kW\overline{W'}\right\}&=&t\left(\sum_{i=1}^m\lambda_i^{2k}\right)^{\frac{1}{k}-1}\left(\sum_{i=1}^m\sum_{j=1}^m\overline{P_{ij}}\lambda_i^{2k} P_{ij}\mu_j^2\right)\\
		&\leq& t\left(\sum_{i=1}^m\lambda_i^{2k}\right)^{\frac{1}{k}}\cdot \max_{1\leq j\leq m}\{\mu_j^2\}.
	\end{eqnarray*}
	So that
	\begin{eqnarray}
		G(0;V)G(0;W)
		&=&\left(\frac{m+n}{1+t}\right)^2\left\{\left[\left(\sum_{i=1}^{m}\lambda_i^2\right)
		+t\left(\sum_{i=1}^{m}\lambda_i^{2k}\right)^{\frac{1}{k}}\right]\left(\sum_{i=1}^{m}\mu_i^2\right)\right.\\
		&&\left.+t\left[\sum_{i=1}^{m}\lambda_i^2+t\left(\sum_{i=1}^{m}\lambda_i^{2k}\right)^{\frac{1}{k}}\right]\left(\sum_{i=1}^{m}\mu_i^{2k}\right)^{\frac{1}{k}}\right\}
		\label{BB2}
	\end{eqnarray}
	and
	\begin{eqnarray}
		0&\leq& \mbox{tr}\left\{V\overline{V'}W\overline{W'}+t\mathfrak{B}_{k}^{\frac{1}{k}-1}(0;V)\left(V\overline{V'}\right)^{k}W\overline{W'}\right\}\nonumber\\
		&\leq &\max_{1\leq i\leq m}\{\lambda_i^2\}\cdot \left(\sum_{j=1}^m\mu_j^2\right)
		+t\left(\sum_{i=1}^{m}\lambda_i^{2k}\right)^{\frac{1}{k}}\cdot \max_{1\leq j\leq m}\{\mu_j^2\}.\label{BB3}
	\end{eqnarray}
	Now notice that
	\begin{eqnarray} &&\frac{\max\limits_{1\leq i\leq m}\{\lambda_i^2\}\cdot\left(\sum\limits_{j=1}^{m}\mu_j^2\right)}{\left[\left(\sum\limits_{i=1}^{m}\lambda_i^2\right)+t\left(\sum\limits_{i=1}^{m}\lambda_i^{2k}\right)^{\frac{1}{k}}\right]\left(\sum\limits_{i=1}^{m}\mu_i^2\right)}
		\leq\frac{1}{1+t}\label{BB40}
	\end{eqnarray}
	and
	\begin{eqnarray}
		\frac{\left(\sum\limits_{i=1}^{m}\lambda_i^{2k}\right)^{\frac{1}{k}}\cdot \max\limits_{1\leq j\leq m}\{\mu_j^2\}}
		{\left[\sum\limits_{i=1}^{m}\lambda_i^2 +t\left(\sum\limits_{i=1}^{m}\lambda_i^{2k}\right)^{\frac{1}{k}}\right]\left(\sum\limits_{i=1}^{m}\mu_i^{2k}\right)^{\frac{1}{k}}}	
		\leq\frac{1}{1+t}. \label{BB50}
	\end{eqnarray}
	By \eqref{hbsc-F-I}, it follows that
	$$-\frac{4}{m+n} \leq B_I(\mathfrak{R}_{I}(m,n))\leq 0. $$
	Notice that $ B_{I}(0;E_{11},E_{11})=K_{I}(0;E_{11})=-\frac{4}{m+n} $ and $ B_{I}(0;E_{11},E_{22})=0. $
	Thus $ -\frac{4}{m+n} $  and 0 are infimum and supremum of $ B_{I} $ respectively.
	
	(2)	In this case, it suffices to note that for $V,W\in T_0^{1,0}\mathfrak{R}_{II}(p)$, there exist unitary matrices $U_1,U_2\in U(p)$ such that
	$$
	U_1VU_1'=\begin{pmatrix}
		\lambda_1 &\cdots  & 0  \\
		\vdots& \ddots & \vdots  \\
		0& \cdots & \lambda_p  \\
	\end{pmatrix}_{p\times p},\quad
	U_2WU_2'=\begin{pmatrix}
		\mu_1 &\cdots  & 0  \\
		\vdots& \ddots & \vdots  \\
		0& \cdots & \mu_p  \\
	\end{pmatrix}_{p\times p},
	$$
	where $\lambda_i,\mu_i\geq0$ for $i=1,\cdots, p$. Then follows the steps of (1) we get \eqref{hb-II}.
	Notice that $ B_{II}(0;E_{11},E_{11})=K_{II}(0;E_{11})=-\frac{4}{p+1} $ and $ B_{II}(0;E_{11},E_{22})=0. $
	Thus $ -\frac{4}{p+1} $  and 0 are infimum and supremum of $ B_{II} $ respectively.
	
	(3) We just give the proof for $q=2\nu$, the case for $q=2\nu+1$ is the same. For vectors $V,W\in T_0^{1,0}\mathfrak{R}_{III}$, there exist  unitary matrices $U_1,U_2\in U(q)$ such that
	\begin{eqnarray*}
		U_1VU_1'=\begin{pmatrix}
			0 & \lambda_1 & \cdots & 0 & 0 \\
			-\lambda_1 & 0 & \cdots & 0 & 0 \\
			\vdots & \vdots &\ddots  & \vdots & \vdots \\
			0 & 0 &\cdots  & 0 & \lambda_\nu \\
			0 & 0 & \cdots & -\lambda_\nu & 0 \\
		\end{pmatrix},\quad
		U_2WU_2'=\begin{pmatrix}
			0 & \mu_1 & \cdots & 0 & 0 \\
			-\mu_1 & 0 & \cdots & 0 & 0 \\
			\vdots & \cdots &\ddots  & \vdots & \vdots \\
			0 & 0 &\cdots  & 0 & \mu_\nu \\
			0 & 0 & \cdots & -\mu_\nu & 0 \\
		\end{pmatrix}
	\end{eqnarray*}
	for $\lambda_i,\mu_i\geq0,i=1,\cdots \nu$. In this case
	\begin{eqnarray}
		G(0;V)G(0;W)
		&=&\left(\frac{q-1}{1+t}\right)^2\left\{\left[2\left(\sum_{i=1}^{\nu}\lambda_i^2\right)
		+t\left(2\sum_{i=1}^{\nu}\lambda_i^{2k}\right)^{\frac{1}{k}}\right]\left(2\sum_{i=1}^{\nu}\mu_i^2\right)\right.\\
		&&\left.+t\left[2\sum_{i=1}^{\nu}\lambda_i^2+t\left(2\sum_{i=1}^{\nu}\lambda_i^{2k}\right)^{\frac{1}{k}}\right]\left(2\sum_{i=1}^{\nu}\mu_i^{2k}\right)^{\frac{1}{k}}\right\}
		\label{BB12}
	\end{eqnarray}
	and
	\begin{eqnarray}
		0&\leq& \mbox{tr}\left\{V\overline{V'}W\overline{W'}+t\mathfrak{B}_{k}^{\frac{1}{k}-1}(0;V)\left(V\overline{V'}\right)^{k}W\overline{W'}\right\}\nonumber\\
		&\leq &\max_{1\leq i\leq \nu}\{\lambda_i^2\}\cdot \left(2\sum_{j=1}^\nu\mu_j^2\right)
		+t\left(2\sum_{i=1}^{\nu}\lambda_i^{2k}\right)^{\frac{1}{k}}\cdot \max_{1\leq j\leq \nu}\{\mu_j^2\}.\label{BB3}
	\end{eqnarray}
	Now notice that
	\begin{eqnarray} &&\frac{\max\limits_{1\leq i\leq \nu }\{\lambda_i^2\}\cdot\left(2\sum\limits_{j=1}^{\nu}\mu_j^2\right)}{\left[\left(2\sum\limits_{i=1}^{\nu}\lambda_i^2\right)+t\left(2\sum\limits_{i=1}^{\nu}\lambda_i^{2k}\right)^{\frac{1}{k}}\right]\left(2\sum\limits_{i=1}^{\nu}\mu_i^2\right)}
		\leq\frac{1}{2+t\sqrt[k]{2}}\label{BB40}
	\end{eqnarray}
	and
	\begin{eqnarray}
		\frac{\left(2\sum\limits_{i=1}^{\nu}\lambda_i^{2k}\right)^{\frac{1}{k}}\cdot \max\limits_{1\leq j\leq \nu}\{\mu_j^2\}}
		{\left[2\sum\limits_{i=1}^{\nu}\lambda_i^2 +t\left(2\sum\limits_{i=1}^{\nu}\lambda_i^{2k}\right)^{\frac{1}{k}}\right]\left(2\sum\limits_{i=1}^{\nu}\mu_i^{2k}\right)^{\frac{1}{k}}}	
		\leq\frac{1}{2+t\sqrt[k]{2}}. \label{BB50}
	\end{eqnarray}
	By \eqref{hbsc-F-III}, we obtain \eqref{hb-III}.
	Notice that $ B_{III}(0;E_{12}-E_{21},E_{12}-E_{21})=K_{III}(0;E_{12}-E_{21})=-\frac{4(1+t)}{q-1}\cdot \frac{1}{2+t\sqrt[k]{2}} $ and $ B_{III}(0;E_{12}-E_{21},E_{34}-E_{43})=0. $
	Thus $ -\frac{4(1+t)}{q-1}\cdot \frac{1}{2+t\sqrt[k]{2}}  $  and 0 are infimum and supremum of $ B_{III} $ respectively.
	
\end{proof}

\subsection{$\mbox{Aut}(\mathfrak{R}_{IV})$-invariant K\"ahler-Berwald metrics on $\mathfrak{R}_{IV}$}\label{subsection-3.6}

In this subsection, we shall construct $\mbox{Aut}(\mathfrak{R}_{IV})$-invariant strongly pseudoconvex complex Finsler metrics on $\mathfrak{R}_{IV}$. First let's recall the holomorphic automorphism of $\mathfrak{R}_{IV}$.

For any fixed point $z_0\in\mathfrak{R}_{IV}$, there exists a $\phi_{z_0}\in\mbox{Aut}(\mathfrak{R}_{IV})$ such that $\phi_{z_0}(z_0)=0$. Indeed, (cf. \cite{Hua}, \cite{Look3})
\begin{equation}
	w=\phi_{z_0}(z)=\left\{\left[\left(\frac{1+zz'}{2},\frac{1-zz'}{2i}\right)-zX_0'\right]A\begin{pmatrix}
		1 \\
		i \\
	\end{pmatrix}
	\right\}^{-1}\left[z-\left(\frac{1+zz'}{2},\frac{1-zz'}{2i}\right)X_0\right]D,
	\label{p-IV}
\end{equation}
where
\begin{equation}
	X_0=2\begin{pmatrix}
		z_0z_0'+1 & i(z_0z_0'-1) \\
		\overline{z_0z_0'}+1 & -i(\overline{z_0z_0'}-1) \\
	\end{pmatrix}
	\begin{pmatrix}
		z_0 \\
		\overline{z_0} \\
	\end{pmatrix}
	=\frac{-1}{1-|z_0z_0'|^2}\begin{pmatrix}
		(\overline{z_0z_0'}-1)z_0+(z_0z_0'-1)\overline{z_0}  \\
		i(z_0z_0'+1)\overline{z_0}-i(\overline{z_0z_0'}+1)z_0 \\
	\end{pmatrix}
	\label{p-IV-a}
\end{equation}
is a real $2$-by-$N$ matrix satisfying $I-X_0X_0'>0$, and
$A\in\mathscr{M}(2)$, $D\in\mathscr{M}(N)$ are real matrices satisfying
\begin{equation}
	AA'=(I-X_0X_0')^{-1},\quad DD'=(I-X_0'X_0)^{-1},\quad \det A>0.\label{p-IV-b}
\end{equation}

Conversely, every element in $\mbox{Aut}(\mathfrak{R}_{IV})$ can be expressed into the form \eqref{p-IV}  (cf. \cite{Hua}, \cite{Look3}).

By Theorem \ref{hfm}, in order to construct an $\mbox{Aut}(\mathfrak{R}_{IV})$-invariant strongly pseudoconvex complex Finsler metric on $\mathfrak{R}_{IV}$,
it suffices to construct an $\mbox{Iso}(\mathfrak{R}_{IV})$-invariant complex Minkowski norm on $T_0^{1,0}\mathfrak{R}_{IV}$.

Note that by \eqref{p-IV}-\eqref{p-IV-b}, it follows that the isotropy subgroup $\mbox{Iso}(\mathfrak{R}_{IV})$ at the origin $0\in\mathfrak{R}_{IV}$ consists of mappings of the following form:
\begin{equation}
	w=\phi_0(z)=e^{i\theta}zD,\quad \forall \theta\in \mathbb{R}, z\in\mathfrak{R}_{IV}, D\in O(N;\mathbb{R}).
\end{equation}

Thus it suffices to construct a  complex Minkowski norm which is both rotational invariant  and orthogonal invariant.

Let $\phi:[0,1]\rightarrow (0,+\infty)$ be a smooth function. Let's consider the following function
\begin{equation}
	f_{IV}(\xi)=\sqrt{\pmb{r}\phi(\pmb{s})},\quad\pmb{r}:=\xi\overline{\xi}',\quad \pmb{s}:=\frac{|\xi\xi'|^2}{\pmb{r}^2}\in[0,1],\quad 0\neq \xi=(\xi_1,\cdots,\xi_N)\in\mathbb{C}^N.\label{cmn}
\end{equation}
Since $\pmb{s}$ is zero homogeneous with respect to
$\xi\in\mathbb{C}^N\setminus\{0\}$, $\pmb{s}$ actually defines a smooth function on $\mathbb{CP}^{N-1}$, hence $\phi$ is bounded, namely there exist positive constants $C_1,C_2$ such that $C_1\leq \phi(\pmb{s})\leq C_2$.
It is clear that $f_{IV}:\mathbb{C}^N\rightarrow [0,+\infty)$ is a complex norm on $\mathbb{C}^N$.
\begin{proposition}
	Let $f_{IV}$ be the complex norm on $\mathbb{C}^N\cong T_0^{1,0}\mathfrak{R}_{IV}$ defined by \eqref{cmn}. Then $f_{IV}$ is
	$\mbox{Iso}(\mathfrak{R}_{IV})$-invariant.
\end{proposition}
\begin{proof}
	It suffices to notice that $\pmb{r}$ and $\pmb{s}$ are both rotational invariant and orthogonal invariant.
\end{proof}

In the following we shall derive the necessary and sufficient conditions for $\phi$ such that $f_{IV}(\xi)$ is a complex Minkowski norm. For this purpose, we denote $g:=f_{IV}^2$, and for functions of $\xi\in\mathbb{C}^N$ we use lower subscripts to denote
the derivatives with respect to $\xi_i$ or $\overline{\xi_j}$, for example
$g_i=\frac{\partial g}{\partial \xi_i},g_{i\overline{j}}=\frac{\partial^2g}{\partial \xi_i\overline{\xi_j}},\pmb{s}_i=\frac{\partial \pmb{s}}{\partial \xi_i},\pmb{s}_{i\bar{j}}=\frac{\partial^2 \pmb{s}}{\partial \xi_i\partial\overline{\xi_j}}$, and so on.
\begin{proposition}\label{prop}
	Let $f_{IV}(\xi)=\sqrt{\pmb{r}\phi(\pmb{s})}$ be the complex norm defined by \eqref{cmn}. Then $f_{IV}(\xi)$ is strongly pseudoconvex if and only if
	\begin{equation}
		\phi-2\pmb{s}\phi'>0\quad \mbox{and}\quad \phi[\phi+2(2-3\pmb{s})\phi']+4\pmb{s}(1-\pmb{s})[\phi \phi''-(\phi')^2]>0,\label{sn}
	\end{equation}
	where $\phi',\phi''$ denote the derivatives with respect to $\pmb{s}$.
\end{proposition}
\begin{remark}\label{remark}
	If we write $\phi(\pmb{s})=e^{\psi(\pmb{s})}$ for a real-valued smooth function $\psi(\pmb{s})$ defined on $[0,1]$, it is easy to check that \eqref{sn} is equivalent to the following inequalities
	\begin{equation}
		1-2\pmb{s}\psi'(\pmb{s})>0,\quad 1+2(2-3\pmb{s})\psi'(\pmb{s})+4\pmb{s}(1-\pmb{s})\psi''(\pmb{s})>0.\label{ssn}
	\end{equation}
\end{remark}	
\begin{proof}
	Note that
	\begin{eqnarray}
		\pmb{s}_i&=&\frac{2}{\pmb{r}^2}\{\overline{\xi\xi'}\xi_i-\pmb{rs}\overline{\xi_i}\},\label{s-i}\\
		\pmb{s}_{i\overline{j}}&=&\frac{2}{\pmb{r}^2}\{2\xi_i\overline{\xi_j}-\pmb{s}\overline{\xi_i}\xi_j-\pmb{rs}\delta_{ij}-\pmb{r}\overline{\xi_i}\pmb{s}_{\overline{j}}-\pmb{r}\pmb{s}_i\xi_j\}.\label{s-ij}
	\end{eqnarray}
	It follows that
	\begin{eqnarray}
		g_{i\overline{j}}&=&c_0\delta_{ij}-\phi'\overline{\xi_i}\pmb{s}_{\overline{j}}-\phi'\pmb{s}_i\xi_j+\pmb{r}\phi''\pmb{s}_i\pmb{s}_{\overline{j}}-\frac{2\pmb{s}}{\pmb{r}}\phi'\overline{\xi_i}\xi_j+\frac{4}{\pmb{r}}\phi'\xi_i\overline{\xi_j},\label{g-ija}
	\end{eqnarray}
	where we denote $c_0:=\phi-2\pmb{s}\phi'$. Substituting  \eqref{s-i}  into \eqref{g-ija} we obtain
	\begin{eqnarray}
		g_{i\overline{j}}
		&=&c_0\delta_{ij}+\frac{1}{\pmb{r}}c_1\xi_i\overline{\xi_j}
		+\frac{1}{\pmb{r}^2}c_2\xi\xi'\overline{\xi_i\xi_j}+\frac{1}{\pmb{r}^2}c_2\overline{\xi\xi'}\xi_i\xi_j-\frac{\pmb{s}}{\pmb{r}}c_2\overline{\xi_i}\xi_j\label{g-ijb}
	\end{eqnarray}
	where
	\begin{eqnarray}
		c_1=4\left(\phi'+\pmb{s}\phi''\right),\quad c_2=-2\left(\phi'+2\pmb{s}\phi''\right).
	\end{eqnarray}
	The strongly pseudoconvexity of $f_{IV}$ is equivalent to the positive definite of the Hermitian matrix $(g_{i\bar{j}})$ for any
	$\xi\in\mathbb{C}^N\setminus\{0\}$, while the later is equivalent to the fact that $(g_{i\bar{j}})$ has $n$ positive eigenvalues for any $\xi\in\mathbb{C}^N\setminus\{0\}$.
	
	For this purpose, we denote $H=(g_{i\bar{j}})$ and set
	\begin{equation}
		B=\begin{bmatrix}
			\xi_1 & \overline{\xi_1} \\
			\vdots & \vdots \\
			\xi_N & \overline{\xi_N} \\
		\end{bmatrix}
		,\quad
		X=\begin{bmatrix}
			\frac{1}{\pmb{r}}c_1&\frac{1}{\pmb{r}^2}c_2\overline{\xi\xi'}   \\
			\frac{1}{\pmb{r}^2}c_2\xi\xi'  & -\frac{\pmb{s}}{\pmb{r}}c_2\\
		\end{bmatrix}.\label{B}
	\end{equation}
	Note that $B^\ast$ is just the Hermitian transpose of $B$, namely $B^\ast=\overline{B'}$.
	Thus
	$$g_{i\bar{j}}=c_0\delta_{i\bar{j}}+\begin{bmatrix}
		\xi_i & \overline{\xi_i} \\
	\end{bmatrix}X\begin{bmatrix}
		\overline{\xi_j} \\
		\xi_j \\
	\end{bmatrix}
	$$
	or equivalently
	\begin{equation}
		H=c_0I_N+BXB^\ast.\label{H}
	\end{equation}
	
	By \eqref{H}, we have
	\begin{eqnarray}
		\det\{\lambda I_N-H\}&=&\det\{(\lambda-c_0)I_N-BXB^\ast\}\nonumber\\
		&=&(\lambda-c_0)^{N-2}\det\{(\lambda-c_0)I_2-B^\ast BX\},\label{BX}
	\end{eqnarray}
	where in the last equality we used the Lemma 4.1 in \cite{XZ1}. By \eqref{B}, we have
	$$
	B^\ast B=\begin{pmatrix}
		\pmb{r} &\overline{\xi\xi'} \\
		\xi\xi' &\pmb{r}\\
	\end{pmatrix},
	$$
	so that
	\begin{equation}
		B^\ast BX
		=\begin{pmatrix}
			c_1+\pmb{s}c_2&\frac{1}{\pmb{r}}(1-\pmb{s})c_2\overline{\xi\xi'}\\
			\frac{1}{\pmb{r}}(c_1+c_2)\xi\xi'&0\\
		\end{pmatrix}.
		\label{BBX}
	\end{equation}
	Substituting \eqref{BBX} into \eqref{BX}, we have
	\begin{eqnarray*}
		\det\{\lambda I_N-H\}&=&(\lambda-c_0)^{N-2}\det
		\begin{bmatrix}
			(\lambda-c_0)-(c_1+\pmb{s}c_2)&-\frac{1}{\pmb{r}}(1-\pmb{s})c_2\overline{\xi\xi'}\\
			-\frac{1}{\pmb{r}}(c_1+c_2)\xi\xi'&(\lambda-c_0)
		\end{bmatrix}.
	\end{eqnarray*}
	By a direct calculation and rearrangement of terms, we obtain
	\begin{eqnarray}
		c_0=\phi-2\pmb{s}\phi',\quad c_1=4\left(\phi'+\pmb{s}\phi''\right),\quad c_2=-2\left(\phi'+2\pmb{s}\phi''\right).
	\end{eqnarray}
	\begin{eqnarray}
		\det\{\lambda I_N-H\}
		=(\lambda-c_0)^{N-2}\Bigg\{
		\lambda^2-2\Big[\phi+(2-3\pmb{s})\phi'+2\pmb{s}(1-\pmb{s})\phi''\Big]\lambda+\tilde{k}\Bigg\},\label{ae}
	\end{eqnarray}
	where
	\begin{equation}
		\tilde{k}=\phi[\phi+2(2-3\pmb{s})\phi']+4\pmb{s}(1-\pmb{s})[\phi \phi''-(\phi')^2].\label{tk}
	\end{equation}
	
	Note that the following equation
	\begin{equation}
		\lambda^2-2\Big[\phi+(2-3\pmb{s})\phi'+2\pmb{s}(1-\pmb{s})\phi''\Big]\lambda+\tilde{k}=0\label{lambda-2}
	\end{equation}
	always has two real roots (counting multiplicities), since
	\begin{eqnarray*}
		\triangle&=&4\Big[\phi+(2-3\pmb{s})\phi'+2\pmb{s}(1-\pmb{s})\phi''\Big]^2-4\tilde{k}\\
		&=&4\Big\{[(2-3\pmb{s})\phi'+2\pmb{s}(1-\pmb{s})\phi'']^2+4\pmb{s}(1-\pmb{s})(\phi')^2\Big\}\geq 0.
	\end{eqnarray*}
	Denote $\lambda_{N-1}$ and $\lambda_N$ the real roots of \eqref{lambda-2}. If $\lambda_{N-1}$ and $\lambda_N$ are positive roots, it is necessary that
	$$\phi+(2-3\pmb{s})\phi'+2\pmb{s}(1-\pmb{s})\phi''>0\quad\mbox{and}\quad\tilde{k}>0.$$
	
	Thus $f_{IV}$ is a complex Minkowski norm if and only if
	$$
	\lambda_1=\cdots=\lambda_{N-2}=c_0>0,\quad\lambda_{N-1}>0,\quad\lambda_N>0.
	$$
	Notice that we have
	$$
	2\phi[\phi+(2-3\pmb{s})\phi'+2\pmb{s}(1-\pmb{s})\phi'']=\tilde{k}+\phi^2+4\pmb{s}(1-\pmb{s})(\phi')^2,
	$$
	which implies that if $\tilde{k}>0$, then it necessary
	$$
	\phi+(2-3\pmb{s})\phi'+2\pmb{s}(1-\pmb{s})\phi''>0.
	$$
	This completes the proof.
\end{proof}

\begin{example}
	Let $f_{IV}(\xi)=\sqrt{\pmb{r}\phi(\pmb{s})}$ with
	$$\phi(\pmb{s})=1+(1+\pmb{s})^{\frac{1}{2}}.$$
	Then $f_{IV}(\xi)$ is a complex Minkowski norm.
\end{example}

Now let $z_0\in\mathfrak{R}_{IV}$ be any fixed point, define
\begin{equation}
	F_{IV}(z_0;v):=f_{IV}(\sqrt{2N}(\phi_{z_0})_\ast(v)),\quad \forall v\in T_{z_0}^{1,0}\mathfrak{R}_{IV},\label{F-IV}
\end{equation}
where $(\phi_{z_0})_\ast$ denotes the differential of $\phi_{z_0}$ at the point $z_0$.

\begin{example} Let $z_0$ be any fixed point in $\mathfrak{R}_{IV}$ and $\phi_{z_0}\in \mbox{Aut}(\mathfrak{R}_{IV})$ is given by \eqref{p-IV} satisfying
	$\phi_{z_0}(z_0)=0$. If we take $\widetilde{\pmb{r}}=2N\cdot\xi\overline{\xi'}$ with $\xi:=(\phi_{z_0})_\ast(v)$ for any
	$v\in T_{z_0}^{1,0}\mathfrak{R}_{IV}$. Then it is easy to check that
	\begin{equation}
		\widetilde{\pmb{r}}(z_0;v)=\frac{2N}{\Delta^2(z_0)}v\left[\triangle(z_0)I_N-2(\overline{z_0z_0'})z_0'z_0-2(1-2z_0\overline{z_0'})z_0'\overline{z_0}
		+2\overline{z_0'}z_0-2(z_0z_0')\overline{z_0'z_0}
		\right]\overline{v'},\label{tr}
	\end{equation}
	is just the Bergman metric on $\mathfrak{R}_{IV}$, here we denote $\triangle(z_0):=1+|z_0z_0'|^2-2z_0\overline{z_0'}$.
\end{example}

Indeed, the Jacobian matrix of $w=\phi_{z_0}(z)$ at $z=z_0$ is given by
$$
\left(\frac{\partial w_j}{\partial z_i}\right)_{z=z_0}
=\frac{1}{\sqrt{\triangle(z_0)}}\Big\{I_N-\frac{2}{1-|z_0z_0'|^2}z_0'\overline{z_0}(I_N-\overline{z_0'}z_0)\Big\}D',
$$
where $D$ satisfies \eqref{p-IV-b}. Thus for any $v\in T_{z_0}^{1,0}\mathfrak{R}_{IV}$,
we have
$$\xi=(\phi_{z_0})_\ast(v)=\frac{1}{\sqrt{\triangle(z_0)}}v\Big\{I_N-\frac{2}{1-|z_0z_0'|^2}z_0'\overline{z_0}(I_N-\overline{z_0'}z_0)\Big\}D'\in T_0^{1,0}\mathfrak{R}_{IV}.$$
Denote
\begin{equation}
	\mathfrak{U}=I_N-\frac{2}{1-|z_0z_0'|^2}(z_0'\overline{z_0}-\overline{z_0z_0'}z_0'z_0),\label{U-a}
\end{equation}
Using \eqref{p-IV-b}, it is easy to check that
\begin{equation}
	D'D=I_N+\frac{2}{\Delta(z_0)}(z_0'\overline{z_0}+\overline{z_0'}z_0).\label{U-b}
\end{equation}
Substituting \eqref{U-a} and \eqref{U-b} into $\widetilde{\pmb{r}}:=2N\cdot \xi\overline{\xi'}=\frac{2N}{\Delta(z_0)}v\mathfrak{U}DD'\mathfrak{U}^\ast\overline{v'}$, one gets
\eqref{tr}.

Note that we have
\begin{equation}
	\widetilde{\pmb{s}}(z_0;v):=\frac{|\xi\xi'|^2}{(\xi\overline{\xi'})^2}=\frac{|v\mathfrak{U}D'D\mathfrak{U}'v'|^2}{\Delta^2(z_0)}\frac{1}{(\xi\overline{\xi'})^2}
	=\frac{4N^2}{\triangle^2(z_0)}\frac{|vv'|^2}{\widetilde{\pmb{r}}^2}.\label{ts}
\end{equation}
Since $z_0$ is any fixed point in $\mathfrak{R}_{IV}$, \eqref{tr} and \eqref{ts} hold for any $z\in \mathfrak{R}_{IV}$ and $v\in T_{z}^{1,0}\mathfrak{R}_{IV}$.

\begin{definition}Let $\phi:[0,1]\rightarrow (0,+\infty)$ be a smooth function satisfying \eqref{sn}. Define
	\begin{equation}
		F_{IV}(z;v)=\sqrt{\widetilde{\pmb{r}}\phi(\widetilde{\pmb{s}})},\quad \forall v\in T_z^{1,0}\mathfrak{R}_{IV},\label{FIV}
	\end{equation}
	where $\widetilde{\pmb{r}}$ and $\widetilde{\pmb{s}}$ are given by \eqref{tr} and \eqref{ts} respectively.
\end{definition}
\begin{remark}
	It is easy to check that any $\mbox{Aut}(\mathfrak{R}_{IV})$-invariant complex Finsler metric can be written in the form \eqref{FIV} for a suitable function $\phi(\widetilde{\pmb{s}})$. In particular the Bergman metric \eqref{tr} and the Kobayashi metric \eqref{s11}  on $\mathfrak{R}_{IV}$ are obtained by taking $\phi(\widetilde{\pmb{s}})\equiv 1$ and $\phi(\widetilde{\pmb{s}})=1+\sqrt{1-\widetilde{\pmb{s}}}$, respectively.
\end{remark}
\begin{theorem}
	$F_{IV}$ is a complete $\mbox{Aut}(\mathfrak{R}_{IV})$-invariant K\"ahler-Berwald metric.
\end{theorem}

\begin{proof}By assumptions, $F_{IV}$ is clear an $\mbox{Aut}(\mathfrak{R}_{IV})$-invariant strongly pseudoconvex complex Finsler metric on $\mathfrak{R}_{IV}$. Thus it suffices to show that it is a K\"ahler-Berwald metric.
	Since at the point $(0;v)\in T_0^{1,0}\mathfrak{R}_{IV}$,
	\begin{eqnarray}
		\frac{\partial \widetilde{\pmb{r}}}{\partial z_i}=0,\quad\frac{\partial^2\widetilde{\pmb{r}}}{\partial z_i\partial\overline{v_a}}=0,\quad
		\frac{\partial \widetilde{\pmb{s}}}{\partial z_i}=0,\quad\frac{\partial^2 \widetilde{\pmb{s}}}{\partial z_i\partial\overline{v_a}}=0,\label{rsz}
	\end{eqnarray}
	thus at $(0;v)$, we have
	\begin{equation}
		\frac{\partial^2F_{IV}^2}{\partial z_i\partial \overline{v_a}}=0.\label{V-IV}
	\end{equation}
	By Theorem \ref{th-2.1} and Remark \ref{RK1}, $F_{IV}$ is a complete K\"ahler-Berwald metric.
\end{proof}

\begin{theorem}\label{HS-IV}
	For any nonzero tangent vectors $v,w\in T_0^{1,0}\mathfrak{R}_{IV}$, the holomorphic sectional curvature $K_{IV}$ and the holomorphic bisectional
	curvature $B_{IV}$ of $F_{IV}$ are given respectively by
	\begin{eqnarray*}
		K_{IV}(0; v)&=&-\frac{2}{N\phi^2(\widetilde{\pmb{s}})}\left\{\phi(\widetilde{\pmb{s}})+(1-\widetilde{\pmb{s}})\left[\phi(\widetilde{\pmb{s}})-2\widetilde{\pmb{s}}\phi'(\widetilde{\pmb{s}})\right]\right\},\label{hsc-IV}\\
		B_{IV}(0; v, w)	&=&-\frac{2}{N}\cdot\frac{\phi\left[\widetilde{\pmb{r}}(0;v)\widetilde{\pmb{r}}(0;w)-|vw'|^2\right]+2\widetilde{\pmb{s}}\phi'|vw'|^2+(\phi-2\widetilde{\pmb{s}}\phi')|v\overline{w'}|^2}{F_{IV}^2(0;v)F_{IV}^2(0;w)}.\label{hbsc-IV}
	\end{eqnarray*}
	That is,  $K_{IV}$ is bounded between two negative constants $-A $ and $-B$, where $$A=\max\limits_{\widetilde{\pmb{s}}\in[0,1]}\left\{\frac{2\left\{\phi(\widetilde{\pmb{s}})+(1-\widetilde{\pmb{s}})\left[\phi(\widetilde{\pmb{s}})-2\widetilde{\pmb{s}}\phi'(\widetilde{\pmb{s}})\right]\right\}}{N\phi^2(\widetilde{\pmb{s}})}\right\}$$
	and
	$$ B=\min\limits_{\widetilde{\pmb{s}}\in[0,1]}\left\{\frac{2\left\{\phi(\widetilde{\pmb{s}})+(1-\widetilde{\pmb{s}})\left[\phi(\widetilde{\pmb{s}})-2\widetilde{\pmb{s}}\phi'(\widetilde{\pmb{s}})\right]\right\}}{N\phi^2(\widetilde{\pmb{s}})}\right\}.
	$$
	
	If furthermore $\phi'\geq0,$  then $B_{IV}$  is always nonpositive, and is bounded from below by a negative constant $-C$, with $$C=\max\limits_{\widetilde{\pmb{s}}\in[0,1]}\left\{\frac{2}{N}\cdot\frac{\phi\left[\widetilde{\pmb{r}}(0;v)\widetilde{\pmb{r}}(0;w)-|vw'|^2\right]+2\widetilde{\pmb{s}}\phi'|vw'|^2+(\phi-2\widetilde{\pmb{s}}\phi')|v\overline{w'}|^2}{F_{IV}^2(0;v)F_{IV}^2(0;w)}\right\}.$$
\end{theorem}
\begin{proof}Denote $G:=F_{IV}^2(z;v)=\widetilde{\pmb{r}}\phi(\widetilde{\pmb{s}})$. Then
	\begin{eqnarray}
		G_{; i\bar{j}}&=&\widetilde{\pmb{r}}_{; i\bar{j}}\phi+\widetilde{\pmb{r}}_{; i}\phi'\widetilde{\pmb{s}}_{; \bar{j}}+\widetilde{\pmb{r}}_{;\bar{j}}\phi'\widetilde{\pmb{s}}_{;i}
		+\widetilde{\pmb{r}}\phi''\widetilde{\pmb{s}}_{;i}\widetilde{\pmb{s}}_{;\bar{j}}+\widetilde{\pmb{r}}\phi'\widetilde{\pmb{s}}_{;i\bar{j}},
	\end{eqnarray}
	where $\phi', \phi''$ denote the derivatives of $\phi$ with respect to $\widetilde{\pmb{s}}$.
	Using \eqref{rsz},
	we have
	\begin{eqnarray}
		G_{;i\bar{j}}(0;v)&=&\left(\widetilde{\pmb{r}}_{; i\bar{j}}\phi+\widetilde{\pmb{r}}\phi'\widetilde{\pmb{s}}_{;i\bar{j}}\right)(0;v).\label{K;3}
	\end{eqnarray}
	Since at the point $(0;v)$, we have
	\begin{eqnarray}
		\widetilde{\pmb{r}}_{;i\bar{j}}(0;v)&=&4N\left(v\overline{v'}\delta_{ij}-v_i\overline{v_j}+v_j\overline{v_i}\right),\label{r;ij}\\
		\widetilde{\pmb{s}}_{;i\bar{j}}(0;v)&=&
		=\frac{4|vv'|^2}{(v\overline{v'})^3}(v_i\overline{v_j}-v_j\overline{v_i})
		=\frac{8N\widetilde{\pmb{s}}}{\widetilde{\pmb{r}}}(v_i\overline{v_j}-v_j\overline{v_i}).\label{s;ij}	
	\end{eqnarray}
	Substituting \eqref{r;ij} and \eqref{s;ij} into \eqref{K;3}, we have
	\begin{eqnarray}
		G_{;i\bar{j}}(0;v)&=&2\widetilde{\pmb{r}}\phi \delta_{ij}+4N(\phi-2\widetilde{\pmb{s}}\phi')(v_j\overline{v_i}-v_i\overline{v_j}).
		\label{K;5}
	\end{eqnarray}
	Thus by Proposition \ref{hsc}, for any nonzero tangent vector $v\in T_0^{1,0}\mathfrak{R}_{IV}$, we have
	\begin{eqnarray*}
		K_{IV}(0; v)
		&=&-\frac{2}{F_{IV}^4(0;v)}\sum_{i,j=1}^N\frac{\partial^2F_{IV}^2}{\partial z_i\partial \overline{z_j}}(0;v)v_i\overline{v_j}\\
		&=&-\frac{4}{\widetilde{\pmb{r}}^2\phi^2}\left\{\frac{1}{2N}\widetilde{\pmb{r}}^2\phi+2N(\phi-2\widetilde{\pmb{s}}\phi')\left[(v\overline{v'})^2-|vv'|^2\right]\right\}\\
		&=&-\frac{2}{N\phi^2(\widetilde{\pmb{s}})}\left\{\phi(\widetilde{\pmb{s}})+(1-\widetilde{\pmb{s}})\left[\phi(\widetilde{\pmb{s}})-2\widetilde{\pmb{s}}\phi'(\widetilde{\pmb{s}})\right]\right\}.\label{K;6}
	\end{eqnarray*}
	Notice that $\widetilde{\pmb{s}}\in[0,1]$ and  $\phi-2\pmb{s}\phi'>0$, we have
	$$\phi\leq\left\{\phi(\widetilde{\pmb{s}})+(1-\widetilde{\pmb{s}})\left[\phi(\widetilde{\pmb{s}})-2\widetilde{\pmb{s}}\phi'(\widetilde{\pmb{s}})\right]\right\}
	\leq2(\phi-s\phi').$$
	Thus		
	$$ -A\leq K_{IV}(0; v)\leq -B. $$
	That is, $K_{IV}$ is bounded between two negative constants $-A $ and $-B$, where $$A=\max\limits_{\widetilde{\pmb{s}}\in[0,1]}\left\{\frac{2\left\{\phi(\widetilde{\pmb{s}})+(1-\widetilde{\pmb{s}})\left[\phi(\widetilde{\pmb{s}})-2\widetilde{\pmb{s}}\phi'(\widetilde{\pmb{s}})\right]\right\}}{N\phi^2(\widetilde{\pmb{s}})}\right\}$$
	and
	$$ B=\min\limits_{\widetilde{\pmb{s}}\in[0,1]}\left\{\frac{2\left\{\phi(\widetilde{\pmb{s}})+(1-\widetilde{\pmb{s}})\left[\phi(\widetilde{\pmb{s}})-2\widetilde{\pmb{s}}\phi'(\widetilde{\pmb{s}})\right]\right\}}{N\phi^2(\widetilde{\pmb{s}})}\right\}.
	$$
	
	Since  $K_{IV}$ is $\mbox{Aut}(\mathfrak{R}_{IV})$-invariant,	it follows that for any nonzero vector $v\in T_z^{1,0}\mathfrak{R}_{IV}$,
	we have \eqref{hsc-IV}.
	
	By Proposition \ref{bhsc} and \eqref{K;5}, for any two  nonzero tangent vectors $v,w\in T_0^{1,0}\mathfrak{R}_{IV}$, we have
	\begin{eqnarray*}
		B_{IV}(0;v,w)
		&=&-\frac{2}{F_{IV}^2(0;v)F_{IV}^2(0;w)}\sum_{i,j=1}^N\frac{\partial^2 F_{IV}^2}{\partial z_i\partial\overline{z_j}}(0;v)w_i\overline{w_j}\\
		&=&-\frac{2}{N}\cdot\frac{\widetilde{\pmb{r}}(0;v)\widetilde{\pmb{r}}(0;w)\phi+(\phi-2\widetilde{\pmb{s}}\phi')\left(|v\overline{w'}|^2-|vw'|^2\right)}
		{F_{IV}^2(0;v)F_{IV}^2(0;w)}\\
		&=&-\frac{2}{N}\cdot\frac{\phi\left[\widetilde{\pmb{r}}(0;v)\widetilde{\pmb{r}}(0;w)-|vw'|^2\right]+2\widetilde{\pmb{s}}\phi'|vw'|^2+(\phi-2\widetilde{\pmb{s}}\phi')|v\overline{w'}|^2}{F_{IV}^2(0;v)F_{IV}^2(0;w)}.
	\end{eqnarray*}
	Since
	$$\widetilde{\pmb{r}}(0;v)\widetilde{\pmb{r}}(0;w)-|vw'|^2=(v\overline{v'})(w\overline{w'})-|vw'|^2 \geq0,\quad \phi-2s\phi'>0,$$
	thus if furthermore $\phi'\geq0$, 	we have
	$$
	-\frac{2}{N}\cdot\frac{1}{\phi(\widetilde{\pmb{s}}(0;w))}
	\leq B_{IV}(0;v,w)
	\leq
	-\frac{4}{N}\cdot\frac{\widetilde{\pmb{s}}\phi'|vw'|^2}{F_{IV}^2(0;v)F_{IV}^2(0;w)}.
	$$		
	That is, $B_{IV}$ is always nonpositive, and is bounded from below by a negative constant $-C$ with $$C=\max\limits_{\widetilde{\pmb{s}}\in[0,1]}\left\{\frac{2}{N}\cdot\frac{\phi\left[\widetilde{\pmb{r}}(0;v)\widetilde{\pmb{r}}(0;w)-|vw'|^2\right]+2\widetilde{\pmb{s}}\phi'|vw'|^2+(\phi-2\widetilde{\pmb{s}}\phi')|v\overline{w'}|^2}{F_{IV}^2(0;v)F_{IV}^2(0;w)}\right\}.$$
\end{proof}

\begin{remark}
	(1) The constants $A,B,C$ in Theorem \ref{HS-IV} can be attained since $K_{IV}(0;v)$  is actually a smooth function for any
	$(0;[v])\in PT_0^{1,0}\mathfrak{R}_{IV}\cong\mathbb{CP}^{N-1}$, and $B_{IV}(0;v, w)$ is also a smooth function for any $(0;[v])$ and $(0;[w])$ in $PT_0^{1,0}\mathfrak{R}_{IV}\cong\mathbb{CP}^{N-1}$, here  $PT_0^{1,0}\mathfrak{R}_{IV}$ denotes the projective fiber of $T_0^{1,0}\mathfrak{R}_{IV}\cong \mathbb{C}^N$. Since $\mathfrak{R}_{IV}$ is homogeneous, the constants $A,B$ are the same for $K_{IV}(z;v)$ for any $(z;v)\in \widetilde{T^{1,0}\mathfrak{R}_{IV}}$, and the constant $C$ is the same for $B_{IV}(z;v,w)$ for any $(z;v),(z;w)\in \widetilde{T^{1,0}\mathfrak{R}_{IV}}$.
	
	(2) In particular, taking $\phi(\pmb{s})\equiv 1$, we get the holomorphic sectional curvature and the holomorphic bisectional curvature of the Bergman metric on $\mathfrak{R}_{IV}$ as follows:
	$$
	K_{IV}(0; v)=-\frac{2}{N}\left\{2-\frac{|vv'|^2}{(v\overline{v'})^2}\right\}, \quad\forall 0\neq v,w\in T_0^{1,0}\mathfrak{R}_{IV},
	$$		
	$$
	B_{IV}(0;v,w)=-\frac{2}{N}\cdot \frac{(v\overline{v'})(w\overline{w'})-|vw'|^2+|v\overline{w'}|^2}
	{(v\overline{v'})(w\overline{w'})}, \quad\forall 0\neq v,w\in T_0^{1,0}\mathfrak{R}_{IV}.
	$$
	
	(3)  By Propsition \ref{prop} and Remark \ref{remark}, for $\phi(\pmb{s})=e^{1+t\sqrt[k]{1+\pmb{s}}}$ with any fixed $k\geq2$ and any fixed $t\in(0,\frac{1}{4k})$, the corresponding  $F_{IV}$ is non-Hermitian quadratic holomorphic invariant complete K\"ahler-Berwald metric with $\phi'>0.$
\end{remark}

We also obtain the corresponding Schwarz lemma for holomorphic mappings $f$ from $\mathfrak{R}_A$ into itself whenever $\mathfrak{R}_A$ is endowed
with the $\mbox{Aut}(\mathfrak{R}_A)$-invariant K\"ahler-Berwald metric $F_A$ for $A=I,II,III, IV$, respectively. This will appear elsewhere.
\vskip0.4cm
\textbf{Acknowledgements}\quad  This work was supported by the National Natural Science Foundation of China (Grant Nos. 12071386, 11671330). The authors are also very grateful to the anonymous referees for careful reading of the manuscript and very valuable suggestions.

\end{document}